\documentclass[11pt,reqno]{amsart}
\usepackage{amssymb,amsmath,amsfonts,amsthm,enumerate,stmaryrd}
\usepackage[left=.75 in, right=.75 in,top=.75 in, bottom=.75 in]{geometry}
\usepackage{tikz}
\usetikzlibrary{backgrounds}
\usetikzlibrary{
    decorations,
    decorations.pathmorphing,
}

\providecommand{\abs}[1]{\left\vert#1\right\vert}
\providecommand{\parens}[1]{\left( #1\right)}

\providecommand{\norm}[1]{\left\Vert#1\right\Vert}

\def\nab{\nabla}
\def\dt{\partial_t}

\def\ls{\lesssim}
\def\pt{\partial}

\def\ddt{\frac{d}{dt}}
\def\coloneqq{:=}

\def\calA{\mathcal{A}}
\def\calD{\mathcal{D}}
\def\calE{\mathcal{E}}
\def\calF{\mathcal{F}}
\def\calG{\mathcal{G}}
\def\calH{\mathcal{H}}

\def\calK{\mathcal{K}}
\def\calM{\mathcal{M}}
\def\calN{\mathcal{N}}
\def\calP{\mathcal{P}}

\def\calS{\mathcal{S}}

\def\calW{\mathcal{W}}
\def\calY{\mathcal{Y}}

\def\al{\alpha}
\def\ov{\overline}
\def\lb{\lambda}

\newenvironment{aligneq}{
    \begin{equation}
    \begin{aligned}
} {
    \end{aligned}
    \end{equation}
}

 \def\N{\mathbb{N}}

\DeclareMathOperator{\diverge}{div}

\title[Faraday stability]{The nonlinear stability regime of the viscous Faraday wave problem}

\author{David Altizio}
\address{
Department of Mathematical Sciences\\
Carnegie Mellon University\\
Pittsburgh, PA 15213, USA
}
\email[D. Altizio]{daltizio@andrew.cmu.edu}

\author{Ian Tice}
\address{
Department of Mathematical Sciences\\
Carnegie Mellon University\\
Pittsburgh, PA 15213, USA
}
\email[I. Tice]{iantice@andrew.cmu.edu}

\author{Xinyu Wu}
\address{
Department of Mathematical Sciences\\
Carnegie Mellon University\\
Pittsburgh, PA 15213, USA
}
\email[X. Wu]{xinyuw1@andrew.cmu.edu}

\author{Taisuke Yasuda}
\address{
Department of Mathematical Sciences\\
Carnegie Mellon University\\
Pittsburgh, PA 15213, USA
}
\email[T. Yasuda]{taisukey@andrew.cmu.edu}

\thanks{I. Tice was supported by an NSF CAREER Grant (DMS \#1653161).  D. Altizio, X. Wu, and T. Yasuda were supported by the summer research support provided by this grant.}



\newtheorem{lem}{Lemma}[section]

\newtheorem{prop}[lem]{Proposition}
\newtheorem{thm}[lem]{Theorem}

\theoremstyle{definition}

\numberwithin{equation}{section} 

\begin{document}

\begin{abstract}
This paper concerns the dynamics of a layer of incompressible viscous fluid lying above a vertically oscillating rigid plane and with an upper boundary given by a free surface.  We consider the problem with gravity and surface tension for horizontally periodic flows. This problem gives rise to flat but vertically oscillating equilibrium solutions, and the main thrust of this paper is to study the asymptotic stability of these equilibria in certain parameter regimes. We prove that both with and without surface tension there exists a parameter regime in which sufficiently small perturbations of the equilibrium at time $t = 0$ give rise to global-in-time solutions that decay to equilibrium at an identified quantitative rate.
\end{abstract}

\maketitle

\section{Introduction}

\subsection{Faraday waves}\label{sec:faraday}

Consider a flat rigid surface in three dimensions, and suppose that a finite layer of incompressible fluid is deposited on the surface and held there by a uniform gravitational field.  The upper surface of the fluid is free.  Suppose that the rigid lower surface is then oscillated in the vertical direction as indicated in Figure \ref{fig:faraday-wave-problem}. 

\begin{figure}[h]
\centering

\begin{tikzpicture}[decoration={
    coil,
    amplitude = 1.5mm,
    aspect=1,
    pre length=1mm,
    post length=1mm,
}]

\pgfmathsetmacro{\cubex}{10}
\pgfmathsetmacro{\cubey}{1}
\pgfmathsetmacro{\cubez}{6}
\pgfmathsetmacro{\trayx}{10}
\pgfmathsetmacro{\trayy}{0.3}
\pgfmathsetmacro{\trayz}{6}

\draw[decorate] (-5,-\cubey-\trayy,-3) -- (-5,-\cubey-\trayy-3,-3);

\draw[fill=gray,opacity=0.85] (0,-\cubey,0) -- ++(-\trayx,0,0) -- ++(0,-\trayy,0) -- ++(\trayx,0,0) -- cycle;
\draw[fill=gray,opacity=0.85] (0,-\cubey,0) -- ++(0,0,-\trayz) -- ++(0,-\trayy,0) -- ++(0,0,\trayz) -- cycle;
\draw[fill=gray,opacity=0.85] (0,-\cubey,0) -- ++(-\trayx,0,0) -- ++(0,0,-\trayz) -- ++(\trayx,0,0) -- cycle;

\draw[fill=blue,opacity=0.3] (0,0,0) -- ++(-\cubex,0,0) -- ++(0,-\cubey,0) -- ++(\cubex,0,0) -- cycle;
\draw[fill=blue,opacity=0.3] (0,0,0) -- ++(0,0,-\cubez) -- ++(0,-\cubey,0) -- ++(0,0,\cubez) -- cycle;
\draw[fill=blue,opacity=0.3] (0,0,0) -- ++(-\cubex,0,0) -- ++(0,0,-\cubez) -- ++(\cubex,0,0) -- cycle;

\draw[<->, ultra thick] (-\cubex-1,-\cubey-\trayy,0) -- (-\cubex-1,-\trayy,0);

\end{tikzpicture}

\caption{A layer of fluid evolves on a vertically oscillating rigid surface.}
\label{fig:faraday-wave-problem}
\end{figure}
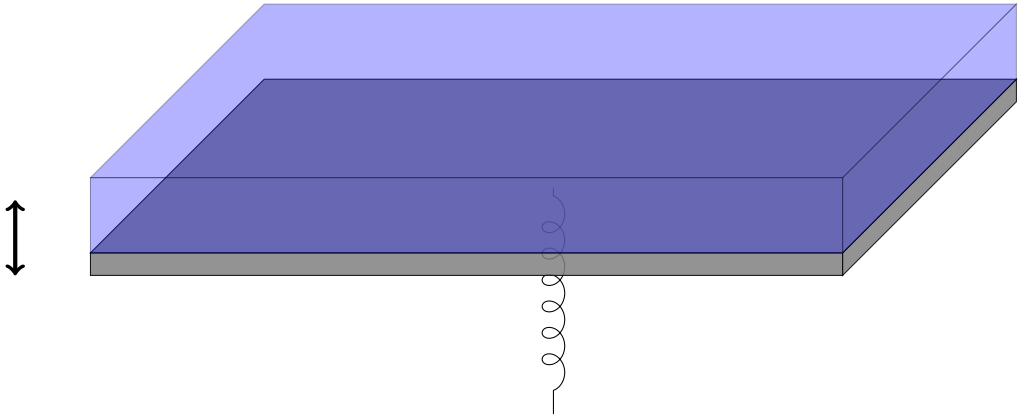

It was observed by Faraday \cite{faraday1831forms}  in the nineteenth century that in certain regions of the frequency-amplitude parameter space the free surface of the fluid forms standing waves, and in the complement of this region the free surface remains flat.  This phenomenon is now given the moniker \emph{Faraday waves}, and the various fascinating patterns formed by these standing waves have been studied intensively, both experimentally and theoretically.  As such, we will only attempt a very brief survey of the mathematical literature related to Faraday waves.  For a more thorough survey of the literature, especially for the case of inviscid fluids, we refer to the review by Miles-Henderson \cite{miles_henderson_1990}.
 
From a mathematical perspective, the linearized problem has been analyzed in the inviscid case by Benjamin-Ursell  \cite{benjamin1954stability} and in the viscous case by Kumar \cite{kumar1996linear} and Kumar-Tuckerman \cite{kumar1994parametric} to determine conditions for the onset of these surface waves, or more precisely to characterize the stability or instability of the flat interface.  In the inviscid case it is known \cite{benjamin1954stability} that the instability mechanism is equivalent to the parametric instability mechanism of the Mathieu ODE, about which much is known (see, for instance, McLachlan's book \cite{mclachlan}).  The viscous problem is more complicated and does not reduce to the Mathieu ODE, but the numerical approximations of \cite{kumar1996linear} show that instability regions persist and are qualitatively similar to those in the inviscid case.  In  the work of Skeldon-Rucklidge \cite{skeldon2015can}  and Westra-Binks-VanDeWater \cite{westra2003patterns} the tools of weakly nonlinear analysis were employed to explain the various surface wave patterns observed in experiments.  Simulations and numerical studies also have achieved results that agree well with experiments in various settings: see for instance P\'{e}rinet-Juric-Tuckerman \cite{perinet2009numerical} and Qadeer \cite{qadeer2018simulating}.  Faraday waves have also been studied with linear and numerical analysis in compressible fluids by Das-Morris-Bhattacharyay \cite{dasetal_2009}.

Faraday waves have recently experienced a renewed interest since the experimental discovery of  Couder-Proti\`{e}re-Fort-Boudaoud \cite{couder2005dynamical}, which showed that Faraday waves coupled with fluid droplets can ``walk.''  These walking water droplets can further be coupled with other water droplets, and have been shown to exhibit behavior analogous to quantum mechanical phenomena. We refer the reader to the review of Bush \cite{bush2015pilot} for more details of this line of work. 
 
To the best of our knowledge, there has been no fully nonlinear analysis of the viscous Faraday wave problem.   In particular, there are no results rigorously establishing the existence of a stable parameter regime.  The principal goal in this paper is to prove such a result and to provide some quantitative estimates for where in the oscillation parameter space Faraday waves \emph{do not} occur, i.e. where the flat free interface remains stable.

\subsection{Free boundary Navier-Stokes equations in an oscillating domain}

We now properly formulate the problem to be studied in the paper.

\subsubsection{Overview of assumptions}
We consider a layer of viscous incompressible  fluid evolving above a flat plane in three dimensions. We assume the fluid is subjected to a uniform gravitational force field of the form $-ge_3\in\mathbb R^3$ where $g>0$ is a constant and $e_3 = (0,0,1)$ is the vertical unit vector. Furthermore, we work in a situation where the layer of fluid lies on top of a lower boundary that moves in the vertical direction, so that the vertical component at time $t$ is given by $Af(\omega t)-b$ where $f:\mathbb{T} = \mathbb{R} / \mathbb{Z} \to[-1,1]$ is a smooth, non-constant oscillation profile, $A>0$ is an amplitude parameter, $\omega>0$ is a frequency parameter, and $b>0$ is a constant depth parameter.  A typical choice of the oscillation profile is $f(t) = \cos(2\pi t - \delta)$ for some $\delta \in [0,2\pi)$.  We allow for the more general profile $f$ in order to highlight that it is the amplitude and frequency of the oscillation profile that play the dominant role in determining stability.  Note in particular that since $f$ is a smooth function on the torus $\mathbb{T}$, the assumption that it is not a constant implies that none of its derivatives may vanish identically. 

In addition to the above assumption on the external force acting on the fluid, we will assume three other main features. First, we assume that the fluid is bounded above by a free surface that evolves with the fluid. Second, we assume that above the free interface the fluid is bordered by a trivial fluid of constant pressure (for instance, a vacuum). Third, we assume that the fluid is horizontally periodic so that we can determine its dynamics by studying a single horizontal periodicity cell.

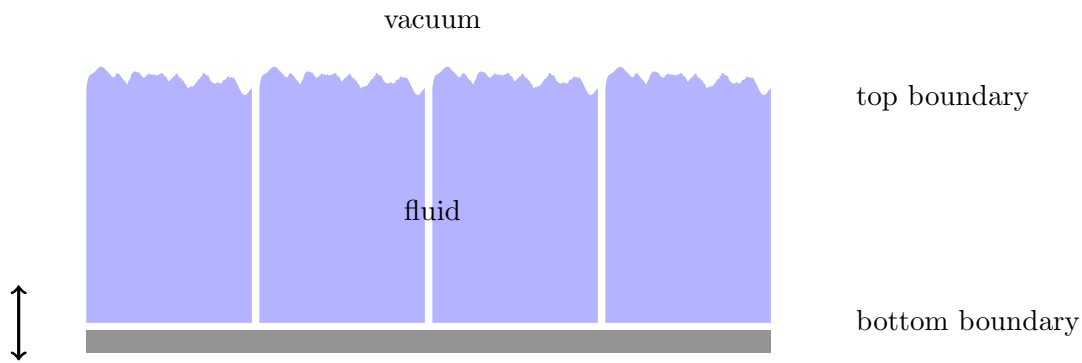
\begin{figure}[ht]
    \centering
\begin{tikzpicture}
    \fill[blue,opacity=0.3](0.9,-3) -- (0.9, 0.12291074114698593) -- (0.91, 0.1950766528886293) -- (0.92, 0.24200926113724006) -- (0.93, 0.2712938679570716) -- (0.9400000000000001, 0.2888723911039713) -- (0.9500000000000001, 0.2992824262429433) -- (0.96, 0.30587502902418484) -- (0.97, 0.3110121066369235) -- (0.98, 0.31624430846038326) -- (0.99, 0.3224703054312061) -- (1.0, 0.3300783477466577) -- (1.01, 0.33907099052294526) -- (1.02, 0.3491738770279719) -- (1.03, 0.3599294691078594) -- (1.04, 0.370776614426565) -- (1.05, 0.38111684013791797) -- (1.06, 0.39036826260940677) -- (1.07, 0.3980080028170425) -- (1.08, 0.4036039970306254) -- (1.09, 0.40683709240874616) -- (1.1, 0.4075143171228428) -- (1.11, 0.40557421462965126) -- (1.12, 0.4010851317113529) -- (1.1300000000000001, 0.39423734990278975) -- (1.1400000000000001, 0.3853299499250271) -- (1.15, 0.37475330343260255) -- (1.1600000000000001, 0.3563436163344125) -- (1.17, 0.3546293581133293) -- (1.1800000000000002, 0.3434681248208313) -- (1.19, 0.3327381589801685) -- (1.2, 0.32249234960120016) -- (1.21, 0.3105946161151378) -- (1.22, 0.3039600186535996) -- (1.23, 0.28688840435559215) -- (1.24, 0.27631910871380294) -- (1.25, 0.2747579808655225) -- (1.26, 0.26610266617760764) -- (1.27, 0.2683769995776154) -- (1.28, 0.2798991507270649) -- (1.29, 0.2972375574794057) -- (1.3, 0.32102948738582615) -- (1.31, 0.32322284786522776) -- (1.32, 0.31511228334186053) -- (1.33, 0.31770309839586575) -- (1.34, 0.3043232265728078) -- (1.35, 0.2832286162768577) -- (1.36, 0.2858645872666763) -- (1.37, 0.2567921162831364) -- (1.38, 0.25278514529597557) -- (1.3900000000000001, 0.2297799813267726) -- (1.4, 0.22430275029217728) -- (1.4100000000000001, 0.22633700091393613) -- (1.42, 0.20203059129091314) -- (1.4300000000000002, 0.19204121569457774) -- (1.44, 0.1735901411541578) -- (1.4500000000000002, 0.17860360496082062) -- (1.46, 0.22025362930432146) -- (1.4700000000000002, 0.22647129893759455) -- (1.48, 0.2452841776959024) -- (1.49, 0.27831045314021513) -- (1.5, 0.3040172354034933) -- (1.51, 0.3218491799129439) -- (1.52, 0.3293557371379631) -- (1.53, 0.3468652114865258) -- (1.54, 0.34069034881962607) -- (1.55, 0.3487454260558381) -- (1.56, 0.33502678233029454) -- (1.57, 0.3367456213349664) -- (1.58, 0.33924651126943545) -- (1.59, 0.3259352151060257) -- (1.6, 0.3351175565011755) -- (1.6099999999999999, 0.32383773251777914) -- (1.62, 0.3025265471902197) -- (1.63, 0.29184398704110365) -- (1.6400000000000001, 0.27657253508962204) -- (1.65, 0.2708614764435344) -- (1.6600000000000001, 0.2562853073192842) -- (1.67, 0.2700283808179787) -- (1.6800000000000002, 0.2620687792273508) -- (1.69, 0.2884015360479295) -- (1.7000000000000002, 0.2949241784037812) -- (1.71, 0.29121957994539427) -- (1.7200000000000002, 0.3166675184892683) -- (1.73, 0.31889690773197693) -- (1.74, 0.312622530208679) -- (1.75, 0.3037430762082276) -- (1.76, 0.2992760577627454) -- (1.77, 0.3120572822123504) -- (1.78, 0.3015057161801582) -- (1.79, 0.2949393021572416) -- (1.8, 0.3032230335291212) -- (1.81, 0.309692421267794) -- (1.82, 0.29008219147393366) -- (1.83, 0.29982693859833404) -- (1.84, 0.28663426122514096) -- (1.85, 0.29754001868824165) -- (1.8599999999999999, 0.3035160622976074) -- (1.87, 0.30675757118512953) -- (1.88, 0.31329626865856225) -- (1.8900000000000001, 0.31000751490920925) -- (1.9, 0.32862872883316396) -- (1.9100000000000001, 0.32064796035167137) -- (1.92, 0.3158760370562475) -- (1.9300000000000002, 0.30076133063703425) -- (1.94, 0.2752668594243774) -- (1.9500000000000002, 0.2828525168901326) -- (1.96, 0.2785026421420431) -- (1.9700000000000002, 0.25431789256062) -- (1.98, 0.23898488057598666) -- (1.9900000000000002, 0.21523907035225673) -- (2.0, 0.2100289328185864) -- (2.0100000000000002, 0.22812729125449938) -- (2.02, 0.23468648292710514) -- (2.0300000000000002, 0.2566632765750441) -- (2.04, 0.24999062616202805) -- (2.0500000000000003, 0.26360141741461085) -- (2.06, 0.2842660179076294) -- (2.07, 0.2769942093883895) -- (2.08, 0.29923414302046863) -- (2.09, 0.31105152915158135) -- (2.1, 0.3180046787445816) -- (2.11, 0.
3051889334223304) -- (2.12, 0.2896068141897182) -- (2.13, 0.2682238733296975) -- (2.14, 0.2563200885198964) -- (2.15, 0.26102823269063496) -- (2.16, 0.25396322920696013) -- (2.17, 0.2372007247914635) -- (2.18, 0.21874468908207356) -- (2.19, 0.22825200657613456) -- (2.2, 0.20209326491871923) -- (2.21, 0.17647321657206416) -- (2.22, 0.18643508870458336) -- (2.23, 0.15462426880618899) -- (2.24, 0.13810319771610757) -- (2.25, 0.11479261815267255) -- (2.2600000000000002, 0.12742329130550709) -- (2.27, 0.1424672298799308) -- (2.2800000000000002, 0.1352717413318057) -- (2.29, 0.13988266122672208) -- (2.3000000000000003, 0.15703009061828763) -- (2.31, 0.14525602387440534) -- (2.32, 0.14145960026143486) -- (2.33, 0.16716773115098948) -- (2.34, 0.15990750178413923) -- (2.35, 0.18445224405025792) -- (2.36, 0.19484575016623093) -- (2.37, 0.20566380344128107) -- (2.38, 0.227972803192257) -- (2.39, 0.24292364770484107) -- (2.4, 0.2436230539007264) -- (2.41, 0.24352677715651802) -- (2.42, 0.278176844260742) -- (2.43, 0.2791946350075714) -- (2.44, 0.28962140625908567) -- (2.45, 0.326524384131525) -- (2.46, 0.3216386813882537) -- (2.47, 0.32162356121497143) -- (2.48, 0.30696567855193563) -- (2.49, 0.2943548687488767) -- (2.5, 0.29245006251798983) -- (2.5100000000000002, 0.2928572249741567) -- (2.52, 0.2823502677206839) -- (2.5300000000000002, 0.2765065394539235) -- (2.54, 0.2433930847266742) -- (2.5500000000000003, 0.22348278992515525) -- (2.56, 0.21717148276008622) -- (2.57, 0.2006338446884197) -- (2.58, 0.19431705809638378) -- (2.59, 0.19114576002993414) -- (2.6, 0.18282749679578145) -- (2.61, 0.1912632117611317) -- (2.62, 0.17764843665243524) -- (2.63, 0.19767578691172694) -- (2.64, 0.19894618017689) -- (2.65, 0.1994744094806289) -- (2.66, 0.19778440299585046) -- (2.67, 0.1806292676865254) -- (2.68, 0.1927146966084676) -- (2.69, 0.18762088941065658) -- (2.7, 0.18696303485149857) -- (2.71, 0.21062836277453864) -- (2.72, 0.2346039423037656) -- (2.73, 0.22958388000832955) -- (2.74, 0.23707853509843294) -- (2.75, 0.2388321681426053) -- (2.7600000000000002, 0.2499766560480241) -- (2.77, 0.2681608701011314) -- (2.7800000000000002, 0.27724452973449004) -- (2.79, 0.278593985580782) -- (2.8000000000000003, 0.2678030890680097) -- (2.81, 0.25829413934868034) -- (2.82, 0.26896304759769113) -- (2.83, 0.2673357265327721) -- (2.84, 0.2504390038513554) -- (2.85, 0.27014740026587564) -- (2.86, 0.2600495678446209) -- (2.87, 0.2468397638740145) -- (2.88, 0.2308394794742706) -- (2.89, 0.21248525975284394) -- (2.9, 0.1923110962518332) -- (2.91, 0.17092805331053249) -- (2.92, 0.14900168524461208) -- (2.93, 0.1272278058622552) -- (2.94, 0.10630717183669108) -- (2.9499999999999997, 0.08691964145512139) -- (2.96, 0.06969837026377242) -- (2.9699999999999998, 0.05520460512903382) -- (2.98, 0.04390363823431376) -- (2.9899999999999998, 0.036142482532374474) -- (3.0, 0.03212983017349893) -- (3.01, 0.03191885542859909) -- (3.02, 0.03539342362728182) -- (3.03, 0.04225826763132976) -- (3.04, 0.05203369336219451) -- (3.05, 0.0640553759036931) -- (3.06, 0.07747980769844459) -- (3.07, 0.09129596035911765) -- (3.08, 0.10434372161278753) -- (3.09, 0.11533966890004411) -- (3.1, 0.12291074114698593) -- (3.1,-3) -- cycle;
\fill[blue,opacity=0.3](3.2,-3) -- (3.2, 0.12291074114698593) -- (3.21, 0.1950766528886293) -- (3.22, 0.24200926113724006) -- (3.23, 0.2712938679570716) -- (3.24, 0.2888723911039713) -- (3.25, 0.2992824262429433) -- (3.2600000000000002, 0.30587502902418484) -- (3.27, 0.3110121066369235) -- (3.2800000000000002, 0.31624430846038326) -- (3.29, 0.3224703054312061) -- (3.3000000000000003, 0.3300783477466577) -- (3.31, 0.33907099052294526) -- (3.3200000000000003, 0.3491738770279719) -- (3.33, 0.3599294691078594) -- (3.3400000000000003, 0.370776614426565) -- (3.35, 0.38111684013791797) -- (3.3600000000000003, 0.39036826260940677) -- (3.37, 0.3980080028170425) -- (3.3800000000000003, 0.4036039970306254) -- (3.39, 0.40683709240874616) -- (3.4000000000000004, 0.4075143171228428) -- (3.41, 0.40557421462965126) -- (3.4200000000000004, 0.4010851317113529) -- (3.43, 0.39423734990278975) -- (3.4400000000000004, 0.3853299499250271) -- (3.45, 0.37475330343260255) -- (3.46, 0.3563436163344125) -- (3.47, 0.3546293581133293) -- (3.4800000000000004, 0.3434681248208313) -- (3.49, 0.3327381589801685) -- (3.5, 0.32249234960120016) -- (3.5100000000000002, 0.3105946161151378) -- (3.52, 0.3039600186535996) -- (3.5300000000000002, 0.28688840435559215) -- (3.54, 0.27631910871380294) -- (3.5500000000000003, 0.2747579808655225) -- (3.56, 0.26610266617760764) -- (3.5700000000000003, 0.2683769995776154) -- (3.58, 0.2798991507270649) -- (3.5900000000000003, 0.2972375574794057) -- (3.6, 0.32102948738582615) -- (3.6100000000000003, 0.32322284786522776) -- (3.62, 0.31511228334186053) -- (3.6300000000000003, 0.31770309839586575) -- (3.64, 0.3043232265728078) -- (3.6500000000000004, 0.2832286162768577) -- (3.66, 0.2858645872666763) -- (3.6700000000000004, 0.2567921162831364) -- (3.68, 0.25278514529597557) -- (3.6900000000000004, 0.2297799813267726) -- (3.7, 0.22430275029217728) -- (3.71, 0.22633700091393613) -- (3.72, 0.20203059129091314) -- (3.7300000000000004, 0.19204121569457774) -- (3.74, 0.1735901411541578) -- (3.75, 0.17860360496082062) -- (3.7600000000000002, 0.22025362930432146) -- (3.7700000000000005, 0.22647129893759455) -- (3.7800000000000002, 0.2452841776959024) -- (3.79, 0.27831045314021513) -- (3.8000000000000003, 0.3040172354034933) -- (3.81, 0.3218491799129439) -- (3.8200000000000003, 0.3293557371379631) -- (3.83, 0.3468652114865258) -- (3.8400000000000003, 0.34069034881962607) -- (3.85, 0.3487454260558381) -- (3.8600000000000003, 0.33502678233029454) -- (3.87, 0.3367456213349664) -- (3.8800000000000003, 0.33924651126943545) -- (3.89, 0.3259352151060257) -- (3.9000000000000004, 0.3351175565011755) -- (3.91, 0.32383773251777914) -- (3.92, 0.3025265471902197) -- (3.93, 0.29184398704110365) -- (3.9400000000000004, 0.27657253508962204) -- (3.95, 0.2708614764435344) -- (3.96, 0.2562853073192842) -- (3.97, 0.2700283808179787) -- (3.9800000000000004, 0.2620687792273508) -- (3.99, 0.2884015360479295) -- (4.0, 0.2949241784037812) -- (4.01, 0.29121957994539427) -- (4.0200000000000005, 0.3166675184892683) -- (4.03, 0.31889690773197693) -- (4.04, 0.312622530208679) -- (4.05, 0.3037430762082276) -- (4.0600000000000005, 0.2992760577627454) -- (4.07, 0.3120572822123504) -- (4.08, 0.3015057161801582) -- (4.09, 0.2949393021572416) -- (4.1000000000000005, 0.3032230335291212) -- (4.11, 0.309692421267794) -- (4.12, 0.29008219147393366) -- (4.13, 0.29982693859833404) -- (4.140000000000001, 0.28663426122514096) -- (4.15, 0.29754001868824165) -- (4.16, 0.3035160622976074) -- (4.17, 0.30675757118512953) -- (4.18, 0.31329626865856225) -- (4.19, 0.31000751490920925) -- (4.2, 0.32862872883316396) -- (4.21, 0.32064796035167137) -- (4.220000000000001, 0.3158760370562475) -- (4.23, 0.30076133063703425) -- (4.24, 0.2752668594243774) -- (4.25, 0.2828525168901326) -- (4.26, 0.2785026421420431) -- (4.2700000000000005, 0.25431789256062) -- (4.28, 0.23898488057598666) -- (4.29, 0.21523907035225673) -- (4.300000000000001, 0.2100289328185864) -- (4.3100000000000005, 0.22812729125449938) -- (4.32, 0.23468648292710514) -- (4.33, 0.2566632765750441) -- (4.34, 0.24999062616202805) -- 
(4.3500000000000005, 0.26360141741461085) -- (4.36, 0.2842660179076294) -- (4.37, 0.2769942093883895) -- (4.38, 0.29923414302046863) -- (4.390000000000001, 0.31105152915158135) -- (4.4, 0.3180046787445816) -- (4.41, 0.3051889334223304) -- (4.42, 0.2896068141897182) -- (4.43, 0.2682238733296975) -- (4.44, 0.2563200885198964) -- (4.45, 0.26102823269063496) -- (4.46, 0.25396322920696013) -- (4.470000000000001, 0.2372007247914635) -- (4.48, 0.21874468908207356) -- (4.49, 0.22825200657613456) -- (4.5, 0.20209326491871923) -- (4.51, 0.17647321657206416) -- (4.5200000000000005, 0.18643508870458336) -- (4.53, 0.15462426880618899) -- (4.54, 0.13810319771610757) -- (4.550000000000001, 0.11479261815267255) -- (4.5600000000000005, 0.12742329130550709) -- (4.57, 0.1424672298799308) -- (4.58, 0.1352717413318057) -- (4.59, 0.13988266122672208) -- (4.6000000000000005, 0.15703009061828763) -- (4.61, 0.14525602387440534) -- (4.62, 0.14145960026143486) -- (4.63, 0.16716773115098948) -- (4.640000000000001, 0.15990750178413923) -- (4.65, 0.18445224405025792) -- (4.66, 0.19484575016623093) -- (4.67, 0.20566380344128107) -- (4.68, 0.227972803192257) -- (4.69, 0.24292364770484107) -- (4.7, 0.2436230539007264) -- (4.71, 0.24352677715651802) -- (4.720000000000001, 0.278176844260742) -- (4.73, 0.2791946350075714) -- (4.74, 0.28962140625908567) -- (4.75, 0.326524384131525) -- (4.76, 0.3216386813882537) -- (4.7700000000000005, 0.32162356121497143) -- (4.78, 0.30696567855193563) -- (4.79, 0.2943548687488767) -- (4.800000000000001, 0.29245006251798983) -- (4.8100000000000005, 0.2928572249741567) -- (4.82, 0.2823502677206839) -- (4.83, 0.2765065394539235) -- (4.84, 0.2433930847266742) -- (4.8500000000000005, 0.22348278992515525) -- (4.86, 0.21717148276008622) -- (4.87, 0.2006338446884197) -- (4.88, 0.19431705809638378) -- (4.890000000000001, 0.19114576002993414) -- (4.9, 0.18282749679578145) -- (4.91, 0.1912632117611317) -- (4.92, 0.17764843665243524) -- (4.93, 0.19767578691172694) -- (4.94, 0.19894618017689) -- (4.95, 0.1994744094806289) -- (4.96, 0.19778440299585046) -- (4.970000000000001, 0.1806292676865254) -- (4.98, 0.1927146966084676) -- (4.99, 0.18762088941065658) -- (5.0, 0.18696303485149857) -- (5.01, 0.21062836277453864) -- (5.0200000000000005, 0.2346039423037656) -- (5.03, 0.22958388000832955) -- (5.04, 0.23707853509843294) -- (5.050000000000001, 0.2388321681426053) -- (5.0600000000000005, 0.2499766560480241) -- (5.07, 0.2681608701011314) -- (5.08, 0.27724452973449004) -- (5.09, 0.278593985580782) -- (5.1000000000000005, 0.2678030890680097) -- (5.11, 0.25829413934868034) -- (5.12, 0.26896304759769113) -- (5.13, 0.2673357265327721) -- (5.140000000000001, 0.2504390038513554) -- (5.15, 0.27014740026587564) -- (5.16, 0.2600495678446209) -- (5.17, 0.2468397638740145) -- (5.18, 0.2308394794742706) -- (5.19, 0.21248525975284394) -- (5.2, 0.1923110962518332) -- (5.210000000000001, 0.17092805331053249) -- (5.220000000000001, 0.14900168524461208) -- (5.23, 0.1272278058622552) -- (5.24, 0.10630717183669108) -- (5.25, 0.08691964145512139) -- (5.26, 0.06969837026377242) -- (5.27, 0.05520460512903382) -- (5.28, 0.04390363823431376) -- (5.29, 0.036142482532374474) -- (5.300000000000001, 0.03212983017349893) -- (5.3100000000000005, 0.03191885542859909) -- (5.32, 0.03539342362728182) -- (5.33, 0.04225826763132976) -- (5.34, 0.05203369336219451) -- (5.35, 0.0640553759036931) -- (5.36, 0.07747980769844459) -- (5.37, 0.09129596035911765) -- (5.380000000000001, 0.10434372161278753) -- (5.390000000000001, 0.11533966890004411) -- (5.4, 0.12291074114698593) -- (5.4,-3) -- cycle;
\fill[blue,opacity=0.3](5.5,-3) -- (5.5, 0.12291074114698593) -- (5.51, 0.1950766528886293) -- (5.52, 0.24200926113724006) -- (5.53, 0.2712938679570716) -- (5.54, 0.2888723911039713) -- (5.55, 0.2992824262429433) -- (5.56, 0.30587502902418484) -- (5.57, 0.3110121066369235) -- (5.58, 0.31624430846038326) -- (5.59, 0.3224703054312061) -- (5.6, 0.3300783477466577) -- (5.61, 0.33907099052294526) -- (5.62, 0.3491738770279719) -- (5.63, 0.3599294691078594) -- (5.64, 0.370776614426565) -- (5.65, 0.38111684013791797) -- (5.66, 0.39036826260940677) -- (5.67, 0.3980080028170425) -- (5.68, 0.4036039970306254) -- (5.69, 0.40683709240874616) -- (5.7, 0.4075143171228428) -- (5.71, 0.40557421462965126) -- (5.72, 0.4010851317113529) -- (5.73, 0.39423734990278975) -- (5.74, 0.3853299499250271) -- (5.75, 0.37475330343260255) -- (5.76, 0.3563436163344125) -- (5.77, 0.3546293581133293) -- (5.78, 0.3434681248208313) -- (5.79, 0.3327381589801685) -- (5.8, 0.32249234960120016) -- (5.81, 0.3105946161151378) -- (5.82, 0.3039600186535996) -- (5.83, 0.28688840435559215) -- (5.84, 0.27631910871380294) -- (5.85, 0.2747579808655225) -- (5.86, 0.26610266617760764) -- (5.87, 0.2683769995776154) -- (5.88, 0.2798991507270649) -- (5.89, 0.2972375574794057) -- (5.9, 0.32102948738582615) -- (5.91, 0.32322284786522776) -- (5.92, 0.31511228334186053) -- (5.93, 0.31770309839586575) -- (5.94, 0.3043232265728078) -- (5.95, 0.2832286162768577) -- (5.96, 0.2858645872666763) -- (5.97, 0.2567921162831364) -- (5.98, 0.25278514529597557) -- (5.99, 0.2297799813267726) -- (6.0, 0.22430275029217728) -- (6.01, 0.22633700091393613) -- (6.02, 0.20203059129091314) -- (6.03, 0.19204121569457774) -- (6.04, 0.1735901411541578) -- (6.05, 0.17860360496082062) -- (6.0600000000000005, 0.22025362930432146) -- (6.07, 0.22647129893759455) -- (6.08, 0.2452841776959024) -- (6.09, 0.27831045314021513) -- (6.1, 0.3040172354034933) -- (6.11, 0.3218491799129439) -- (6.12, 0.3293557371379631) -- (6.13, 0.3468652114865258) -- (6.14, 0.34069034881962607) -- (6.15, 0.3487454260558381) -- (6.16, 0.33502678233029454) -- (6.17, 0.3367456213349664) -- (6.18, 0.33924651126943545) -- (6.19, 0.3259352151060257) -- (6.2, 0.3351175565011755) -- (6.21, 0.32383773251777914) -- (6.22, 0.3025265471902197) -- (6.23, 0.29184398704110365) -- (6.24, 0.27657253508962204) -- (6.25, 0.2708614764435344) -- (6.26, 0.2562853073192842) -- (6.27, 0.2700283808179787) -- (6.28, 0.2620687792273508) -- (6.29, 0.2884015360479295) -- (6.3, 0.2949241784037812) -- (6.3100000000000005, 0.29121957994539427) -- (6.32, 0.3166675184892683) -- (6.33, 0.31889690773197693) -- (6.34, 0.312622530208679) -- (6.35, 0.3037430762082276) -- (6.36, 0.2992760577627454) -- (6.37, 0.3120572822123504) -- (6.38, 0.3015057161801582) -- (6.39, 0.2949393021572416) -- (6.4, 0.3032230335291212) -- (6.41, 0.309692421267794) -- (6.42, 0.29008219147393366) -- (6.43, 0.29982693859833404) -- (6.44, 0.28663426122514096) -- (6.45, 0.29754001868824165) -- (6.46, 0.3035160622976074) -- (6.47, 0.30675757118512953) -- (6.48, 0.31329626865856225) -- (6.49, 0.31000751490920925) -- (6.5, 0.32862872883316396) -- (6.51, 0.32064796035167137) -- (6.52, 0.3158760370562475) -- (6.53, 0.30076133063703425) -- (6.54, 0.2752668594243774) -- (6.55, 0.2828525168901326) -- (6.5600000000000005, 0.2785026421420431) -- (6.57, 0.25431789256062) -- (6.58, 0.23898488057598666) -- (6.59, 0.21523907035225673) -- (6.6, 0.2100289328185864) -- (6.61, 0.22812729125449938) -- (6.62, 0.23468648292710514) -- (6.63, 0.2566632765750441) -- (6.640000000000001, 0.24999062616202805) -- (6.65, 0.26360141741461085) -- (6.66, 0.2842660179076294) -- (6.67, 0.2769942093883895) -- (6.68, 0.29923414302046863) -- (6.6899999999999995, 0.31105152915158135) -- (6.7, 0.3180046787445816) -- (6.71, 0.3051889334223304) -- (6.72, 0.2896068141897182) -- (6.73, 0.2682238733296975) -- (6.74, 0.2563200885198964) -- (6.75, 0.26102823269063496) -- (6.76, 0.25396322920696013) -- (6.77, 0.2372007247914635) -- (6.78, 0.21874468908207356) -- (6.79, 0.22825200657613456) -- (6.8, 0.20209326491871923) -- (6.8100000000000005, 
0.17647321657206416) -- (6.82, 0.18643508870458336) -- (6.83, 0.15462426880618899) -- (6.84, 0.13810319771610757) -- (6.85, 0.11479261815267255) -- (6.86, 0.12742329130550709) -- (6.87, 0.1424672298799308) -- (6.88, 0.1352717413318057) -- (6.890000000000001, 0.13988266122672208) -- (6.9, 0.15703009061828763) -- (6.91, 0.14525602387440534) -- (6.92, 0.14145960026143486) -- (6.93, 0.16716773115098948) -- (6.9399999999999995, 0.15990750178413923) -- (6.95, 0.18445224405025792) -- (6.96, 0.19484575016623093) -- (6.97, 0.20566380344128107) -- (6.98, 0.227972803192257) -- (6.99, 0.24292364770484107) -- (7.0, 0.2436230539007264) -- (7.01, 0.24352677715651802) -- (7.02, 0.278176844260742) -- (7.03, 0.2791946350075714) -- (7.04, 0.28962140625908567) -- (7.05, 0.326524384131525) -- (7.0600000000000005, 0.3216386813882537) -- (7.07, 0.32162356121497143) -- (7.08, 0.30696567855193563) -- (7.09, 0.2943548687488767) -- (7.1, 0.29245006251798983) -- (7.11, 0.2928572249741567) -- (7.12, 0.2823502677206839) -- (7.13, 0.2765065394539235) -- (7.140000000000001, 0.2433930847266742) -- (7.15, 0.22348278992515525) -- (7.16, 0.21717148276008622) -- (7.17, 0.2006338446884197) -- (7.18, 0.19431705809638378) -- (7.1899999999999995, 0.19114576002993414) -- (7.2, 0.18282749679578145) -- (7.21, 0.1912632117611317) -- (7.22, 0.17764843665243524) -- (7.23, 0.19767578691172694) -- (7.24, 0.19894618017689) -- (7.25, 0.1994744094806289) -- (7.26, 0.19778440299585046) -- (7.27, 0.1806292676865254) -- (7.28, 0.1927146966084676) -- (7.29, 0.18762088941065658) -- (7.3, 0.18696303485149857) -- (7.3100000000000005, 0.21062836277453864) -- (7.32, 0.2346039423037656) -- (7.33, 0.22958388000832955) -- (7.34, 0.23707853509843294) -- (7.35, 0.2388321681426053) -- (7.36, 0.2499766560480241) -- (7.37, 0.2681608701011314) -- (7.38, 0.27724452973449004) -- (7.390000000000001, 0.278593985580782) -- (7.4, 0.2678030890680097) -- (7.41, 0.25829413934868034) -- (7.42, 0.26896304759769113) -- (7.43, 0.2673357265327721) -- (7.4399999999999995, 0.2504390038513554) -- (7.45, 0.27014740026587564) -- (7.46, 0.2600495678446209) -- (7.47, 0.2468397638740145) -- (7.48, 0.2308394794742706) -- (7.49, 0.21248525975284394) -- (7.5, 0.1923110962518332) -- (7.51, 0.17092805331053249) -- (7.52, 0.14900168524461208) -- (7.53, 0.1272278058622552) -- (7.54, 0.10630717183669108) -- (7.55, 0.08691964145512139) -- (7.5600000000000005, 0.06969837026377242) -- (7.57, 0.05520460512903382) -- (7.58, 0.04390363823431376) -- (7.59, 0.036142482532374474) -- (7.6, 0.03212983017349893) -- (7.609999999999999, 0.03191885542859909) -- (7.62, 0.03539342362728182) -- (7.63, 0.04225826763132976) -- (7.640000000000001, 0.05203369336219451) -- (7.65, 0.0640553759036931) -- (7.66, 0.07747980769844459) -- (7.67, 0.09129596035911765) -- (7.68, 0.10434372161278753) -- (7.6899999999999995, 0.11533966890004411) -- (7.7, 0.12291074114698593) -- (7.7,-3) -- cycle;
\fill[blue,opacity=0.3](7.8,-3) -- (7.8, 0.12291074114698593) -- (7.81, 0.1950766528886293) -- (7.819999999999999, 0.24200926113724006) -- (7.83, 0.2712938679570716) -- (7.84, 0.2888723911039713) -- (7.85, 0.2992824262429433) -- (7.859999999999999, 0.30587502902418484) -- (7.87, 0.3110121066369235) -- (7.88, 0.31624430846038326) -- (7.89, 0.3224703054312061) -- (7.8999999999999995, 0.3300783477466577) -- (7.91, 0.33907099052294526) -- (7.92, 0.3491738770279719) -- (7.93, 0.3599294691078594) -- (7.9399999999999995, 0.370776614426565) -- (7.95, 0.38111684013791797) -- (7.96, 0.39036826260940677) -- (7.97, 0.3980080028170425) -- (7.9799999999999995, 0.4036039970306254) -- (7.99, 0.40683709240874616) -- (8.0, 0.4075143171228428) -- (8.01, 0.40557421462965126) -- (8.02, 0.4010851317113529) -- (8.03, 0.39423734990278975) -- (8.04, 0.3853299499250271) -- (8.05, 0.37475330343260255) -- (8.06, 0.3563436163344125) -- (8.07, 0.3546293581133293) -- (8.08, 0.3434681248208313) -- (8.09, 0.3327381589801685) -- (8.1, 0.32249234960120016) -- (8.11, 0.3105946161151378) -- (8.12, 0.3039600186535996) -- (8.129999999999999, 0.28688840435559215) -- (8.14, 0.27631910871380294) -- (8.15, 0.2747579808655225) -- (8.16, 0.26610266617760764) -- (8.17, 0.2683769995776154) -- (8.18, 0.2798991507270649) -- (8.19, 0.2972375574794057) -- (8.2, 0.32102948738582615) -- (8.209999999999999, 0.32322284786522776) -- (8.22, 0.31511228334186053) -- (8.23, 0.31770309839586575) -- (8.24, 0.3043232265728078) -- (8.25, 0.2832286162768577) -- (8.26, 0.2858645872666763) -- (8.27, 0.2567921162831364) -- (8.28, 0.25278514529597557) -- (8.29, 0.2297799813267726) -- (8.3, 0.22430275029217728) -- (8.31, 0.22633700091393613) -- (8.32, 0.20203059129091314) -- (8.33, 0.19204121569457774) -- (8.34, 0.1735901411541578) -- (8.35, 0.17860360496082062) -- (8.36, 0.22025362930432146) -- (8.37, 0.22647129893759455) -- (8.379999999999999, 0.2452841776959024) -- (8.39, 0.27831045314021513) -- (8.4, 0.3040172354034933) -- (8.41, 0.3218491799129439) -- (8.42, 0.3293557371379631) -- (8.43, 0.3468652114865258) -- (8.44, 0.34069034881962607) -- (8.45, 0.3487454260558381) -- (8.459999999999999, 0.33502678233029454) -- (8.47, 0.3367456213349664) -- (8.48, 0.33924651126943545) -- (8.49, 0.3259352151060257) -- (8.5, 0.3351175565011755) -- (8.51, 0.32383773251777914) -- (8.52, 0.3025265471902197) -- (8.53, 0.29184398704110365) -- (8.54, 0.27657253508962204) -- (8.55, 0.2708614764435344) -- (8.56, 0.2562853073192842) -- (8.57, 0.2700283808179787) -- (8.58, 0.2620687792273508) -- (8.59, 0.2884015360479295) -- (8.6, 0.2949241784037812) -- (8.61, 0.29121957994539427) -- (8.62, 0.3166675184892683) -- (8.629999999999999, 0.31889690773197693) -- (8.64, 0.312622530208679) -- (8.65, 0.3037430762082276) -- (8.66, 0.2992760577627454) -- (8.67, 0.3120572822123504) -- (8.68, 0.3015057161801582) -- (8.69, 0.2949393021572416) -- (8.7, 0.3032230335291212) -- (8.709999999999999, 0.309692421267794) -- (8.72, 0.29008219147393366) -- (8.73, 0.29982693859833404) -- (8.74, 0.28663426122514096) -- (8.75, 0.29754001868824165) -- (8.76, 0.3035160622976074) -- (8.77, 0.30675757118512953) -- (8.78, 0.31329626865856225) -- (8.79, 0.31000751490920925) -- (8.8, 0.32862872883316396) -- (8.81, 0.32064796035167137) -- (8.82, 0.3158760370562475) -- (8.83, 0.30076133063703425) -- (8.84, 0.2752668594243774) -- (8.85, 0.2828525168901326) -- (8.86, 0.2785026421420431) -- (8.87, 0.25431789256062) -- (8.879999999999999, 0.23898488057598666) -- (8.89, 0.21523907035225673) -- (8.9, 0.2100289328185864) -- (8.91, 0.22812729125449938) -- (8.92, 0.23468648292710514) -- (8.93, 0.2566632765750441) -- (8.94, 0.24999062616202805) -- (8.95, 0.26360141741461085) -- (8.959999999999999, 0.2842660179076294) -- (8.969999999999999, 0.2769942093883895) -- (8.98, 0.29923414302046863) -- (8.99, 0.31105152915158135) -- (9.0, 0.3180046787445816) -- (9.01, 0.3051889334223304) -- (9.02, 0.2896068141897182) -- (9.03, 0.2682238733296975) -- (9.04, 0.2563200885198964) -- (9.05, 0.26102823269063496) -- (9.06, 0.25396322920696013) -- (9.07, 0.2372007247914635) 
-- (9.08, 0.21874468908207356) -- (9.09, 0.22825200657613456) -- (9.1, 0.20209326491871923) -- (9.11, 0.17647321657206416) -- (9.12, 0.18643508870458336) -- (9.129999999999999, 0.15462426880618899) -- (9.14, 0.13810319771610757) -- (9.15, 0.11479261815267255) -- (9.16, 0.12742329130550709) -- (9.17, 0.1424672298799308) -- (9.18, 0.1352717413318057) -- (9.19, 0.13988266122672208) -- (9.2, 0.15703009061828763) -- (9.209999999999999, 0.14525602387440534) -- (9.219999999999999, 0.14145960026143486) -- (9.23, 0.16716773115098948) -- (9.24, 0.15990750178413923) -- (9.25, 0.18445224405025792) -- (9.26, 0.19484575016623093) -- (9.27, 0.20566380344128107) -- (9.28, 0.227972803192257) -- (9.29, 0.24292364770484107) -- (9.3, 0.2436230539007264) -- (9.31, 0.24352677715651802) -- (9.32, 0.278176844260742) -- (9.33, 0.2791946350075714) -- (9.34, 0.28962140625908567) -- (9.35, 0.326524384131525) -- (9.36, 0.3216386813882537) -- (9.37, 0.32162356121497143) -- (9.379999999999999, 0.30696567855193563) -- (9.39, 0.2943548687488767) -- (9.4, 0.29245006251798983) -- (9.41, 0.2928572249741567) -- (9.42, 0.2823502677206839) -- (9.43, 0.2765065394539235) -- (9.44, 0.2433930847266742) -- (9.45, 0.22348278992515525) -- (9.46, 0.21717148276008622) -- (9.469999999999999, 0.2006338446884197) -- (9.48, 0.19431705809638378) -- (9.49, 0.19114576002993414) -- (9.5, 0.18282749679578145) -- (9.51, 0.1912632117611317) -- (9.52, 0.17764843665243524) -- (9.53, 0.19767578691172694) -- (9.54, 0.19894618017689) -- (9.55, 0.1994744094806289) -- (9.56, 0.19778440299585046) -- (9.57, 0.1806292676865254) -- (9.58, 0.1927146966084676) -- (9.59, 0.18762088941065658) -- (9.6, 0.18696303485149857) -- (9.61, 0.21062836277453864) -- (9.62, 0.2346039423037656) -- (9.629999999999999, 0.22958388000832955) -- (9.64, 0.23707853509843294) -- (9.65, 0.2388321681426053) -- (9.66, 0.2499766560480241) -- (9.67, 0.2681608701011314) -- (9.68, 0.27724452973449004) -- (9.69, 0.278593985580782) -- (9.7, 0.2678030890680097) -- (9.71, 0.25829413934868034) -- (9.719999999999999, 0.26896304759769113) -- (9.73, 0.2673357265327721) -- (9.74, 0.2504390038513554) -- (9.75, 0.27014740026587564) -- (9.76, 0.2600495678446209) -- (9.77, 0.2468397638740145) -- (9.78, 0.2308394794742706) -- (9.79, 0.21248525975284394) -- (9.8, 0.1923110962518332) -- (9.81, 0.17092805331053249) -- (9.82, 0.14900168524461208) -- (9.83, 0.1272278058622552) -- (9.84, 0.10630717183669108) -- (9.85, 0.08691964145512139) -- (9.86, 0.06969837026377242) -- (9.87, 0.05520460512903382) -- (9.879999999999999, 0.04390363823431376) -- (9.89, 0.036142482532374474) -- (9.9, 0.03212983017349893) -- (9.91, 0.03191885542859909) -- (9.92, 0.03539342362728182) -- (9.93, 0.04225826763132976) -- (9.94, 0.05203369336219451) -- (9.95, 0.0640553759036931) -- (9.96, 0.07747980769844459) -- (9.969999999999999, 0.09129596035911765) -- (9.98, 0.10434372161278753) -- (9.99, 0.11533966890004411) -- (10.0, 0.12291074114698593) -- (10,-3) -- cycle;
    
    \fill[gray,opacity=0.85] (0.9,-3-0.1) -- (10,-3-0.1) -- (10,-3-0.1-0.3) -- (0.9,-3-0.1-0.3) -- cycle;
    
    \draw[<->, very thick] (0,-3.5) -- (0,-2.5);
    \node[anchor=west] at (11,-3) {bottom boundary};
    \node[anchor=west] at (11,0) {top boundary};
    \node[] at (5.5,-1.5) {fluid};
    \node[] at (5.5,1) {vacuum};
\end{tikzpicture}
    \caption{Cross-sectional side view of the top free boundary and bottom rigid oscillating boundary of a horizontally periodic fluid.}
\end{figure}
 
\subsubsection{Fluid domain and boundaries}
We introduce the horizontal cross section $\Sigma = (L_1\mathbb T)\times(L_2\mathbb T)$ for horizontal periodicity parameters $L_1,L_2>0$, and we assume that the moving upper boundary of the fluid is given by the graph of an unknown function $\tilde\eta:\Sigma\times\mathbb R^+\to\mathbb R$, so that the moving fluid domain is modeled by the set
\begin{equation}
    \tilde\Omega(t) = \{x=(x',x_3)\in\Sigma\times\mathbb R : Af(\omega t)-b < x_3 < \tilde\eta(x',t)\}.
\end{equation}
Note that the lower boundary of $\tilde\Omega(t)$ is the oscillating set
\begin{equation}
    \tilde\Sigma_b(t) = \{x=(x',x_3)\in\Sigma\times\mathbb R : x_3 = Af(\omega t)-b\},
\end{equation}
while the moving upper surface is
\begin{equation}
    \tilde\Sigma(t) = \{x=(x',x_3)\in\Sigma\times\mathbb R : x_3 = \tilde\eta(x',t)\}.
\end{equation}

\subsubsection{Equations of motion}
For each $t\geq 0$, the fluid is described by its velocity and pressure functions $(\tilde u,\tilde p):\tilde\Omega(t)\to\mathbb R^3\times\mathbb R$. We require that $(\tilde u,\tilde p,\tilde\eta)$ satisfy the incompressible Navier-Stokes equations in $\tilde\Omega(t)$ for $t>0$:
\begin{equation}\label{system:ns-original}
\begin{cases}
    \partial_t\tilde u+\tilde u\cdot\nabla\tilde u+\nabla\tilde p - \mu\Delta\tilde u = -ge_3 & \text{in $\tilde\Omega(t)$} \\
    \diverge\tilde u = 0 & \text{in $\tilde\Omega(t)$} \\
    \partial_t \tilde{\eta} + \tilde u_1\partial_1\tilde\eta + \tilde u_2\partial_2\tilde\eta = \tilde u_3  & \text{on $\tilde\Sigma(t)$} \\
    (\tilde p I - \mu\mathbb D\tilde u)\nu = (P_{\mathrm{ext}}-\sigma\mathfrak H(\tilde\eta))\nu & \text{on $\tilde\Sigma(t)$} \\
    \tilde u = A\omega f'(\omega t)e_3 & \text{on $\tilde\Sigma_b(t)$}
\end{cases}.
\end{equation}
Here, $\mu>0$ is the fluid viscosity, $(\mathbb D\tilde u)_{ij} = \partial_i\tilde u_j + \partial_j\tilde u_i$ is the symmetric gradient of $\tilde u$, $\nu$ is the outward-point unit normal vector on $\tilde\Sigma(t)$, $I$ is the $3\times 3$ identity matrix, $P_{\mathrm{ext}}\in\mathbb R$ is the constant pressure above the fluid, $\sigma>0$ is the surface tension coefficient, and
\begin{equation}
    \mathfrak H(\tilde\eta) = \diverge\left(\frac{\nabla\tilde\eta}{\sqrt{1+\abs{\nabla\tilde\eta}^2}}\right)
\end{equation}
is (minus) twice the mean curvature of $\Sigma(t)$, which models the force of surface tension on the free interface. The first two equations of \eqref{system:ns-original} are the standard incompressible Navier-Stokes equations, the third is the kinematic transport equation for $\tilde\eta$, the fourth is the balance of stress at the interface, and the fourth is the no-slip boundary condition at the bottom. The problem is augmented with initial data $\tilde\eta_0:\tilde\Sigma\to(−b + Af(0),\infty)$, which determines the initial domain $\tilde\Omega_0$, as well as an initial velocity field $\tilde u_0:\tilde\Omega_0\to\mathbb R^3$. Note that the assumption $\tilde\eta_0>-b+Af(0)$ on $\Sigma$ means that $\tilde\Omega_0$ is well-defined.

We will assume that the constant $b>0$ is chosen such that the mass of the fluid, which is conserved in time due to the incompressibility, is given by 
\begin{equation}
 \mathcal{M} := b \abs{\Sigma} = b L_1 L_2.
\end{equation}
Rewriting this condition in terms of $\tilde{\eta}$ shows that 
\begin{equation}
 b \abs{\Sigma} = \mathcal{M} =  \int_{\Sigma} [\tilde{\eta}(x',t) - (A f(\omega t) - b)] dx' = b \abs{\Sigma} +  \int_{\Sigma} [\tilde{\eta}(x',t) - A f(\omega t)] dx', 
\end{equation}
or equivalently 
\begin{equation}\label{zero_avg_proto}
 \int_{\Sigma} [\tilde{\eta}(x',t) - A f(\omega t)] dx' =0.
\end{equation}

\subsection{Recasting the problem in the oscillating frame}
Here we recast the problem in the oscillating fluid frame and make some convenient changes of unknowns in order to simplify further analysis.

\subsubsection{Change of coordinates}
First, we make a Galilean change of coordinates. The above formulation of the problem is intuitive as an external observer, but it is more convenient to view the problem from the frame of the fluid itself and fix the moving lower boundary. As such, we employ the following change of coordinates and unknowns:
\begin{equation}
\begin{aligned}
        \tilde u(x,t) &= \bar u(x',x_3-Af(\omega t),t)+A\omega f'(\omega t)e_3 \\
        \tilde p(x,t) &= \bar p(x',x_3-Af(\omega t),t) \\
        \tilde\eta(x',t) &= \bar\eta(x',t) + Af(\omega t)
\end{aligned}
\end{equation}

By plugging the above into \eqref{system:ns-original}, we obtain the equivalent set of equations  
\begin{equation}
    \begin{cases}
        \partial_t\bar u + \bar u\cdot\nabla \bar u+\nabla\bar p-\mu\Delta\bar u + A\omega^2f''(\omega t)e_3 = -g e_3 &  \text{in $\Omega(t)$} \\
        \diverge\bar u = 0 & \text{in $\Omega(t)$} \\
        \partial_t\bar \eta +\bar u_1\partial_1\bar \eta + \bar u_2\partial_2\bar \eta = \bar u_3 & \text{on $\Sigma(t)$} \\
        (\bar pI-\mu\mathbb D\bar u)\nu = \left(P_{ext} -\sigma\mathfrak H(\bar\eta)\right)\nu & \text{on $\Sigma(t)$} \\
        \bar u = 0 & \text{on $\Sigma_b$}
    \end{cases}
\end{equation}
where
\begin{equation}
\begin{aligned}
    \Omega(t) &= \{x = (x',x_3)\in\Sigma\times\mathbb R : -b < x_3 < \bar\eta(x',t)\} \\
    \Sigma(t) &= \{x = (x',x_3)\in\Sigma\times\mathbb R : x_3 = \bar\eta(x',t)\} \\
    \Sigma_b &= \{x = (x',x_3)\in\Sigma\times\mathbb R : x_3 = -b\}
\end{aligned}
\end{equation}
are the new versions of the domains where the lower boundary is now unmoving and the upper boundary is now defined by the new graph function $\bar\eta$. 

\subsubsection{Modifying the pressure}
Next, we modify the pressure to remove the term $A\omega^2 f''(\omega t)e_3 + g e_3$ from the first equation and to eliminate $P_{ext}$ on the boundary. To this end we define
\begin{equation}
    \bar p_{\mathrm{new}} \coloneqq \bar p_{\mathrm{old}} -  P_{ext} + (g + A\omega^2f''(\omega t) )x_3 
\end{equation}
in order to arrive (after dropping the subscript) at the equivalent problem
\begin{equation}\label{system:ns-final}
    \begin{cases}
        \partial_t\bar u + \bar u\cdot\nabla \bar u+\nabla\bar p-\mu\Delta\bar u = 0 &  \text{in $\Omega(t)$} \\
        \diverge\bar u = 0 & \text{in $\Omega(t)$} \\
        \partial_t\bar\eta + \bar u_1\partial_1\bar\eta + \bar u_2\partial_2\bar\eta = \bar u_3 & \text{on $\Sigma(t)$} \\
        (\bar pI-\mu\mathbb D\bar u)\nu = \left(-\sigma\mathfrak H(\bar\eta)+\left(g+A\omega^2f''(\omega t)\right)\bar\eta\right)\nu & \text{on $\Sigma(t)$} \\
        \bar u = 0 & \text{on $\Sigma_b$}.
    \end{cases}
\end{equation}
In summary, to go from \eqref{system:ns-original} to \eqref{system:ns-final} we have made the following changes of unknowns:
\begin{equation}
\begin{aligned}
    \tilde u(x,t) &= \bar u(x',x_3-Af(\omega t),t) + A\omega f'(\omega t)e_3 \\
    \tilde p(x,t) &= \bar p(x',x_3-Af(\omega t),t) + P_{\mathrm{ext}} - \left(g+A\omega^2f''(\omega t)\right)\left(x_3-Af(\omega t)\right) \\
    \tilde\eta(x',t) &= \bar\eta(x',t) + Af(\omega t).
\end{aligned}
\end{equation}

Note that \eqref{zero_avg_proto} now becomes 
\begin{equation}\label{zero_avg}
 \int_{\Sigma} \bar{\eta}(x',t) dx' =0 \text{ for } t \ge 0.
\end{equation}
However, for sufficiently regular solutions to \eqref{system:ns-final} we have that  $\pt_t\bar\eta = \bar u\cdot\nu\sqrt{1+(\pt_1\bar\eta)^2+(\pt_2\bar\eta)^2}$, and hence
\begin{equation}
    \ddt\int_\Sigma\bar\eta(x',t) dx' = \int_\Sigma\pt_t\bar\eta(x',t) dx' = \int_{\Sigma(t)}\bar u\cdot\nu = \int_{\Omega(t)}\diverge\bar u = 0.
\end{equation}
Thus \eqref{zero_avg} is satisfied provided that the initial surface function satisfies the ``zero average'' condition
\begin{equation}
    \frac1{L_1L_2}\int_\Sigma \bar\eta_0 = 0,
\end{equation}
a condition that we henceforth assume.  Note, though, that this condition is no real loss of generality, as it can always be achieved with a coordinate shift via the relation between the fluid mass $\mathcal{M}$ and the parameter $b$.  See, for instance, the introduction of \cite{guo2013almost} for an explanation of how the coordinate shift works.

\subsection{Steady oscillating solution}
Note that $\overline U(x,t) = 0, \overline P(x,t) = 0,\overline H(x,t) = 0$  is a solution to the reparameterized system \eqref{system:ns-final} when we set $\bar u = \overline U, \bar p = \overline P, \bar\eta = \overline H$. In the original system, this corresponds to the steady oscillation solution
\begin{equation}\label{eqn:steady-oscillating-solution}
\begin{aligned}
    \tilde U(x,t) &= A\omega f'(\omega t)e_3 \\
    \tilde P(x,t) &= P_{\mathrm{ext}} - \left(g+A\omega^2f''(\omega t)\right)\left(x_3-Af(\omega t)\right) \\
    \tilde H(x,t) &= Af(\omega t)
\end{aligned}
\end{equation}
and it is easy to check that this indeed satisfies system \eqref{system:ns-original} along with the fixed mass condition $\mathcal{M} = b L_1 L_2 = b \abs{\Sigma}$.

We will study the Faraday problem in the reparametrization \eqref{system:ns-final}, with the aim of showing that the above steady oscillation solution is asymptotically stable for some range of the parameters.  In order to justify why we might expect such a stability result, consider the natural energy-dissipation equation associated with \eqref{system:ns-final} (for details of the derivation, see Proposition \ref{prop:geom-ed}):
\begin{equation}\label{eqn:full-ed}
    \ddt\left(\int_{\Omega(t)}\frac{\abs{\bar u}^2}2+\int_{\Sigma} g\frac{\abs{\bar\eta}^2}2 + \sigma\sqrt{1+\abs{\nabla\bar\eta}^2}\right)+\int_{\Omega(t)}\frac{\mu\abs{\mathbb D\bar u}^2}2 = -(A\omega^2 f''(\omega t))\int_{\Sigma} \bar\eta \pt_t\bar\eta .
\end{equation}
This identity establishes that the competition between the viscous dissipation (the integral with $\mu$ on the left) and power supplied by the oscillation of the plate (the term on the right) will determine the stability of the system.  In particular, it shows that if we can absorb the oscillation term with the dissipation, then we should expect stability, and this is indeed what we will prove.  Obviously, the natural dissipation term involves only the velocity field $\bar{u}$ and does not control $\bar{\eta}$ or $\pt_t \bar{\eta}$, so to complete our analysis we will need to introduce a host of auxiliary estimates that provide dissipative control of these terms.

\subsection{Previous work}

A brief survey of previous mathematical work on the Faraday problem was recorded above in Section \ref{sec:faraday}.  To the best of our knowledge, there are no rigorous results on the fully nonlinear analysis of this problem, either in the stable or unstable parameter regimes.  However, when $f =0$, i.e. when the rigid bottom is not oscillated, the fully nonlinear dynamics of the free boundary problem \eqref{system:ns-final} and its variants are well-understood for small data.  Nishida-Teramoto-Yoshihara \cite{nishida_teramoto_yoshihara} constructed global solutions for the problem with surface tension and showed that the solutions decay to equilibrium at an exponential rate.  The corresponding problem without surface tension was handled by Hataya \cite{hataya2009decaying}, who constructed global solutions decaying at a fixed algebraic rate, and later by Guo-Tice \cite{guo2013almost}, who constructed global solutions that decay almost exponentially.  Tan-Wang \cite{tan2014zero} established the vanishing surface tension limit.  In the non-periodic setting many related results are known; see for instance the work of Beale \cite{beale1981initial,beale1984large}, Beale-Nishida \cite{beale_nishida}, Tani-Tanaka \cite{tani_tanaka}, and Guo-Tice \cite{guo_tice_inf}.  The stability of the periodic problem without Faraday oscillation has also been studied with more physical effects included.  Gravitational fields with horizontal components, corresponding to sliding along a tilted incline plane, were studied by Tice \cite{tice2018asymptotic}.  The coupling to the MHD system was studied by Tan-Wang \cite{tan_wang_mhd}.  Remond--Tiedrez-Tice \cite{remondtiedrez_tice} studied stability with more general surface forces generated by bending energies.

\subsection{Reformulation in a flattened coordinate system}\label{section:reformulation}
The moving domain $\Omega(t)$ is inconvenient for analysis, so we will reformulate the problem \eqref{system:ns-final} in the fixed equilibrium domain
\begin{equation}\label{definition:Omega}
    \Omega = \{x = (x',x_3)\in\Sigma\times\mathbb R  : -b < x_3 < 0\}.
\end{equation}
We will think of $\Sigma$ as the upper boundary of $\Omega$ and view $\bar\eta$ as a function on $\Sigma\times\mathbb R^+$. We then define
\begin{equation}
    \hat\eta \coloneqq \calP\bar\eta
\end{equation}
to be the harmonic extension of $\bar\eta$ into the lower half space as in Section \ref{section:poisson-extension}. Then, we flatten the coordinate domain via the mapping $\Phi:\Omega\times\mathbb R^+\to\Omega(t)$
\begin{equation}\label{definition:Phi}
    \Phi(x,t) = \left(x_1,x_2,x_3 + \hat\eta(x',t)\left(1+\frac{x_3}b\right)\right).
\end{equation}
Note that $\Phi(\cdot,t)$ is smooth and extends to $\overline\Omega$ in such a way that $\Phi(\Sigma,t) = \Sigma(t)$ and $\Phi(\Sigma_b,t) = \Sigma_b$, i.e.\ $\Phi$ maps $\Sigma$ to the free surface and keeps the lower surface fixed. We have
\begin{equation}
    \nabla\Phi = 
    \begin{pmatrix}
        1&0&0\\0&1&0\\A&B&J
    \end{pmatrix},\qquad
    \mathcal A \coloneqq \left(\nabla\Phi^{-1}\right)^\top = 
    \begin{pmatrix}
        1&0&-AK\\0&1&-BK\\0&0&K
    \end{pmatrix}
\end{equation}
where under the notational convenience $\tilde b = (1+x_3/b)$ we have
\begin{equation}
    A=\partial_1\hat\eta\tilde b,\qquad B=\partial_2\hat\eta\tilde b,\qquad J=\left(1+\frac{\hat\eta}b+\pt_3\hat\eta\tilde b\right),\qquad K=J^{-1}.
\end{equation}
Note that $J = \det\nabla\Phi$ is the determinant of the transformation.

Using the matrix $\mathcal A$, we define a collection of $\mathcal A$-dependent differential operators. We define the differential operators $\nabla_{\mathcal A}$ and $\diverge_{\mathcal A}$ with their actions given by
\begin{equation}
    \left(\nabla_{\mathcal A}f\right)_i \coloneqq \mathcal A_{ij}\partial_j f,\qquad \diverge_{\mathcal A} X \coloneqq \mathcal A_{ij}\partial_jX_i
\end{equation}
for appropriate $f$ and $X$. We extend $\diverge_{\mathcal A}$ to act on symmetric tensors in the usual way. Now write the change of coordinates as
\begin{equation}
    u(x,t) = \bar u(\Phi(x,t),t), \qquad
    p(x,t) = \bar p(\Phi(x,t),t), \qquad
    \eta(x',t) = \bar\eta(x',t).
\end{equation}
We then also write
\begin{equation}
    \left(\mathbb D_{\mathcal A}u\right)_{ij} \coloneqq \mathcal A_{ik}\partial_ku_j + \mathcal A_{jk}\partial_ku_i,\qquad S_{\mathcal A}(u,p) \coloneqq pI - \mu\mathbb D_{\mathcal A}u,
\end{equation}
and we define
\begin{equation}
    \mathcal N \coloneqq (-\partial_1\bar\eta,-\partial_2\bar\eta,1)
\end{equation}
for the non-unit normal to $\Sigma(t)$. 
In this new coordinate system, the new system of PDEs \eqref{system:ns-final} becomes the following equivalent system:
\begin{equation}\label{system:ns-flattened}
\begin{cases}
    \partial_t u - \partial_t\hat\eta\tilde b K\partial_3 u + u\cdot\nabla_{\mathcal A} u + \diverge_{\mathcal A}S_{\mathcal A}(u,p) = 0 & \text{in $\Omega$} \\
    \diverge_{\mathcal A} u = 0 & \text{in $\Omega$} \\
    \partial_t\eta = u\cdot\mathcal N & \text{on $\Sigma$} \\
    S_{\mathcal A}(u,p)\mathcal N = \left(-\sigma\mathfrak H(\eta) + \left(g+A\omega^2f''(\omega t)\right)\eta\right)\mathcal N & \text{on $\Sigma$} \\
    u = 0 & \text{on $\Sigma_b$}
\end{cases}.
\end{equation}

\section{Main results and discussion}
\subsection{Notation and definitions}
In order to properly state our main results we must first introduce some notation and define various functionals that will be used throughout the paper. We begin with some notational conventions.

\emph{Einstein summation and constants:}
We will employ the Einstein convention of summing over repeated indices for vector and tensor operations. Throughout the paper $C > 0$ will denote a generic constant that can depend on $\Omega$ and its dimensions as well as on $g$, $\mu$, and the oscillation profile $f$, but not on the parameters $\sigma,$ $A,$ or $\omega$. Such constants are referred to as ``universal'', and they are allowed to change from one inequality to another. We employ the notation $a\lesssim b$ to mean that $a\leq Cb$ for a universal constant $C > 0$.

\emph{Norms:}
We write $H^k(\Omega)$ with $k\geq 0$ and $H^s(\Sigma)$ with $s\in\mathbb R$ for the usual $L^2$-based Sobolev spaces. In particular $H^0=L^2$. In the interest of concision, we neglect to write $H^k(\Omega)$ or $H^k(\Sigma)$ in our norms and typically write only $\norm{\cdot}_k$. The price we pay for this is some minor ambiguity in the set on which the norm is computed, but we mitigate potential confusion by always writing the space for the norm when traces are involved.

\emph{Multi-indices:}
We will write $\mathbb N^k$ for the usual set of multi-indices, where here we employ the convention that $0\in\mathbb N$. For $\al\in\mathbb N^k$ we define the spatial differential operator $\pt^\al = \pt_1^{\al_1}\pt_2^{\al_2}\dots\pt_k^{\al_k}$. We will also write $\mathbb N^{1+k}$ for the set of space-time multi-indices
\begin{equation}
    \mathbb N^{1+k} = \left\{(\alpha_0,\alpha_1,\dots,\alpha_k) : \text{$\alpha_i\in\mathbb N$ for $0\leq i\leq k$}\right\}.
\end{equation}
For a multi-index $\alpha\in\mathbb N^{1+k}$, we define the differential operator $\pt^\al = \pt_t^{\al_0}\pt_1^{\al_1}\dots\pt_k^{\al_k}$. Also, for a space-time multi-index $\al\in\mathbb N^{1+k}$ we use the parabolic counting scheme $\abs\al = 2\al_0 + \al_1 + \dots + \al_k$. 

\emph{Energy and dissipation functionals:}
Throughout the paper we will make frequent use of various energy and dissipation functionals, dependent on time. We define these now. The basic (with bars) and full (no bars) energy functionals, respectively, are defined as
\begin{equation}\label{definition:basic-E}
    \ov{\calE_n^\sigma} \coloneqq \sum_{\substack{\alpha\in\mathbb N^{1+2} \\ \abs\alpha\leq 2n}}\norm{\partial^\alpha u}_0^2 + g\norm{\partial^\alpha\eta}_0^2 + \sigma\norm{\nabla\partial^\alpha\eta}_0^2
\end{equation}
and
\begin{equation}\label{definition:full-E}
    \calE_n^\sigma \coloneqq \ov{\calE_n^\sigma} +\sum_{j=0}^{n}\norm{\partial_t^j u}_{2n-2j}^2 + \sum_{j=0}^{n-1}\norm{\partial_t^jp}_{2n-2j-1}^2+\sigma\norm{\eta}_{2n-2j+1}^2+\norm{\eta}_{2n}^2 + \sum_{j=1}^{n}\norm{\pt_t^j\eta}_{2n-2j+3/2}^2.
\end{equation}
The corresponding basic and full dissipation functionals are 
\begin{equation}\label{definition:basic-D}
    \ov\calD_n \coloneqq \sum_{\substack{\alpha\in\mathbb N^{1+2} \\ \abs\alpha\leq 2n}}\norm{\mathbb D\partial^\alpha u}_0^2
\end{equation}
and
\begin{equation}\label{definition:full-D}
\begin{aligned}
    \calD_n^\sigma &\coloneqq \ov\calD_n + \sum_{j=0}^{n}\norm{\pt_t^j u}_{2n-2j+1}^2 + \sum_{j=0}^{n-1}\norm{\pt_t^j p}_{2n-2j}^2 + \sum_{j=0}^{n-1}\left(\norm{\pt_t^j\eta}_{2n-2j-1/2}^2+\sigma^2\norm{\pt_t^j\eta}_{2n-2j+3/2}^2\right) \\
    &\hspace{6ex}+ \sum_{j=3}^{n+1}\norm{\pt_t^j\eta}_{2n-2j+5/2}^2 + \norm{\pt_t\eta}^2_{2n-1} +\sigma^2\norm{\pt_t\eta}^2_{2n+1/2} +  \norm{\pt_t^2\eta}_{2n-2}^2 + \sigma^2\norm{\pt_t^2\eta}_{2n-3/2}^2.
\end{aligned}
\end{equation}

We will also need to make frequent reference to two functionals that are not naturally of energy or dissipation type. We refer to these as
\begin{equation}\label{definition:F_n}
    \calF_n \coloneqq \norm{\eta}_{2n+1/2}^2
\end{equation}
and
\begin{equation}\label{definition:calK}
    \calK \coloneqq \norm{u}_{C_b^2(\Omega)}^2 + \norm{u}_{H^3(\Sigma)}^2 + \norm{p}_{H^3(\Sigma)}^2 + \norm{\eta}_{5/2}^2.
\end{equation}

\subsection{Local existence theory}

The main content of this work is the a priori estimates that can be combined with local existence theory in order to construct local-in-time solutions to the PDE of equation (\ref{system:ns-flattened}). In this section, we only state the local existence theory that is needed to make this work without proof.  Such an omission is justified by the abundance of similar local existence results based on the corresponding a priori estimates.  We refer, for instance, to the works \cite{guo2013almost,jang2016compressible, tan2014zero, wang2014viscous,wu2014well}. 

To state these local existence results, we need to introduce function spaces in which our solutions exist, as well as \emph{compatibility conditions}, which give sufficient constraints on our initial data for constructing solutions using our a priori estimates. Our function spaces are the following:
\begin{equation}
\begin{aligned}
    {}_0H^1(\Omega) &= \{v\in H^1(\Omega;\mathbb R^3) : v\mid_{\Sigma_b} = 0 \} \\
    \mathcal X_T &= \{u\in L^2([0,T];{}_0H^1(\Omega)) : \diverge_{\mathcal A(t)} u(t) = 0 \text{ for a.e. } t\in[0,T] \},
\end{aligned}
\end{equation}
where the $\mathcal A(t)$ here is determined by the $\eta:\Sigma\times[0,T]\to\mathbb R$ coming from the solution. We refer the reader to \cite{guo2013almost,jang2016compressible, tan2014zero, wang2014viscous,wu2014well} for the compatibility conditions, as they are simple yet cumbersome to record. 

When $\sigma>0$ is fixed and positive, we have the following local existence result, similar to that of  \cite{wang2014viscous}:

\begin{thm}[Local existence for fixed positive $\sigma$]\label{thm:local-existence-sigma}
    Let $\sigma>0$ be fixed and positive and let $n\geq 1$ be an integer and suppose that the initial data $(u_0,\eta_0)$ satisfy
    \begin{equation}
        \norm{u_0}_{2n}^2 + \norm{\eta_0}_{2n+1/2}^2 + \sigma\norm{\nabla\eta}_{2n}^2 < \infty
    \end{equation}
    as well as the natural compatibility conditions associated with $n$. Then there exist $0<\delta_*,T_*<1$ such that if
    \begin{equation}
        \norm{u_0}_{2n}^2 + \norm{\eta_0}_{2n+1/2}^2 + \sigma\norm{\nabla\eta}_{2n}^2 \leq\delta_*
    \end{equation}
    and $0<T\leq T_*$, then there exists a unique triple $(u,p,\eta)$ that achieves the initial data, solves (\ref{system:ns-flattened}), and obeys the estimates
    \begin{equation}
        \sup_{0\leq t\leq T}(\mathcal E_n^\sigma(t)+\mathcal F_n(t)) + \int_0^T\mathcal D_n^\sigma(t)~dt + \norm{\pt_t^{n+1}u}_{(\mathcal X_T)^*}^2\lesssim \norm{u_0}_{2n}^2 + \norm{\eta_0}_{2n+1/2}^2 + \sigma\norm{\nabla\eta}_{2n}^2.
    \end{equation}
\end{thm}

We also consider the vanishing surface tension regime, in which we obtain the following result by requiring $n$ to be larger \cite{guo2013almost,jang2016compressible,wu2014well,tan2014zero}:
\begin{thm}[Local existence for vanishing $\sigma$]\label{thm:local-existence}
    Let $n\geq 2$ be an integer and suppose that the initial data $(u_0,\eta_0)$ satisfy
    \begin{equation}
        \norm{u_0}_{2n}^2 + \norm{\eta_0}_{2n+1/2}^2 + \sigma\norm{\nabla\eta}_{2n}^2 < \infty
    \end{equation}
    as well as the natural compatibility conditions associated with $n$. Then there exist $0<\delta_*,T_*<1$ such that if
    \begin{equation}
        \norm{u_0}_{2n}^2 + \norm{\eta_0}_{2n+1/2}^2 + \sigma\norm{\nabla\eta}_{2n}^2 \leq\delta_*
    \end{equation}
    and $0<T\leq T_*$, then there exists a unique triple $(u,p,\eta)$ that achieves the initial data, solves (\ref{system:ns-flattened}), and obeys the estimates
    \begin{equation}
        \sup_{0\leq t\leq T}(\mathcal E_n^\sigma(t)+\mathcal F_n(t)) + \int_0^T\mathcal D_n^\sigma(t)~dt + \norm{\pt_t^{n+1}u}_{(\mathcal X_T)^*}^2\lesssim \norm{u_0}_{2n}^2 + \norm{\eta_0}_{2n+1/2}^2 + \sigma\norm{\nabla\eta}_{2n}^2.
    \end{equation}
\end{thm}

\subsection{Statement of main results}
The main result of this paper is the global well-posedness of the problem and decay of solutions, which establishes the asymptotic stability of the equilibrium solutions.  We begin with the result for a fixed value of surface tension.

\begin{thm}\label{thm:main-fixed-surface-tension}
Fix $\sigma >0$.  Suppose that initial data $(u_0,\eta_0)$ satisfy $\norm{u_0}_{2}^2 + \norm{\eta_0}_{5/2}^2 + \sigma\norm{\nabla\eta}_{2}^2< \infty$ as well as the compatibility conditions of Theorem \ref{thm:local-existence-sigma}. There exist constants $\gamma_0, \kappa_0 \in(0,1)$, both depending on $\sigma$, such that if 
\begin{equation}
\norm{u_0}_{2}^2 + \norm{\eta_0}_{5/2}^2 + \sigma\norm{\nabla\eta}_{2}^2 \leq \kappa_0 \text{ and } A\omega^2 + A\omega^3 \leq \gamma_0, 
\end{equation}
then there exists a unique (within the energy class) solution $(u,p,\eta)$ that solves \eqref{system:ns-flattened} on the temporal interval $(0,\infty)$ and achieves the initial data.  Moreover, there exists constants $\lb > 0$ and $C_0,C_1 > 0$, depending on $A$, $\omega$ and $\sigma$, such that the solution obeys the estimate
\begin{equation}
    \sup_{0\leq t \leq \infty} e^{\lb t} \calE_1^\sigma(t) + \int_0^\infty e^{\lb t} \calD_1^\sigma(t)\,dt \ls C_0 \calE_1^\sigma(0) \le C_1 \left(\norm{u_0}_{2}^2 + \norm{\eta_0}_{5/2}^2 + \sigma\norm{\nabla\eta}_{2}^2 \right)
\end{equation}
\end{thm}

Theorem \ref{thm:main-fixed-surface-tension} requires a fixed positive value of surface tension and guarantees that solutions return to equilibrium exponentially fast in the topology determined by $\mathcal{E}^\sigma_1$.  Our next main result considers the cases $\sigma=0$ and $\sigma$ small but positive. We view the latter as the ``vanishing surface tension'' regime, as we will employ it to establish this limit. In these cases we work in a more complicated functional setting that changes depending on whether $\sigma$ vanishes or not. We introduce this with the following functional, defined for any integer $N\geq 3$ and time $t\in[0,\infty]$:
\begin{equation}\label{definition:calG}
\calG_{2N}^\sigma(t) \coloneqq \sup_{0\leq r\leq t}\calE_{2N}^\sigma(r) + \int_0^t\calD_{2N}^\sigma(r)\,dr + \sup_{0\leq r\leq t}(1+r)^{4N-8}\calE_{N+2}^\sigma(r) + \sup_{0\leq r\leq t}\frac{\calF_{2N}(r)}{1+r},
\end{equation}
where here $\calE_n^\sigma$, $\calD_n^\sigma$, and $\calF_n$ are defined by \eqref{definition:full-E}, \eqref{definition:full-D}, and \eqref{definition:F_n}, respectively. Note that the condition $N\geq 3$ implies that $2N>N+2$ and that $4N-8>0$. 

We can now state our second main result.

\begin{thm}\label{thm:main-vanishing-surface-tension}
Let $\Omega$ be given by \eqref{definition:Omega}, let $N\geq3$, and define $\calG_{2N}^\sigma$ via \eqref{definition:calG}. Suppose that the initial data $(u_0,\eta_0)$ satisfy $ \norm{u_0}_{4N}^2 + \norm{\eta_0}_{4N+1/2}^2 + \sigma\norm{\nabla\eta}_{4N}^2<\infty$ as well as compatibility conditions of Theorem \ref{thm:local-existence}. There exist universal constants $\gamma_0,\kappa_0\in(0,1)$  such that if 
\begin{equation}
 \norm{u_0}_{4N}^2 + \norm{\eta_0}_{4N+1/2}^2 + \sigma\norm{\nabla\eta}_{4N}^2 \leq\kappa_0,\; 0\leq \sum_{\ell=2}^{2N+2}A\omega^\ell\leq\gamma_0, \text{ and } 0\leq\sigma\leq 1,  
\end{equation}
then there exists a unique (within the energy class) triple $(u,p,\eta)$ that solves the (\ref{system:ns-flattened}) on the temporal interval $(0,\infty)$, achieves the initial data, and obeys the estimate
\begin{equation}\label{eqn:main-vanishing-surface-tension}
    \calG_{2N}^\sigma(\infty)\lesssim \calE_{2N}^\sigma(0) + \calF_{2N}(0) \ls  \norm{u_0}_{4N}^2 + \norm{\eta_0}_{4N+1/2}^2 + \sigma\norm{\nabla\eta}_{4N}^2.
\end{equation}
\end{thm}

In particular, the bound in Theorem \ref{thm:main-vanishing-surface-tension} establishes the decay estimate
\begin{equation}
    \calE_{N+2}^\sigma(t)\lesssim \frac{\calE_{2N}^\sigma(0) + \calF_{2N}(0)}{(1+t)^{4N-8}}.
\end{equation}
This is an algebraic decay rate, slower than the exponential rate proved in Theorem  \ref{thm:main-fixed-surface-tension} with a fixed $\sigma>0$. Two remarks about this are in order. First, by choosing $N$ larger, we arrive at a faster rate of decay. In fact, by taking $N$ to be arbitrarily large we can achieve arbitrarily fast algebraic decay rates, which is what is known as ``almost exponential decay.''  The trade-off in the theorem is that faster decay requires smaller data in higher regularity classes. The second point is that when $0<\sigma\leq 1$ in the theorem, it is still possible to prove that $\calE_{2n}^\sigma$ decays exponentially by modifying the arguments used later in Theorem  \ref{thm:a-priori-est}. We neglect to state this properly here because we only care about the vanishing surface tension limit, and in this case we cannot get uniform control of the exponential decay parameter $\lambda(\sigma)$ from Theorem  \ref{thm:main-fixed-surface-tension}. 

Theorems \ref{thm:main-fixed-surface-tension} and \ref{thm:main-vanishing-surface-tension} also guarantee enough regularity to switch back to Eulerian coordinates. Consequently, the theorem tells us that the steady oscillating solution in \eqref{eqn:steady-oscillating-solution} remains asymptotically stable with and without surface tension, but that the rate of decay to equilibrium is faster with surface tension. 

Our third result establishes the vanishing surface tension limit for the problem \eqref{system:ns-flattened} in the same spirit as the result proved in \cite{tan2014zero}. 

\begin{thm}\label{thm:main-vanishing-surface-tension-limit}
Let $\Omega$ be given by \eqref{definition:Omega}, let $N\geq3$, and consider a decreasing sequence $\{\sigma_m\}_{m=0}^\infty\subset(0,1)$ such that $\sigma_m\to0$ as $m\to\infty$. Let $\kappa_0,\gamma_0\in(0,1)$ be as in Theorem \ref{thm:main-vanishing-surface-tension}, and assume that $0\leq\sum_{\ell=2}^{2N+2}A\omega^\ell\leq\gamma_0$. Suppose that for each $m\in\mathbb N$ we have initial data $(u_0^{(m)},\eta_0^{(m)})$ satisfying $ \norm{u_0}_{4N}^2 + \norm{\eta_0}_{4N+1/2}^2 + \sigma_m \norm{\nabla\eta}_{4N}^2<\kappa_0$ as well as the compatibility conditions of Theorem \ref{thm:local-existence}. Let $(u^{(m)},p^{(m)},\eta^{(m)})$ be the global solutions to \eqref{system:ns-flattened} associated to the data given by Theorem \ref{thm:main-vanishing-surface-tension}. Further assume that
\begin{equation}
    \text{$u_0^{(m)}\to u_0$ in $H^{4N}(\Omega)$, $\eta_0^{(m)}\to\eta_0$ in $H^{4N+1/2}(\Sigma)$, and $\sqrt{\sigma_m}\nabla\eta_0^{(m)}\to 0$ in $H^{4N}(\Sigma)$}
\end{equation}
as $m\to\infty$.

Then the following hold.
\begin{enumerate}
    \item The pair $(u_0,\eta_0)$ satisfy the compatibility conditions of Theorem \ref{thm:local-existence} with $\sigma=0$. 
    \item As $m\to\infty$, the triple $(u^{(m)},p^{(m)},\eta^{(m)})$ converges to $(u,p,\eta)$, where the latter triple is the unique solution to \eqref{system:ns-flattened} with $\sigma=0$ and initial data $(u_0,\eta_0)$. The convergence occurs in any space into which the space of triples $(u,p,\eta)$ obeying $\calG_{2N}^0(\infty)<\infty$ compactly embeds. 
\end{enumerate}
\end{thm}

\subsection{Discussion and plan of paper}

The strategy of the current paper is similar to that of \cite{guo2013almost,tice2018asymptotic}, which proved similar results for related problems with $f=0$.  As in these papers, the main focus of this paper is to establish a priori estimates for solutions to the PDE \eqref{system:ns-flattened}, which allow us to prove Theorems \ref{thm:main-fixed-surface-tension} and \ref{thm:main-vanishing-surface-tension} by standard arguments coupling these estimates with local existence results.  The scheme of a priori estimates developed in this paper is a variant of the nonlinear energy method employed in \cite{guo2013almost,tice2018asymptotic} and is designed to carefully track the dependence on the surface tension $\sigma$, the oscillation amplitude parameter $A$, and the oscillation frequency parameter $\omega$ in order to optimize the parameter regime in which we obtain the desired existence and stability theorem.  In the case with fixed surface tension $\sigma >0$ we obtain sufficient conditions for asymptotic stability of the form
\begin{equation}\label{stable_regime}
    A\omega^2+A\omega^3\lesssim 1,
\end{equation}
without bounds on $A$ or $\omega$ individually (see Figure \ref{fig:faraday-stable_regime}). Thus, the Faraday oscillation system can be stable for arbitrarily large $A$ or $\omega$, so long as the other parameter is sufficiently small for \eqref{stable_regime} to hold.  In the vanishing surface tension case, we obtain a similar result, although the trade-off is that more stringent constraints on $A$ and $\omega$ are necessary, namely 
\begin{equation}
A \sum_{\ell=2}^{2N+2}\omega^\ell \ls 1.
\end{equation}
We note that although our technique is capable of rigorously identifying a stable regime in the oscillation parameter space, it tells us nothing about the complement of this set.  The numerics for the linearized problem in  \cite{kumar1996linear} suggest that the complement indeed contains both stable and unstable components.

\begin{figure}[ht]
\centering
\begin{tikzpicture}
    \clip(-1,-1) rectangle (11,5); 
    \draw[->] (-0.5,0) -- (10,0) node[right] {$\omega$};
    \draw[->] (0,-0.5) -- (0,4.2) node[above] {$A$};
    \fill[xscale=5.0,yscale=0.5,domain=0.5:2,smooth,variable=\x,blue,opacity=0.3] 
    (0,8) -- (0.5,8) -- 
    plot ({\x},{3/(\x*\x+\x*\x*\x)})
    -- (2,0.25) -- (2,0) -- (0,0) -- cycle;
    \node[text width=10cm] at (10,2) 
    {unknown};
    \node[text width=10cm] at (5.5,0.5) 
    {asymptotic stability, \\exponential decay};
    \draw[xscale=5.0,yscale=0.5,domain=0.5:2,smooth,variable=\x,ultra thick,line width=1mm] plot ({\x},{3/(\x*\x+\x*\x*\x)});
    \node[text width=10cm] at (8,4) 
    {$A\omega^2+A\omega^3 = C$};
\end{tikzpicture}

\caption{Bounds on the stability regime with fixed $\sigma >0$.}
\label{fig:faraday-stable_regime}
\end{figure}
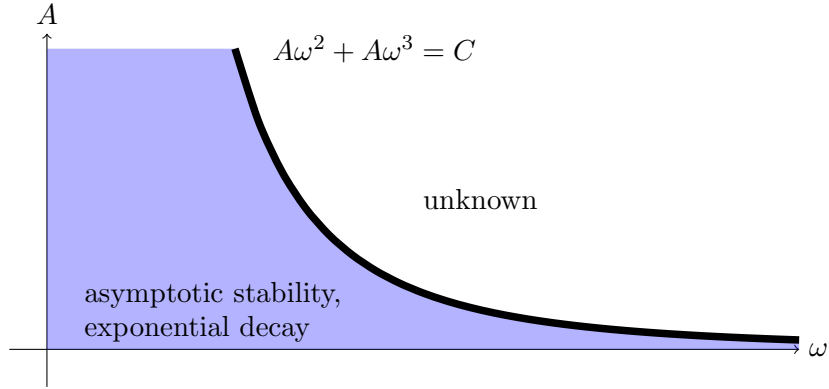
 
Our strategy for obtaining the a priori estimates is essentially the same nonlinear energy method as that of \cite{tice2018asymptotic}, so we refer the reader to the introduction of that paper for a detailed outline and opt for a terse summary here.  First, we obtain \emph{horizontal energy estimates} by applying horizontal and temporal derivatives to the problem and using the basic energy-dissipation structure.  It turns out to be convenient to do these estimates in two different forms depending on whether the derivatives involve only temporal derivatives or a mixture. These estimates are developed in Section \ref{section:evolution-ed}.

The next step in the nonlinear energy method is \emph{energy and dissipation enhancement}, where we employ various auxiliary estimates in order to gain control of more quantities in terms of those already controlled by the horizontal estimates.  The main tools are elliptic regularity for the Stokes problem and elliptic regularity for the capillary problem, both of which are recorded in Appendix \ref{section:appendix-b}.  The enhanced estimates are recorded in Section \ref{section:a-priori} and are predicated on the various estimates of the nonlinearities presented in Section \ref{section:nonlin-est}.

We combine the above estimates into a scheme of \emph{a priori estimates.}  In Section \ref{section:vanishing-surface-tension} we study the cases $\sigma =0$ and $\sigma \to 0$.  When coupled with the local existence theory, the a priori estimates allow us to complete the proofs of Theorems \ref{thm:main-vanishing-surface-tension} and \ref{thm:main-vanishing-surface-tension-limit}.  In Section \ref{section:fixed-surface-tension} we study the fixed surface tension problem and prove Theorem \ref{thm:main-fixed-surface-tension}.

\section{Evolution of the energy and dissipation}\label{section:evolution-ed}

In this section we record the energy-dissipation evolution equations for two linearized versions of the problem \eqref{system:ns-flattened}: the geometric form and the flattened form. We also record the forms of the nonlinear forcing terms that appear in the analysis of \eqref{system:ns-flattened}. 

\subsection{Geometric form}
Let $\Phi,\mathcal A,\mathcal N,J$, etc.\ be given in terms of $\eta$ as in Section \ref{section:reformulation}. We give the geometric linearization of \eqref{system:ns-flattened} { for $(v,q,\zeta)$}:
\begin{equation}\label{system:ns-lin-geometric}
\begin{cases}
    \partial_t v-\partial_t\hat\eta\tilde bK\partial_3 q+u\cdot\nabla_{\mathcal A}\zeta + \diverge_{\mathcal A}S_{\mathcal A}(v,q) = \Psi^1 & \text{in $\Omega$} \\
    \diverge_{\mathcal A}v = \Psi^2 & \text{in $\Omega$} \\
    \partial_t\zeta - v\cdot\mathcal N = \Psi^3 & \text{on $\Sigma$} \\
    S_{\mathcal A}(v,q)\mathcal N = \left(-\sigma\Delta\zeta + g\zeta + \Psi^5\right)\mathcal N + \Psi^4 & \text{on $\Sigma$} \\
    v = 0 & \text{on $\Sigma_b$}
\end{cases}
\end{equation}
Note that although we have written $\Psi^5$ in a form that suggests it will play the same nonlinear role as $\Psi^1,\dotsc,\Psi^4$, this term is actually linear.  We have left it in this general form for convenience in the derivation of the energy-dissipation equation.

\subsubsection{Energy-dissipation}
The next result records the energy-dissipation equation associated to the solutions of \eqref{system:ns-lin-geometric}.  
\begin{prop}[Geometric energy-dissipation]\label{prop:geom-ed}
Let $\eta$ and $u$ be given and satisfy
\begin{equation}\label{assumption:geometric-ed}
\begin{cases}
    \diverge_{\mathcal A}u = 0 & \text{in $\Omega$} \\
    \partial_t\eta = u\cdot\mathcal N & \text{on $\Sigma$}
\end{cases}.
\end{equation}
Suppose that $(v,q,\zeta)$ solve \eqref{system:ns-lin-geometric}, where $\Phi,\mathcal A, J$, etc.\ are determined by $\eta$ as before. Then,
\begin{equation}\label{equation:geometric-ed}
\begin{aligned}
    &\ddt\left[\int_\Omega\frac{\abs{v}^2J}2+\int_\Sigma\frac{\sigma\abs{\nabla{\eta}}^2}2 + \frac{g\abs{\eta}^2}2\right] + \int_\Omega\mu\frac{\abs{\mathbb D_{\mathcal A}{v}}^2J}2  \\
    &\hspace{5ex} = \int_\Omega J\left(v\cdot\Psi^1+q\Psi^2\right) 
 + \int_\Sigma \left(-\sigma\Delta\zeta + g\zeta \right)\Psi^3 -\Psi^4\cdot v - \Psi^5 v \cdot \mathcal{N}.
\end{aligned}
\end{equation}
\end{prop}

\begin{proof}
We take the dot product of the first equation in \eqref{system:ns-lin-geometric} with $v$, multiply by $J$, and integrate over $\Omega$ to see that
\begin{equation}\label{equation:ed-I-II}
    I+II=\int_\Omega \Psi^1\cdot vJ
\end{equation}
for
\begin{equation}
 I = \int_\Omega\partial_t v\cdot vJ - \partial_t\hat\eta\tilde b\partial_3 v\cdot v + (u\cdot\nabla_{\mathcal A}v)\cdot vJ \text{ and }
    II = \int_\Omega \diverge_{\mathcal A}S_{\mathcal A}(v,q)\cdot vJ.
\end{equation}
In order to integrate these terms by parts, we will utilize the geometric identity $\partial_k(J\mathcal A_{ik}) = 0$ (which is readily verified by direct computation) for each $i$. 

To handle the term $I$, we first compute
\begin{equation}
    I = \partial_t\int_\Omega\frac{\abs{v}^2J}2 + \int_\Omega -\frac{\abs{v}^2\partial_tJ}2-\partial_t\hat\eta\tilde b\partial_3\frac{\abs{v}^2}2 + u_j\partial_k\left(J\mathcal A_{jk}\frac{\abs{v}^2}2\right) =: I_1 + I_2.
\end{equation}
Since $\tilde b = (1+x_3/b)$, an integration by parts, an application of the boundary condition $v=0$ on $\Sigma_b$ reveals that
\begin{equation}
\begin{aligned}
    I_2 &= \int_\Omega -\frac{\abs{v}^2\partial_tJ}2-\partial_t\hat\eta\tilde b\partial_3\frac{\abs{v}^2}2 + u_j\partial_k\left(J\mathcal A_{jk}\frac{\abs{v}^2}2\right) \\
    &= \int_\Omega -\frac{\abs{v}^2\partial_tJ}2+\frac{\abs{v}^2}2\partial_3\left(\partial_t\hat\eta\tilde b\right) + u_j\partial_k\left(J\mathcal A_{jk}\frac{\abs{v}^2}2\right)+\int_\Sigma -\frac{\abs{v}^2}2\partial_t\hat\eta\tilde b \\
    &= \int_\Omega -\frac{\abs{v}^2}2\left(\frac{\pt\hat\eta}b+\pt_3\pt_t\hat\eta\tilde b\right)+\frac{\abs{v}^2}2\left(\pt_3\pt_t\hat\eta\tilde b + \frac{\pt_t\hat\eta}b\right) - \left(\partial_ku_j\right)\left(J\mathcal A_{jk}\frac{\abs{v}^2}2\right) \\
    &\hspace{5ex}+\int_\Sigma -\frac{\abs{v}^2}2\partial_t\hat\eta+u_jJ\mathcal A_{jk}\frac{\abs{v}^2}2(e_3\cdot e_k) \\
    &= \int_\Omega - J\frac{\abs{v}^2}2\diverge_{\mathcal A}u+\int_\Sigma -\frac{\abs{v}^2}2\partial_t\hat\eta+u_jJ\mathcal A_{jk}\frac{\abs{v}^2}2(e_3\cdot e_k).
\end{aligned}
\end{equation}
Now note that $J\mathcal A_{jk}(e_3\cdot e_k) = \mathcal N_j$ on $\Sigma$ and also we have that $u$ and $\eta$ satisfy \eqref{assumption:geometric-ed}, so the above becomes
\begin{equation}
    I_2 = \int_\Omega - J\frac{\abs{v}^2}2\diverge_{\mathcal A}u+\int_\Sigma \frac{\abs{v}^2}2\left(- \partial_t\eta+ u\cdot\mathcal N\right) = 0
\end{equation}
and hence
\begin{equation}\label{equation:ed-I}
    I = I_1 + I_2 = \partial_t\int_\Omega\frac{\abs{v}^2J}2
\end{equation}
so $I$ is purely just the transport of the quantity $\abs{v}^2J$ along the flow $u$. 

We begin our analysis of the term $II$ with a similar integration by parts, which reveals that \begin{equation}
\begin{aligned}
    II &= \int_\Omega\diverge_{\mathcal A}S_{\mathcal A}(v,q)\cdot vJ = \int_\Omega \mathcal A_{jk}\partial_k \left(S_{\mathcal A}(v,q)\right)_{ij}v_iJ = \int_\Omega v_iJ\mathcal A_{jk}\partial_k \left(S_{\mathcal A}(v,q)\right)_{ij} \\
    &= \int_\Omega -\partial_k\left(v_iJ\mathcal A_{jk}\right)\left(S_{\mathcal A}(v,q)\right)_{ij} + \int_\Sigma v_iJ\mathcal A_{jk}\left(S_{\mathcal A}(v,q)\right)_{ij}(e_3\cdot e_k) \\
    &= \int_\Omega -\left[\partial_k\left(v_iJ\mathcal A_{jk}\right)-v_i\partial_k\left(J\mathcal A_{jk}\right)\right]\left(S_{\mathcal A}(v,q)\right)_{ij} + \int_\Sigma v_iJ\mathcal A_{j3}\left(S_{\mathcal A}(v,q)\right)_{ij} \\
    &= \int_\Omega -J\mathcal A_{jk}\partial_k v_i\left(S_{\mathcal A}(v,q)\right)_{ij} + \int_\Sigma v_i\left(S_{\mathcal A}(v,q)\right)_{ij}\mathcal A_{j3}J \\
    &= \int_\Omega -J\left(\nabla_{\mathcal A}v\right)_{ij}\left(S_{\mathcal A}(v,q)\right)_{ij} + \int_\Sigma v_i\left(S_{\mathcal A}(v,q)\right)_{ij}\mathcal N_j \\
    &= \int_\Omega -JS_{\mathcal A}(v,q):\nabla_{\mathcal A}v + \int_\Sigma S_{\mathcal A}(v,q)\mathcal N\cdot v \\
    &= \int_\Omega -J\left(q\diverge_{\mathcal A}v-\frac{\mu\abs{\mathbb D_{\mathcal A}v}^2}2\right) + \int_\Sigma S_{\mathcal A}(v,q)\mathcal N\cdot v \\
    &= \int_\Omega -J\left(q\Psi^2-\frac{\mu\abs{\mathbb D_{\mathcal A}v}^2}2\right) + \int_\Sigma S_{\mathcal A}(v,q)\mathcal N\cdot v.
\end{aligned}
\end{equation}
Now using the third and fourth equations in \eqref{system:ns-lin-geometric}, we rewrite the integral on $\Sigma$ as
\begin{equation}
\begin{aligned}
    \int_\Sigma S_{\mathcal A}(v,q)\mathcal N\cdot v &= \int_\Sigma \left[\left(-\sigma\Delta\zeta + g\zeta + \Psi^5\right)\mathcal N + \Psi^4\right]\cdot v \\
    &= \int_\Sigma \left(-\sigma\Delta\zeta + g\zeta + \Psi^5\right)\mathcal N\cdot v + \int_\Sigma \Psi^4\cdot v \\
    &= \int_\Sigma \left(-\sigma\Delta\zeta + g\zeta \right)\left(\partial_t\zeta - \Psi^3\right) + \int_\Sigma \Psi^4\cdot v + \Psi^5 v \cdot \mathcal{N}\\
    &= \partial_t\left[\int_\Sigma\frac{\sigma\abs{\nabla\zeta}^2}2 + \frac{g\abs{\zeta}^2}2\right]  - \int_\Sigma \left(-\sigma\Delta\zeta + g\zeta \right)\Psi^3 - \Psi^4\cdot v - \Psi^5 v \cdot \mathcal{N}
\end{aligned}
\end{equation}
so on sum, we have
\begin{equation}\label{equation:ed-II}
\begin{aligned}
    II &= \int_\Omega -J\left(q\Psi^2-\frac{\mu\abs{\mathbb D_{\mathcal A}v}^2}2\right)+\ddt\left[\int_\Sigma\frac{\sigma\abs{\nabla\zeta}^2}2 + \frac{g\abs{\zeta}^2}2\right] \\
    &\qquad- \int_\Sigma \left(-\sigma\Delta\zeta + g\zeta \right)\Psi^3 - \Psi^4\cdot v - \Psi^5 v \cdot \mathcal{N}.
\end{aligned}
\end{equation}
Now to see that equation \eqref{equation:geometric-ed} holds, just plug \eqref{equation:ed-I} and \eqref{equation:ed-II} into \eqref{equation:ed-I-II} and rearrange.
\end{proof}

\subsubsection{Forcing terms}
We now record the form of the forcing terms that will appear in our analysis. Recall that this geometric form of the linearization is responsible for the highest order time derivatives $\pt_t^n$, so we build this into the notation by writing $F^{j,n}$ for the $j$th forcing term generated by applying $\pt_t^n$ to \eqref{system:ns-flattened}. 

Applying $\pt_t^n$ to the $i$th component of the first equation results in
\begin{equation}
\begin{aligned}
    \partial_t(\partial_t^n u_i) &+ \sum_{0\leq \ell\leq n}C_{\ell n }\left[-\partial_t^\ell\left(\partial_t\hat\eta\tilde b K\right)\partial_t^{n-\ell}\left(\partial_3 u_i\right) + \pt_t^\ell\left(u_j\calA_{jk}\right)\pt_t^{n-\ell}\left(\pt_ku_i\right)\right] \\
    &+\sum_{0\leq \ell\leq n}C_{\ell n}\left[\pt_t^\ell\calA_{jk}\pt_t^{n-\ell}\pt_k\left(S_{\calA}(u,p)\right)_{ij}\right] = 0.
\end{aligned}
\end{equation}
Then in the above, the first term as well as the terms in the summation corresponding to $\ell=0$ gives the left hand side of \eqref{system:ns-lin-geometric}, except for the last sum for which we have an extra term
\begin{equation}
\begin{aligned}
    \calA_{jk}\pt_t^n\pt_k\left(S_\calA(u,p)_{ij}\right) - \calA_{jk}\partial_k\left(S_\calA(\pt^n u,\pt_t^n p)_{ij}\right) &= \calA_{jk}\left(-\sum_{0\leq\ell\leq n}C_{\ell n}\mu\mathbb D_{\pt_t^\ell\calA}\pt_t^{n-\ell}u + \mu\mathbb D_{\calA}\pt_t^n u\right) \\
    &= -\calA_{jk}\sum_{0<\ell\leq n}C_{\ell n}\mu\mathbb D_{\pt_t^\ell\calA}\pt_t^{n-\ell}u.
\end{aligned}
\end{equation}
Thus,
\begin{equation}
\begin{aligned}\label{definition:F1}
    F^{1,n}_i &= \sum_{0< \ell\leq n}C_{\ell n }\left[\partial_t^\ell\left(\partial_t\hat\eta\tilde b K\right)\partial_t^{n-\ell}\left(\partial_3 u_i\right) - \pt_t^\ell\left(u_j\calA_{jk}\right)\pt_t^{n-\ell}\left(\pt_ku_i\right)\right] \\
    &+\sum_{0< \ell\leq n}C_{\ell n}\left[-\pt_t^\ell\calA_{jk}\pt_t^{n-\ell}\pt_k\left(S_{\calA}(u,p)\right)_{ij}+\calA_{jk}\mu\mathbb D_{\pt^\ell\calA}\pt_t^{n-\ell}u\right].
\end{aligned}
\end{equation}
Differentiating the second equation gives
\begin{equation}
    \sum_{0\leq\ell\leq n}C_{\ell n}\partial_t^\ell\mathcal A_{jk}\pt_t^{n-\ell}\pt_ku_j = 0
\end{equation}
so taking all but the $\ell=0$ terms gives
\begin{equation}\label{definition:F2}
    F^{2,n} = \sum_{0<\ell\leq n}C_{\ell n}\partial_t^\ell\mathcal A_{jk}\pt_t^{n-\ell}\pt_ku_j.
\end{equation}
Differentiating the third equation gives
\begin{equation}
    \partial_t(\partial_t^n\eta) - \sum_{0\leq \ell\leq n}C_{\ell n}\pt_t^\ell u\cdot \pt_t^{n-\ell}\calN = 0
\end{equation}
so
\begin{equation}\label{definition:F3}
    F^{3,n} = \sum_{0< \ell\leq n}C_{\ell n}\pt_t^\ell u\cdot \pt_t^{n-\ell}\calN.
\end{equation}
Finally, differentiating the $i$th component of the fourth equation gives
\begin{equation}
\begin{aligned}
    \sum_{0\leq\ell\leq n}C_{\ell n}\pt_t^{n-\ell}\left(S_{\mathcal A}(u,p)\right)_{ij}\pt_t^\ell\mathcal N_j &= \sum_{0\leq\ell\leq n}C_{\ell n}\pt_t^{n-\ell}\left(-\sigma\mathfrak H(\eta) + \left(g+A\omega^2f''(\omega t)\right)\eta\right)\partial_t^\ell\mathcal N_i \\
    &= \sum_{0\leq\ell\leq n}C_{\ell n}\left(-\sigma\pt_t^{n-\ell}\mathfrak H(\eta) + g\pt_t^{n-\ell}\eta + \pt_t^{n-\ell}\left(A\omega^2f''(\omega t)\eta\right)\right)\partial_t^\ell\mathcal N_i
\end{aligned}
\end{equation}
so taking away the $\ell=0$ terms and handling $\pt_t^n\left(S_{\mathcal A}(u,p)\right)_{ij}$ as before as well as handling the $\Delta\eta$ term, we get
\begin{equation}\label{definition:F4}
\begin{aligned}
    F^{4,n}_i &= \sum_{0<\ell\leq n}C_{\ell n}\left[-\pt_t^{n-\ell}\left(S_{\mathcal A}(u,p)\right)_{ij}\pt_t^\ell\mathcal N_j + \left(\mu\mathbb D_{\pt^\ell\calA}\pt_t^{n-\ell}u\right)_{ij}\calN_j\right] \\
    &+ \sum_{0<\ell\leq n}C_{\ell n}\left(-\sigma\pt_t^{n-\ell}\mathfrak H(\eta) + g\pt_t^{n-\ell}\eta + \pt_t^{n-\ell}\left(A\omega^2f''(\omega t)\eta\right)\right)\partial_t^{n-\ell}\mathcal N_i +\left(-\sigma\pt_t^n\left(\mathfrak H(\eta)-\Delta\eta\right)\right)\mathcal N_i
\end{aligned}
\end{equation}
and
\begin{equation}
    F^{5,n} = \pt_t^n\left(A\omega^2f''(\omega t)\eta\right).
\end{equation}
Note in particular that $F^{5,n}$ is a linear term, different in form from the other nonlinear forcing terms.

\subsection{Flattened form}
It will also be useful for us to have a linearized version of \eqref{system:ns-flattened} with constant coefficients. This version is as follows:
\begin{equation}\label{equation:flat-I}
\begin{cases}
    \partial_t v + \diverge S(v,q) = \Theta^1 & \text{in $\Omega$} \\
    \diverge v = \Theta^2 & \text{in $\Omega$} \\
    \partial_t\zeta = v_3 + \Theta^3 & \text{on $\Sigma$} \\
    S(v,q)e_3 = \left(-\sigma\Delta\zeta + g\zeta+\Theta^5\right)e_3 + \Theta^4 & \text{on $\Sigma$} \\
    v = 0 & \text{on $\Sigma_b$}
\end{cases}.
\end{equation}
Again we note that $\Theta^5$ is really a linear term, of a different nature than the other forcing terms.  We have again elected to write it in this general form for ease in the deriving the energy-dissipation equation.

\subsubsection{Energy-dissipation}

\begin{prop}[Flattened energy-dissipation]\label{prop:flattened-ed}
    Suppose $(v,q,\zeta)$ solve \eqref{equation:flat-I}. Then
    \begin{equation}
    \begin{aligned}
        &\ddt\left[\int_\Omega\frac{\abs v^2}2 + \int_\Sigma\frac{\sigma\abs{\nabla\zeta}^2}2 + {\frac{g\abs{\zeta}^2}2}\right] + \frac{\mu}2\int_\Omega\abs{\mathbb Dv}^2 \\ 
        & = \int_\Omega v\cdot\Theta^1 + q\Theta^2 + \int_\Sigma\left(-\sigma\Delta\zeta + g\zeta \right)\Theta^3 - \Theta^4\cdot v - \Theta^5 v_3.
    \end{aligned}
    \end{equation}
\end{prop}

\begin{proof}
    We dot the first equation of \eqref{equation:flat-I} with $v$ and integrate over $\Omega$ to see that
    \begin{equation}
    I + II = \int_\Omega v\cdot\Theta^1
    \end{equation} 
    where 
    \begin{equation}
    I \coloneqq \int_\Omega v\partial_t v = \partial_t\int_\Omega\frac{|v|^2}2 \text{ and } II \coloneqq \int_\Omega v\cdot \diverge(qI - \mu \mathbb{D} v).
    \end{equation}
    To deal with $II$ we compute
    \begin{equation}
 \int_\Omega v\cdot\diverge(qI - \mu\mathbb D v) =\int_\Omega -(qI - \mu\mathbb Dv) : \nabla v + \int_\Sigma (qI - \mu\mathbb D v)e_3\cdot v \coloneqq II_1 + II_2.
    \end{equation}
    A simple computation gives 
    \begin{equation}
    II_1 = \int_\Omega \mu\mathbb Dv : \nabla v - (qI):\nabla v = \frac \mu 2\int_\Omega \abs{\mathbb Dv}^2 - \int_\Omega q \diverge(v) =  \int_\Omega \frac\mu 2 \abs{\mathbb Dv}^2 - q\Theta^2.
    \end{equation}
    Now
    \begin{equation}
    \begin{aligned}
    II_2 &= \int_\Sigma v\cdot\left[(-\sigma\Delta\zeta + g\zeta+\Theta^5)e_3 + \Theta^4\right]\\
    &= \int_\Sigma \left(-\sigma\Delta\zeta + g\zeta+\Theta^5\right)v_3 + \Theta^4\cdot v
    \\
    &= \int_\Sigma \left(-\sigma\Delta\zeta + g\zeta \right)(\partial_t\zeta - \Theta^3) + \Theta^4\cdot v + \Theta^5 v_3 \\
    &=\partial_t\left[\int_\Sigma\frac{\sigma\abs{\nabla\zeta}^2}2 + {\frac{g\abs{\zeta}^2}2}\right] +\int_\Sigma  - (-\sigma\Delta\zeta + g\zeta )\Theta^3 + \Theta^4\cdot v +\Theta^5 v_3. 
    \end{aligned}
    \end{equation}
    Thus, in sum, we have
    \begin{equation}
    \begin{aligned}
        II &= \partial_t\left[\int_\Sigma\frac{\sigma\abs{\nabla\zeta}^2}2 + {\frac{g\abs{\zeta}^2}2}\right] + \frac{\mu}2\int_\Omega\abs{\mathbb Dv}^2 - \int_\Omega q\Theta^2 + \int_\Sigma - (-\sigma\Delta\zeta + g\zeta)\Theta^3 + \Theta^4\cdot v +\Theta^5 v_3.
    \end{aligned}
    \end{equation}
    The result follows by summing and regrouping.
\end{proof}

\subsubsection{Forcing terms}

The forcing terms come from rearranging the equation to get the terms we want -- we then designate everything else as forcing terms. Note that we will take derivatives of the full nonlinear equations in \eqref{system:ns-flattened}, but to get the corresponding forcing terms, we may just take derivatives of the forcing terms here since we constructed our linearization to have constant coefficients. To get the first forcing term, remark that the first equation in \eqref{system:ns-flattened} can be rewritten as
\begin{equation}
\begin{aligned}
    \partial_t u + \diverge S(u,p) &= \partial_t\hat\eta\tilde bK\partial_3 u - u\cdot\nabla_{\mathcal A}u + \left(\diverge S(u,p) - \diverge_{\mathcal A}S_{\mathcal A}(u,p)\right)\\
    &= \partial_t\hat\eta\tilde bK\partial_3 u - u\cdot\nabla_{\mathcal A}u -\diverge\mathbb D_{I-\mathcal A}u-\diverge_{\mathcal A-I}\left(pI-\mathbb D_{\mathcal A}u\right),
\end{aligned}
\end{equation}
and so 
\begin{equation}\label{definition:G1}
G^1 = \partial_t\hat\eta\tilde bK\partial_3 u - u\cdot\nabla_{\mathcal A}u -\diverge\mathbb D_{I-\mathcal A}u-\diverge_{\mathcal A-I}\left(pI-\mathbb D_{\mathcal A}u\right).
\end{equation}
The second term is 
\begin{equation}\label{definition:G2}
G^2 = \diverge_{I-\mathcal A}u;
\end{equation}
this is a result of simply adding and subtracting the two different types of divergence. To handle the third equation, rewrite as $\partial_t\eta = u\cdot e_3 + u\cdot(\mathcal N - e_3),$ and so 
\begin{equation}\label{definition:G3}
G^3 = u\cdot (\mathcal N - e_3).
\end{equation}
Finally, we similarly write the fourth nonlinear term as
\begin{equation}\label{definition:G4}
\begin{aligned}
G^4 &=(pI - \mu\mathbb D u)(e_3-\mathcal N) + (\mu\mathbb D_{\mathcal A - I}u)\mathcal N + \left(g\eta+A\omega^2f''(\omega t)\eta\right)(e_3-\mathcal N) \\
&\qquad- \left(-\sigma\mathfrak H(\eta)\right)\left(e_3-\mathcal N\right) - \left(-\sigma\left(\Delta\eta-\mathfrak H(\eta)\right)\right)e_3
\end{aligned}
\end{equation}
and the fifth error term, which is linear, is written as
\begin{equation}
    G^5 = A\omega^2f''(\omega t)\eta.
\end{equation}

\section{Estimates of the nonlinearities and other error terms}\label{section:nonlin-est}
In this section we develop the estimates of the nonlinearities as well as other error terms needed to close our scheme of a priori estimates.

\subsection{$L^\infty$ estimates}
The next result establishes some key $L^\infty$ bounds that will be used repeatedly throughout the paper.

\begin{prop}\label{prop:L-infinity}
There exists a universal constant $\delta\in(0,1)$ such that if $\norm{\eta}_{5/2}^2\leq\delta$, then the following bounds hold. 
\begin{enumerate}
    \item We have that
    \begin{equation}\label{inequality:L-infinity-1}
        \norm{J-1}_{L^\infty}^2 + \norm{\mathcal N-e_3}_{L^\infty}^2 + \norm{A}_{L^\infty}^2 + \norm{B}_{L^\infty}^2\leq\frac12
    \end{equation}
    and
    \begin{equation}\label{inequality:L-infinity-2}
        \norm{K}_{L^\infty}^2 + \norm{\mathcal A}_{L^\infty}^2\lesssim 1.
    \end{equation}
    \item The mapping given by \eqref{definition:Phi} is a diffeomorphism from $\Omega$ to $\Omega(t)$. 
    \item For all $v\in H^1(\Omega)$ such that $v=0$ on $\Sigma_b$ we have the estimate
    \begin{equation}
        \int_\Omega\abs{\mathbb Dv}^2\leq \int_\Omega J\abs{\mathbb D_{\mathcal A}v}^2 + C\left(\norm{\mathcal A-I}_{L^\infty}+\norm{J-1}_{L^\infty}\right)\int_\Omega\abs{\mathbb Dv}^2
    \end{equation}
    for a universal constant $C>0$. 
\end{enumerate}
\end{prop}
\begin{proof}
Recall that
\begin{equation}
    J-1 = \frac{\hat\eta}b + \partial_3\hat\eta\tilde b,\qquad\mathcal N-e_3 = (-\partial_1\eta,-\partial_2\eta,0),\qquad A=\partial_1\hat\eta\tilde b,\qquad B=\partial_2\hat\eta\tilde b.
\end{equation}
Thus, the left hand side of \eqref{inequality:L-infinity-1} can be bounded above, via Sobolev embedding $H^3(\Omega)\hookrightarrow C^1(\Omega)$, by $\norm{\hat\eta}_3$. This is in turn bounded by $\norm{\eta}_{5/2}$ by Lemma \ref{lem:poisson-L-infinity}. Then \eqref{inequality:L-infinity-2} holds by the definitions of $K$ and $\mathcal A$ and \eqref{inequality:L-infinity-1}.  To see the second item note that $\Psi = I + e_3 \hat{\eta} \tilde{b}$, which means that if $\norm{\hat{\eta}}_{C^1}$ is sufficiently small, then $\Psi$ is a bijection with positive Jacobian $J$.  In this case $\Psi$ is a diffeomorphism thanks to the inverse function theorem.  For the third item, first write
\begin{equation}
\begin{aligned}
    \abs{\mathbb Dv}^2 &= J\abs{\mathbb D_{\mathcal A}v}^2 - (J-1)\abs{\mathbb Dv}^2 - J\left(\abs{\mathbb D_{\mathcal A}v}^2 - \abs{\mathbb Dv}^2\right) \\
    &= J\abs{\mathbb D_{\mathcal A}v}^2 - (J-1)\abs{\mathbb Dv}^2 - J\left(\mathbb D_{\mathcal A}v + \mathbb Dv\right):\left(\mathbb D_{\mathcal A}v - \mathbb Dv\right) \\
    &=: I + II + III.
\end{aligned}
\end{equation}
Since the $I$ and $II$ terms are already in place, we just need to bound $III$. To do this, compute
\begin{equation}
    \left(\mathbb D_{\mathcal A}v\pm\mathbb Dv\right)_{ij} = \left(\mathcal A\pm I\right)_{ik}\partial_k v_j + \left(\mathcal A\pm I\right)_{jk}\partial_k v_i
\end{equation}
and so
\begin{equation}
    III = - J\left(\mathbb D_{\mathcal A}v + \mathbb Dv\right):\left(\mathbb D_{\mathcal A}v - \mathbb Dv\right)\leq \norm{J}_{L^\infty}\norm{\mathcal A+I}_{L^\infty}\norm{\mathcal A-I}_{L^\infty}\abs{\mathbb Dv}^2.
\end{equation}
The $L^\infty$ norms can be bounded by universal constants by \eqref{inequality:L-infinity-1} and \eqref{inequality:L-infinity-2}, so we conclude.
\end{proof}

\subsection{Estimates of the $F$ forcing terms}\label{subsection:est-F-nonlin}
We now present the estimates of the $F$ forcing terms that appear in the geometric form of the equations \eqref{system:ns-lin-geometric}. Estimates of the same general form are now well-known in the literature \cite{guo2013almost, tan2014zero, jang2016compressible, tice2018asymptotic}, so we will focus primarily on the terms that have not appeared before.

\begin{thm}\label{thm:bound-nonlin-F}
Let $F^{j,n}$ be defined by \eqref{definition:F1}, \eqref{definition:F2}, \eqref{definition:F3}, \eqref{definition:F4}. Assume that $\mathcal E\leq\delta$ for the universal $\delta\in(0,1)$ given by Proposition \ref{prop:L-infinity}, and further suppose that $\sum_{\ell=2}^{n+1}A\omega^\ell\lesssim 1$. Then there exists a polynomial $P$ with nonnegative universal coefficients such that 
\begin{equation}
\norm{F^{1,n}}_0^2 + \norm{F^{2,n}}_0^2 + \norm{\partial_t(JF^{2,n})}_0^2 +  \norm{F^{3,n}}_0^2 + \norm{F^{4,n}}_0^2\lesssim P(\sigma)\mathcal E_n^0\mathcal D_n^\sigma
\end{equation}
and
\begin{equation}
\norm{F^{2,n}}_0^2\lesssim \left(\calE_n^0\right)^2.
\end{equation}
\end{thm}

\begin{proof}
The estimates 
\begin{equation}
\norm{F^{1,n}}_0^2 + \norm{F^{2,n}}_0^2 + \norm{\partial_t(JF^{2,n})}_0^2 +  \norm{F^{3,n}}_0^2   \lesssim P(\sigma)\mathcal E_n^0\mathcal D_n^\sigma, \;\;
\norm{F^{2,n}}_0^2\lesssim \left(\calE_n^0\right)^2
\end{equation}
and 
\begin{equation}
 \norm{F^{4,n} -   \sum_{0<\ell\leq n}C_{\ell n}   \pt_t^{n-\ell}\left(A\omega^2f''(\omega t)\eta\right) \partial_t^{n-\ell}\mathcal N_i   }_0^2\lesssim P(\sigma)\mathcal E_n^0\mathcal D_n^\sigma
\end{equation}
are proved in Theorem 4.4 of \cite{tice2018asymptotic}.  To conclude we then use the bound $\sum_{\ell=2}^{n+1}A\omega^\ell\lesssim 1$ together with the Sobolev embeddings on $\Sigma$ to estimate
\begin{equation}
\begin{aligned}
    &\norm{  \pt_t^{n-\ell}\left(A\omega^2f''(\omega t)\eta\right) \partial_t^\ell\mathcal N_i}_0^2 \\
    &\hspace{10ex}\lesssim  P(\sigma) \left( \sum_{\ell=2}^{n+1}A\omega^\ell  \right)^2  \left(\sum_{m=0}^{n-\ell} \norm{\pt_t^m\eta}_2^2\right)\norm{\partial_t^\ell\nabla\eta}_0^2 \lesssim P(\sigma)\calE_n^0\calD_n^\sigma,
\end{aligned}
\end{equation}
and then we sum over $0 < \ell \le n$ to arrive at the desired final estimate.

\end{proof}

\subsection{Estimates of the $G$ forcing terms}
We now present the estimates for the $G^i$ nonlinearities. Define
\begin{equation}\label{eqn:define-Y}
\begin{aligned}
    \mathcal Y_n &\coloneqq \sum_{j=0}^{n-1}\norm{\pt_t^j G^1}_{2n-2j-1}^2 + \norm{\pt_t^j G^2}_{2n-2j}^2 + \norm{\pt_t^j G^4}_{H^{2n-2j-1/2}(\Sigma)}^2 + \sum_{j=2}^n\norm{\pt_t^j G^3}_{H^{2n-2j+1/2}(\Sigma)}^2 \\
    &\hspace{10ex}+\norm{G^3}_{H^{2n-1}(\Sigma)}^2 + \norm{\pt_t G^3}_{H^{2n-2}(\Sigma)}^2+\sigma^2\left(\norm{G^3}_{H^{2n+1/2}(\Sigma)}^2 + \norm{\pt_t G^3}_{H^{2n-3/2}(\Sigma)}^2\right).
\end{aligned}
\end{equation}
and
\begin{equation}\label{eqn:define-W}
    \mathcal W_n \coloneqq \sum_{j=0}^{n-1}\norm{\partial_t^jG^1}_{2n-2j-2}^2 + \norm{\pt_t^j G^2}_{2n-2j-1}^2+\norm{\pt_t^j G^3}_{H^{2n-2j-1/2}(\Sigma)}^2 + \norm{\pt_t^j G^4}_{H^{2n-2j-3/2}(\Sigma)}^2.
\end{equation}
These nonlinearities are the ones generated by elliptic regularity estimates.

\begin{thm}\label{thm:bound-nonlin-G}
Let $G^i$ for $i=1,\dotsc,4$ be defined by \eqref{definition:G1}, \eqref{definition:G2}, \eqref{definition:G3}, and \eqref{definition:G4}. Assume that $\mathcal E\leq\delta$ for the universal $\delta\in(0,1)$ given by Proposition \ref{prop:L-infinity}, and suppose that  $\sum_{\ell=2}^{n+1}A\omega^\ell\lesssim 1$.  Then there exists a polynomial $P$ with nonnegative universal coefficients such that
\begin{equation}\label{est:G-nonlin-1}
    \calY_n \lesssim P(\sigma)\left(\calE_n^0\calD_n^\sigma + \calK\calF_n\right)
\end{equation}
and
\begin{equation}\label{est:G-nonlin-2}
    \calW_n \lesssim P(\sigma)\left(\calE_n^0\calE_n^\sigma + \calK\calF_n\right).
\end{equation}
Furthermore, in the case that $n = 1$ and $\sigma > 0$ is fixed, we have
\begin{equation}\label{est:G-nonlin-base-w}
    \calW_1 \lesssim P(\sigma)(\calE^\sigma_1)^2.
\end{equation}
and
\begin{equation}\label{est:G-nonlin-base-y}
    \norm{G^1}_1^2 + \norm{G^2}_2^2 + \norm{G^3}_{5/2}^2 + \norm{\pt_tG^3}_{1/2}^2 + \norm{G^4}_1^2 \lesssim P(\sigma) \calE_1^\sigma\calD_1^\sigma.
\end{equation}
Note the above is just $\calY_1$ after considering $\sigma$ as a fixed constant, with $\norm{G^4}_1$ replacing the term  $\norm{G^4}_{3/2}$.
\end{thm}

\begin{proof}

The estimates \eqref{est:G-nonlin-1} and \eqref{est:G-nonlin-2} but with $G^4$ replaced by $G^4 - (A\omega^2f''(\omega t)\eta)(e_3-\mathcal N)$ are proved in Theorem 4.2 of \cite{tice2018asymptotic}, so to complete the proof of these estimates it suffices to show that 
\begin{equation} 
 \sum_{j=0}^{n-1} \norm{\dt^j \left((A\omega^2f''(\omega t)\eta)(e_3-\mathcal N) \right)}_{H^{2n-2j-1/2}(\Sigma)}^2 \lesssim P(\sigma)\left(\calE_n^0\calD_n^\sigma + \calK\calF_n\right)
\end{equation}
and 
\begin{equation} 
 \sum_{j=0}^{n-1} \norm{\dt^j \left((A\omega^2f''(\omega t)\eta)(e_3-\mathcal N) \right)}_{H^{2n-2j-3/2}(\Sigma)}^2 \lesssim P(\sigma)\left(\calE_n^0\calE_n^\sigma + \calK\calF_n\right).
\end{equation}
These follow easily from the Leibniz rule and the product estimates of Theorem \ref{est:prod_est}, together with the hypothesis that $\sum_{\ell=2}^{n+1}A\omega^\ell\lesssim 1$.

The bounds 
\begin{equation}
   \calW_1 \lesssim P(\sigma)(\calE^\sigma_1)^2
\end{equation}
and 
\begin{equation}
     \norm{G^1}_1^2 + \norm{G^2}_2^2 + \norm{G^3}_{5/2}^2 + \norm{\pt_tG^3}_{1/2}^2 + \norm{G^4}_0^2 \lesssim P(\sigma) \calE_1^\sigma\calD_1^\sigma
\end{equation}
follow from similar arguments.  To complete the proof of \eqref{est:G-nonlin-base-y} it remains only bound 
\begin{equation}\label{thm:bound-nonlin-G_0}
 \norm{\nabla G^4}_0^2 \lesssim P(\sigma) \calE_1^\sigma\calD_1^\sigma.
\end{equation}

To prove \eqref{thm:bound-nonlin-G_0} we first recall that $G^4$ can be written as the sum of five terms:
\begin{equation}
\begin{aligned}
    G^4 &=(pI - \mu\mathbb D u)(e_3-\mathcal N) + (\mu\mathbb D_{\mathcal A - I}u)\mathcal N + \left(g\eta+A\omega^2f''(\omega t)\eta\right)(e_3-\mathcal N) \\
    &\qquad- \left(-\sigma\mathfrak H(\eta)\right)\left(e_3-\mathcal N\right) - \left(-\sigma\left(\Delta\eta-\mathfrak H(\eta)\right)\right)e_3.
\end{aligned}
\end{equation}
We will handle each in turn.  For the first term we estimate 
\begin{equation}
    \begin{aligned}
        \norm{\nabla \left( (e_3 - \mathcal N)(pI - \mu\mathbb D u)\right)}_0^2 &\lesssim \sum_{\abs\beta + \abs\gamma =   1}\norm{\partial^\beta(e_3 - \mathcal N)\partial^\gamma(pI - \mu\mathbb D u)}_0^2\\ 
        &\lesssim \norm{\eta}_2^2 \left(\norm{p}_2^2 + \norm{u}_3^2\right)\lesssim\calE^\sigma_1\calD^\sigma_1.
    \end{aligned}
\end{equation}
We may argue similarly for the second term to see that 
\begin{equation}
\norm{\nabla \left(  \mu\mathbb D_{\mathcal A - I}u \mathcal N  \right)}_0^2 \lesssim \norm{\eta}_2^2 \left(\norm{p}_2^2 + \norm{u}_3^2\right)\lesssim\calE^\sigma_1\calD^\sigma_1
\end{equation}
For the third term we use the hypothesis $A \omega^2 \ls 1$ to estimate 
\begin{equation}
 \norm{\nabla \left(  (g\eta+A\omega^2f''(\omega t)\eta)(e_3-\mathcal N)   \right)}_0^2 \ls   \norm{\eta}_{2}^2 \norm{\eta}_{3/2}^2  \lesssim\calE^\sigma_1\calD^\sigma_1.
\end{equation}
For the fourth and fifth terms we expand 
\begin{equation}
    \mathfrak H(\eta) = \left(1+\abs{\nabla\eta}^2\right)^{-1/2}\Delta\eta - \left(1+\abs{\nabla\eta}^2\right)^{-3/2}  D^2\eta \nabla \eta \cdot \nabla \eta
\end{equation}
and 
\begin{equation}
 \frac{1}{\sqrt{1+\abs{\nabla \eta}^2}} -1 = - \frac{\abs{\nab \eta}^2}{\sqrt{1+\abs{\nabla \eta}^2}(1+\sqrt{1+\abs{\nabla \eta}^2})}
\end{equation}
in order to arrive at the estimate
\begin{equation}
 \norm{\nab \left( \sigma\mathfrak H(\eta)(e_3-\mathcal N) \right)}_0^2 +  \norm{\nab \left( \sigma (\Delta \eta -  \mathfrak H(\eta) \right)}_0^2  \ls  \norm{\eta}_2^2  \sigma^2  \norm{\eta}_{7/2}^2  \lesssim\calE^\sigma_1\calD^\sigma_1.
\end{equation}
Combining these, we deduce that \eqref{thm:bound-nonlin-G_0} holds, which completes the proof of all of the stated estimates.

\end{proof}

\subsection{Estimates on auxiliary terms}

Our next result provides some bounds for nonlinearities appearing in integrals.

\begin{prop}\label{prop:bound-G3}
Let $\alpha\in\mathbb N^2$ with $\abs\alpha = 2n$. Assume that $\mathcal E_n^\sigma\leq\delta$ for the universal $\delta\in(0,1)$ given by Proposition \ref{prop:L-infinity}. Then there exists a polynomial with nonnegative universal coefficients such that
\begin{equation}
    \abs{\int_\Sigma \partial^\al\eta\partial^\al G^3}\lesssim \sqrt{\calE_n^0}\calD_n^0 + \sqrt{\calD_n^0\calK\calF_n}
\end{equation}
and
\begin{equation}
    \abs{\sigma\int_\Sigma \Delta\partial^\al\eta\partial^\al G^3}\lesssim P(\sigma)\left(\sqrt{\calE_n^0\calD_n^0\calD_n^\sigma} + \sqrt{\calD_n^\sigma\calK\calF_n}\right).
\end{equation}
Moreover, when $n=1$ and $\sigma > 0$, we can improve the estimate above to
\begin{equation}
    \abs{\sigma\int_\Sigma \Delta \pt^\al \eta \pt^\al G^3} \lesssim \frac{1+\sqrt{\sigma}}{\sigma}  \sqrt{\calE^\sigma_1}\calD_1^\sigma.
\end{equation}
\end{prop}
\begin{proof}
The first two estimates are proved in  Lemma 3.5 of \cite{jang2016compressible}.

For the $n=1$ and $\sigma>0$ case, we first estimate
\begin{equation}
    \abs{\int_\Sigma \sigma\Delta\partial^\al\eta\partial^\al G^3} \lesssim \sigma\norm{\Delta\partial^\al\eta}_{-1/2}\norm{\partial^\al G^3}_{1/2}\lesssim  \sigma \norm{\eta}_{7/2}\norm{\partial^\al G^3}_{1/2}\lesssim \sqrt{\mathcal D_1^\sigma }\norm{\partial^\al G^3}_{1/2}.
\end{equation}
To conclude we use the definition of $G^3$ to bound  
\begin{equation}
 \norm{\partial^\alpha G^3}_{1/2} \ls \norm{\eta}_{5/2} \norm{u}_3 + \norm{u}_2  \norm{\eta}_{7/2} \ls \frac{1+\sqrt{\sigma}}{\sigma} \sqrt{\mathcal{E}^\sigma_1 \mathcal{D}^\sigma_1}.
\end{equation}

\end{proof}

We define the following auxiliary term which appear in later sections.
\begin{equation}\label{eqn:def-H}
    \calH_n \coloneqq \int_\Omega-\partial_t^{n-1}pF^{2,n}J + \frac12 |\pt_t^n u|^2(J-1) .
\end{equation}
The next result provides estimates for this term.

\begin{prop}\label{prop:bound-H}
Let $\calH_n$ be defined as in \eqref{eqn:def-H}, and assume $\calE_n^\sigma \leq \delta$ for the universal $\delta \in (0,1)$ given by Proposition \ref{prop:L-infinity}. Furthermore, suppose that $\sum_{\ell = 2}^{n+1}A\omega^\ell\lesssim 1$. Then
\begin{equation}
    |\calH_n| \lesssim (\calE_n^0)^{3/2}.
\end{equation}
\end{prop}
\begin{proof}
We can bound
\begin{equation}
    |\calH_n|\leq \norm{\partial_t^{n-1}p}_0 \norm{F^{2,n}}_0 \norm{J}_{L^\infty} + \frac12 \norm{J-1}_{L^\infty}\norm{\pt^n_t u}_0^2 .
\end{equation}
Then we use Proposition \ref{prop:L-infinity} and Theorem \ref{thm:bound-nonlin-F} to estimate
\begin{equation}
\norm{F^{2,n}}_0\norm{J}_{L^\infty} \lesssim \calE_n^0.
\end{equation}
Using the Sobolev embedding $H^3(\Omega) \hookrightarrow C^1(\Omega)$,
\begin{equation}
    \norm{J-1}_{L^\infty} \lesssim \norm{\hat\eta}_{C^1} \lesssim \norm{\hat\eta}_{H^3} \lesssim \norm{\eta}_{5/2}\lesssim \sqrt{\calE_n^0}.
\end{equation}
Therefore
\begin{equation}
    \abs{\calH_n} \lesssim \sqrt{\calE_n^0}\parens{\norm{\partial_t^{n-1}p}_0\sqrt {\calE_n^0}+\norm{\pt_t^n u}^2_0} \lesssim (\calE_n^0)^{3/2},
\end{equation}
as desired.
\end{proof}

\section{General a priori estimates}\label{section:a-priori}
The purpose of this section is to present a priori estimates that are general in the sense that they are valid for both the problem with and without surface tension. The general estimates presented here will be specially adapted later to each problem to prove different sorts of results.

\subsection{Energy-dissipation evolution estimates}
Let $\al\in\N^{1+2}$, and write
\begin{equation}\label{def:ed-general}
\begin{aligned}
    \ov \calE_\al^\sigma &= \int_\Omega \frac12 \abs{\pt^\al u}^2 + \int_\Sigma\frac12 \abs{g\pt^\al\eta}^2 + \frac\sigma2\abs{\nabla \pt^\al \eta}^2 \\
    \ov \calD_\al &= \int_\Omega \frac12 \abs{\mathbb D \pt^\al u}^2
\end{aligned}
\end{equation}
for the part of the energy and dissipation responsible for the $\alpha$ derivatives.

Our first result derives energy-dissipation estimates for the time derivative component of the energy and dissipation functionals. 
\begin{thm}\label{thm:nonlin-ed-time}
Assume that $\calE_n^\sigma\leq\delta$ for the universal $\delta\in(0,1)$ given by Proposition \ref{prop:L-infinity}.  Suppose further that  $\sum_{\ell = 2}^{n+1}A\omega^\ell\lesssim 1$.  Let $\alpha\in\N^{1+2}$ be given by $\alpha = (n,0,0)$, i.e. $\partial^\alpha = \partial_t^n$.  Then for $\bar\calE_\alpha^\sigma$ and $\calD_\alpha$ given by \eqref{def:ed-general}, there exists a polynomial $P$ with nonnegative universal constants such that we have the estimate
\begin{equation}
    \frac{d}{dt}(\ov{\calE_\alpha^\sigma} + \calH_n) + \ov\calD_\alpha\lesssim \left(\sum_{\ell = 2}^{n+2}A\omega^\ell\right)\calD_n^0 + P(\sigma)\sqrt{\calE_n^0}\calD_n^\sigma.
\end{equation}
\end{thm}

\begin{proof}
    We apply Proposition \ref{prop:geom-ed} with $(v,q,\zeta) = \partial_t^n(u,q,\eta)$ to get
\begin{equation}
\begin{aligned}\label{eqn:geometric-ed-new}
    &\ddt\left[\int_\Omega\frac{\abs{\pt_t^n u}^2J}2+\int_\Sigma\frac{\sigma\abs{\nabla{\pt_t^n\eta}}^2}2 + \frac{g\abs{\pt_t^n\eta}^2}2\right] + \int_\Omega\mu\frac{\abs{\mathbb D_{\mathcal A}{\pt_t^nu}}^2J}2 = \\
    &\hspace{5ex}  +\int_\Omega J\left(\pt_t^nu\cdot F^{1,n}+\pt_t^np\cdot F^{2,n}\right) + \int_\Sigma \left(-\sigma\Delta\pt_t^n\eta + g\pt_t^n\eta \right)F^{3,n} - \int_\Sigma F^{4,n}\cdot {(\pt_t^n u)} + F^{5,n} (\partial_t^n u) \cdot \mathcal{N}.
\end{aligned}
\end{equation}
Now we estimate the terms on the right hand side of \eqref{eqn:geometric-ed-new}. We easily bound the last term by 
\begin{equation}
\begin{aligned}
    \abs{\int_\Sigma F^{5,n} (\partial_t^n u) \cdot \mathcal{N}} = \abs{\int_\Sigma\left(\sum_{\ell = 0}^n C_{\ell,n}A\omega^{2+\ell}f^{(2+\ell)}(\omega t) \partial_t^{n-\ell}\eta\right) (\partial_t^n u) \cdot \mathcal{N}}\lesssim\left(\sum_{\ell = 2}^{n+2}A\omega^\ell\right)\calD_n^0.
\end{aligned}
\end{equation}
To handle the pressure term we first rewrite
\begin{equation}
    \int_\Omega\partial_t^npJF^{2,n} = \frac{d}{dt}\int_\Omega \partial_t^{n-1}pJF^{2,n} - \int_\Omega\partial_t^{n-1}p\partial_t(JF^{2,n}).
\end{equation}

We then use Theorem \ref{thm:bound-nonlin-F} to estimate
\begin{equation}
    \abs{\int_\Omega \pt_t^{n-1}p\pt_t(JF^{2,n})}\leq \norm{\pt_t^{n-1}p}_0\norm{\pt_t(JF^{2,n})}_0\lesssim P(\sigma)\sqrt{\calD_n^\sigma}\sqrt{\calE_n^0\calD_n^\sigma} = P(\sigma)\sqrt{\calE_n^0}\calD_n^\sigma.
\end{equation}
Using Theorem \ref{thm:bound-nonlin-F}, Proposition \ref{prop:L-infinity}, trace theory, we get that
\begin{equation}
    \abs{\int_\Omega J\pt_t^n u \cdot F^{1,n} - \int_\Sigma F^{4,n}\cdot \pt_t^n u} \lesssim \norm{\pt_t^n u}_1(\norm{F^{1,n}}_0 + \norm{F^{4,n}}_0) \lesssim P(\sigma)\sqrt{\calD_n^\sigma}\sqrt{\calE_n^0\calD_n^\sigma} = P(\sigma)\sqrt{\calE_n^0}\calD_n^\sigma.
\end{equation}
For the rest of the terms, we again use Theorem \ref{thm:bound-nonlin-F}  to estimate 
\begin{equation}
    \abs{\int_\Sigma(-\sigma\Delta\partial_t^n\eta + g\partial_t^n\eta )F^{3,n}}\lesssim \left( \sigma \norm{\dt^n \eta}_{2} +    \norm{\partial_t^n \eta}_0  \right) \norm{F^{3,n}}_0 
    \lesssim P(\sigma)\sqrt{\calD_n^\sigma}\sqrt{\calE_n^0\calD_n^\sigma} = P(\sigma)\sqrt{\calE_n^0}\calD_n^\sigma.
\end{equation}

Next we rewrite some of the terms on the left side of the equations.  Proposition \ref{prop:L-infinity} allows us to bound
\begin{equation}
    \frac12\int_\Omega\abs{\mathbb D\partial_t^n u}^2\leq \int_\Omega\frac12\abs{\mathbb D_\calA \pt_t^n u}^2 J + C\sqrt{\calE_n^0}\calD_n^\sigma
\end{equation}
and
\begin{equation}
    \int_\Omega\frac12\abs{\pt_t^n u}^2 J = \int_\Omega\frac12\abs{\pt_t^n u}^2 + \int_\Omega\frac12\abs{\pt_t^n u}^2(J-1).
\end{equation}
The theorem follows by combining the above estimates and rearranging.

\end{proof}

Our next result provides energy-dissipation estimates for all derivatives besides the highest order temporal ones.
\begin{thm}\label{thm:nonlin-ed-space}
Suppose that $\mathcal E_n^\sigma\leq\delta$ for $\delta\in(0,1)$ given in Proposition \ref{prop:L-infinity}. Let $\al\in\mathbb N^{1+2}$ be such that $\abs\al\leq 2n$ and $\al_0 < n$. Suppose further that $\sum_{\ell = 2}^{n+1}A\omega^\ell\lesssim 1$. Then there exists a polynomial $P$ with nonnegative universal coefficients such that
\begin{equation}
    \ddt\ov{\calE_\al^\sigma} + \ov{\calD_\al}\lesssim \left(\sum_{\ell = 2}^{n+1}A\omega^\ell\right)\calD_n^0 + P(\sigma)\left(\sqrt{\calE_n^0}\calD_n^\sigma  + \sqrt{\calD_n^\sigma\calK\calF_n}\right).
\end{equation}

Moreover, when $n=1$ and $\sigma > 0$ is a fixed constant, we can improve this to  
\begin{equation}
    \ddt\ov{\calE_\al^\sigma} + \ov{\calD_\al}\lesssim (A\omega^2 + A\omega^3)\calD_1^0 + \frac{P(\sigma)}{\sigma}\sqrt{\calE_1^\sigma}\calD_n^\sigma.
\end{equation}
\end{thm}
\begin{proof}

We begin by applying Proposition \ref{prop:flattened-ed} on $(v,q,\zeta) = \pt^\al(u,p,\eta)$ to see that
\begin{equation}\label{equation:flattened-ed-alpha}
\begin{aligned}
    &\ddt\ov\calE_\al + \ov\calD_\al = -\int_\Sigma \pt^\al\left(A\omega^2f''(\omega t)\eta\right)  \pt^\al u_3 + \int_\Omega \partial^\al u\cdot\partial^\al G^1 + \partial^\al p\partial^\al G^2 \\
    &\hspace{10ex} + \int_\Sigma\left(-\sigma\Delta\partial^\al\eta + g\pt^\al\eta   \right)\partial^\al G^3 - \partial^\al G^4\cdot\partial^\al u.
\end{aligned}
\end{equation}
{
We will now estimate all of the terms appearing on the right side of \eqref{equation:flattened-ed-alpha}. The first term is easily bounded using the duality between $H^{1/2}(\Sigma)$ and $H^{-1/2}(\Sigma)$ and trace theory:
\begin{multline}
    \abs{\int_\Sigma \pt^\al\left(A\omega^2f''(\omega t)\eta\right)  \pt^\al u_3 } \lesssim  \left(\sum_{\ell = 2}^{n+1}A\omega^\ell\right) \left(\sum_{j=0}^{n-1} \norm{\partial_t^j \eta}_{2n-2j-1/2}  \right) \left( \sum_{j=0}^{n-1} \norm{\partial^j u_3}_{H^{2n-2j+1/2}(\Sigma)} \right) \\
    \lesssim     
    \left(\sum_{\ell = 2}^{n+1}A\omega^\ell\right) \sqrt{\calD_n^0} \left( \sum_{j=0}^{n-1} \norm{\partial^j u}_{2n-2j+1} \right)   \lesssim     
    \left(\sum_{\ell = 2}^{n+1}A\omega^\ell\right)      \calD_n^0.
\end{multline}
}
In order to estimate the remaining terms on the right side of \eqref{equation:flattened-ed-alpha} we will break to cases based on $\alpha$.

\textbf{Case 1 -- Pure spatial derivatives of highest order:}
In this case we first consider $\alpha\in\mathbb N^{1+2}$ with $\abs\alpha = 2n$ and $\al_0 = 0$, i.e.\ $\pt^\al$ is purely spatial derivatives of the highest order. Now write $\alpha = \beta + \gamma$ for $\abs{\beta}=1$. We then use integration by parts and Theorem \ref{thm:bound-nonlin-G} to bound the $G^1$ term via
\begin{equation}
    \abs{\int_\Omega\partial^\al u\cdot\partial^\al G^1} = \abs{\int_\Omega\partial^{\al+\beta} u\cdot\partial^{\gamma} G^1}\lesssim \norm{u}_{2n+1}\norm{G^1}_{2n-1}\lesssim P(\sigma)\sqrt{\calD_n^0}\sqrt{\calE_n^0\calD_n^\sigma+\calK\calF_n}.
\end{equation}
To bound the $G^2$ term, compute
\begin{equation}
    \abs{\int_\Omega\partial^\al p\cdot\partial^\al G^2}\leq \norm{\partial^\al p}_0\norm{\partial^\al G^2}_0\lesssim P(\sigma)\sqrt{\calD_n^0}\sqrt{\calE_n^0\calD_n^\sigma + \calK\calF_n}.
\end{equation}
For the $G^3$ term, the $-\sigma\Delta\partial^\al\eta\partial^\al G^3$ and $g\pt^\al\eta $ terms are handled by Proposition \ref{prop:bound-G3}. Finally, to bound the $G^4$ term, we have
\begin{aligneq}
    \abs{\int_\Sigma\partial^\al G^4\cdot\partial^\al u} &= \norm{\pt^\al G^4}_{H^{-1/2}(\Sigma)}\norm{\pt^\al u}_{H^{1/2}(\Sigma)}\lesssim \norm{G^4}_{H^{2n-1/2}(\Sigma)}\norm{u}_{2n+1} \\
    &\lesssim P(\sigma)\sqrt{\calD_n^0}\sqrt{\calE_n^0\calD_n^\sigma + \calK\calF_n}.
\end{aligneq}
Combining the above estimates yields the desired bound for this case. For the $n=1$ and $\sigma > 0$ case, we can apply the special cases of Theorem \ref{thm:bound-nonlin-G} and Proposition \ref{prop:bound-G3} and the same computations as above to deduce the result for $G^1$, $G^2$ and $G^3$, noting that 
\begin{equation}
 P(\sigma) + \frac{1+\sqrt{\sigma}}{\sigma} \ls \frac{P(\sigma)}{\sigma}
\end{equation}
where $P$ denotes different universal polynomials on each side of the inequality.  For $G^4$ we can use the same method as for $G^1$ to get
\begin{equation}
    \abs{\int_\Sigma \pt^\al G^4 \cdot \pt^\al u} = \abs{\int_\Sigma \pt^\gamma G^4 \cdot \pt^{\al+\beta} u} \lesssim \norm{G^4}_1\norm{u}_3 \lesssim \sqrt{\calE^\sigma_1}\calD_1^\sigma.
\end{equation}

\textbf{Case 2 -- Everything else:}
We now consider the remaining cases, i.e.\ either $\abs\al\leq 2n-1$ or else $\abs\al = 2n$ and $1\leq\al_0 < n$. In this case, the $G^1$, $G^2$, $G^4$ terms may be handled with Theorem \ref{thm:bound-nonlin-G}. For the $G^3$ term, we directly compute
\begin{equation}
  \abs{\int_\Sigma \left(-\sigma\Delta\pt^\al\eta + g\pt^\al\eta \right)\pt^\al G^3} 
  \lesssim \norm{-\sigma\Delta\pt^\al\eta + g\pt^\al\eta }_0\norm{\pt^\al G^3} 
 \lesssim P(\sigma)\sqrt{\calD_n^\sigma}\sqrt{\calE_n^0\calD_n^\sigma + \calK\calF_n}.
\end{equation}

We may now combine the two cases to conclude the desired theorem. In the case of $n=1$ and $\sigma > 0$, we can apply the special cases of Theorem \ref{thm:bound-nonlin-G} in the above.
\end{proof}

By combining Theorems \ref{thm:nonlin-ed-time} and \ref{thm:nonlin-ed-space} we get the following synthesized result.
\begin{thm}\label{thm:nonlin-ed}
Suppose that $\calE_n^\sigma\leq\delta$ for $\delta\in(0,1)$ given by Proposition \ref{prop:L-infinity}. Suppose further that $\sum_{\ell=2}^{n+1}A\omega^\ell\lesssim 1$. Then we have the estimate
\begin{equation}
    \ddt \parens{\ov{\calE_n^\sigma} + \calH_n} + \ov{\calD_n} \lesssim {\left(\sum_{\ell=2}^{n+2}A\omega^\ell\right)}\calD_n^0 + P(\sigma)\left(\sqrt{\calE_n^0}\calD_n^\sigma + \sqrt{\calD_n^\sigma\calK\calF_n}\right),
\end{equation}
where $\calH_n$ is defined as in \eqref{eqn:def-H}. Moreover, when $n=1$ and $\sigma > 0$, we have
\begin{equation}
    \ddt \parens{\ov{\calE_1^\sigma} + \calH_1} + \ov{\calD_1} \lesssim (A\omega^2 + A\omega^3)\calD_1^0 + \frac{P(\sigma)}{\sigma} \sqrt{\calE_1^\sigma}\calD_1^\sigma.
\end{equation}
\end{thm}

\subsection{Comparison estimates}
Our goal now is to show that the full energy and dissipation, $\calE_n$ and $\calD_n$, can be controlled by their horizontal counterparts $\ov{\calE_n^\sigma}$ and $\ov\calD_n$ up to some error terms that can be made small. We begin with the result for the dissipation.

\begin{thm}\label{thm:bootstrap-D}
Suppose that $\calE_n^\sigma\leq\delta$ for $\delta\in(0,1)$ given by Proposition \ref{prop:L-infinity}. Let $\calY_n$ be as defined in \eqref{eqn:define-Y}. If $\sum_{\ell=2}^{n+1}A\omega^\ell\lesssim 1$, then
\begin{equation}\label{est:comparison-D}
    \calD_n^\sigma\lesssim \calY_n + \ov\calD_n.
\end{equation}
\end{thm}
\begin{proof}
We divide the proof into several steps.

\textbf{Step 1 -- Application of Korn's inequality:}
Korn's inequality tells us that
\begin{equation} \sum_{\substack{\alpha\in\mathbb N^{1+2} \\ \abs\alpha\leq 2n}}\norm{\partial^\al u}_1^2\lesssim \ov\calD_n.
\end{equation}
Since $\pt_1$ and $\pt_2$ account for all the spatial differential operators on $\Sigma$, we deduce from standard trace estimates that
\begin{equation}\label{estimate:korn-trace}
    \sum_{j=0}^n \norm{\pt_t^j u}_{H^{2n-2j+1/2}(\Sigma)}\lesssim\sum_{\substack{\alpha\in\mathbb N^{1+2} \\ \abs\alpha\leq 2n}}\norm{\pt^\al u}_{H^{1/2}(\Sigma)}^2\lesssim \ov\calD_n.
\end{equation}

\textbf{Step 2 -- Elliptic estimates for the Stokes problem:}
With \eqref{estimate:korn-trace} in hand, we can now use the elliptic theory associated to the Stokes problem to gain control of the velocity field and the pressure. For $j=0,1,\dots,n-1$ we have that $\pt_t^j(u,p,\eta)$ solve the PDE
\begin{equation}
\begin{cases}
    \diverge S\left(\pt_t^j u,\pt_t^j p\right) = \pt_t^j G^1 - \pt_t\left(\pt_t^j u\right) & \text{in $\Omega$} \\
    \diverge\left(\pt_t^j u\right) = \pt_t^j G^2 & \text{in $\Omega$} \\
    \pt_t^j u = \left.\pt_t^j u\right\vert_\Sigma & \text{on $\Sigma$} \\
    \pt_t^j u = 0 & \text{on $\Sigma_b$}
\end{cases}.
\end{equation}
We may then apply the Stokes problem elliptic regularity estimates in Theorem \ref{thm:stokes-dirichlet} to bound
\begin{equation}\label{eq:stokes-D}
    \norm{\pt_t^{n-1} u}_3^2 + \norm{\nabla \pt_t^{n-1}p}_1^2\lesssim \norm{\pt_t^n u}_1^2 + \norm{\pt_t^{n-1} u}_{H^{5/2}(\Sigma)}^2 + \norm{\pt_t^{n-1} G^1}_1^2 + \norm{\pt_t^{n-1} G^2}_2^2\lesssim \calY_n + \ov\calD_n.
\end{equation}
The control of $\pt_t^{n-1} u$ provided by this bound then allows us to control $\pt_t^{n-2} u$ in a similar manner. We thus proceed iteratively with Theorem \ref{thm:stokes-dirichlet} with $m=2n-2j-1$, counting down from $n-1$ temporal derivatives to $0$ temporal derivatives in order to deduce that
\begin{equation}
    \sum_{j=0}^{n-1}\norm{\pt_t^j u}_{2n-2j+1}^2+\norm{\nabla \pt_t^j p}_{2n-2j-1}^2\lesssim P(\sigma)\left(\calY_n + \ov\calD_n\right).
\end{equation}

\textbf{Step 3 -- Free surface function estimates:}
Next we derive estimates for the free surface function. Consider the dynamic boundary condition on $\Sigma$ to write
\begin{equation}\label{eqn:dynamic-boundary}
    \left[(pI - \mu\mathbb Du)e_3\right]\cdot e_3 = \left[(-\sigma\Delta\eta + (g+A\omega^2f''(\omega t))\eta)e_3 + G^4\right]\cdot e_3.
\end{equation}
Now for $i=1,2$ and $j=0,1,\dots,n-1$, apply $\partial_i\pt_t^j$ to the above and rearrange to obtain
\begin{equation}
\begin{aligned}
    &-\sigma\Delta\pt_i\pt_t^j\eta + (g+A\omega^2 f''(\omega t))\pt_i\pt_t^j\eta = -\sum_{0<\ell\leq j}\pt_t^\ell\left(A\omega^2f''(\omega t)\right)\pt_i\pt^{j-\ell}\eta \\ 
    &\hspace{10ex}+ \left(\pt_i\pt_t^j p - 2\mu\pt_3\pt_i\pt_t^j u_3\right) - \pt_i\pt_t^j G^4\cdot e_3.
\end{aligned}
\end{equation}
We then use this in the capillary operator estimate count up from $j=0,1,\dots,n-1$ in Theorem \ref{thm:capillary-est} and employ \eqref{eq:stokes-D} to see that
\begin{equation}
\begin{aligned}
    &\norm{\partial_i\pt_t^j\eta}_{2n-2j-3/2}^2+\sigma^2\norm{\partial_i\pt_t^j\eta}_{2n-2j+1/2}^2 \lesssim \norm{-\sum_{0<\ell\leq j}\pt_t^\ell\left(A\omega^2f''(\omega t)\right)\pt_i\pt^{j-\ell}\eta}_{2n-2j-3/2}^2 \\
    &\hspace{15ex}+ \norm{\left(\pt_i\pt_t^j p - 2\mu\pt_3\pt_i\pt_t^j u_3\right) - \pt_i\pt_t^j G^4\cdot e_3}_{H^{2n-2j-3/2}(\Sigma)}^2 \\
    &\hspace{5ex}\lesssim \sum_{\ell=0}^{j-1}\norm{\pt_i\pt^\ell\eta}_{2n-2j-3/2}^2 + \norm{\nabla\pt_t^j p}_{2n-2j-1}^2 + \norm{\pt_t^j u}_{2n-2j+1}^2 + \norm{\pt_t^j G^4}_{H^{2n-2j-1/2}(\Sigma)}^2\lesssim \calY + \ov\calD.
\end{aligned}
\end{equation}
Recall that $\eta$ has zero integral over $\Sigma$ via \eqref{zero_avg}, so by using Poincar\'e's inequality, we also obtain
\begin{equation}\label{est:dt-eta-Y+D}
    \sum_{j=0}^{n-1}\norm{\pt_t^j\eta}_{2n-2j-1/2}^2 + \sigma^2\norm{\pt_t^j\eta}_{2n-2j+3/2}^2\lesssim \sum_{j=0}^{n-1}\sum_{i=1}^2\norm{\pt_i\pt_t^j\eta}_{2n-3/2}^2+\sigma^2\norm{\pt_i\pt_t^j\eta}_{2n+1/2}^2\lesssim \calY + \ov\calD.
\end{equation}

Next we estimate $\pt_t^j\eta$ for $j=1,2,\dots,n+1$ by employing the kinematic boundary condition
\begin{equation}\label{eqn:kinematic-boundary}
    \pt_t^{j+1}\eta = \pt_t^j u_3 + \pt_t^j G^3.
\end{equation}
We first use this and \eqref{est:dt-eta-Y+D} to bound
\begin{equation}
\begin{aligned}
    \norm{\pt_t\eta}^2_{2n-1}\lesssim \norm{u_3}_{H^{2n-1}(\Sigma)}^2 + \norm{G^3}_{H^{2n-1}(\Sigma)}^2\lesssim \norm{u}_{2n-1/2}^2+\calY_n\lesssim \calY_n+\ov\calD_n
\end{aligned}
\end{equation}
and then multiply by $\sigma^2$ in order to derive the similar estimate
\begin{equation}
\begin{aligned}
    \sigma^2\norm{\pt_t\eta}^2_{2n+1/2} &\lesssim \sigma^2\norm{u_3}_{H^{2n+1/2}(\Sigma)}^2 + \sigma^2\norm{G^3}_{H^{2n+1/2}(\Sigma)}^2 \\
    &\lesssim \norm{u}_{2n+1}^2 + \sigma^2\norm{G^3}_{H^{2n+1/2}(\Sigma)}^2\lesssim \calY_n + \ov\calD_n.
\end{aligned}
\end{equation}

Next we use a similar argument to control $\pt_t^2\eta$:
\begin{equation}
    \norm{\pt_t^2\eta}_{2n-2}^2\lesssim \norm{\pt_t u_3}_{H^{2n-2}(\Sigma)}^2 + \norm{\pt_t G^3}_{H^{2n-2}(\Sigma)}^2\lesssim \norm{\pt_t u}_{2n-3/2}^2 + \norm{\pt_t G^3}_{H^{2n-2}(\Sigma)}^2\lesssim \calY_n + \ov\calD_n
\end{equation}
and
\begin{equation}
\begin{aligned}
    \sigma^2\norm{\pt_t^2\eta}_{2n-3/2}^2 &\lesssim \sigma^2\norm{\pt_t u_3}_{H^{2n-3/2}(\Sigma)}^2 + \sigma^2\norm{\pt_t G^3}_{H^{2n-3/2}(\Sigma)}^2 \\
    &\lesssim \norm{\pt_t u_3}_{2n-1}^2 + \sigma^2\norm{\pt_t G^3}_{H^{2n-3/2}(\Sigma)}^2\lesssim \calY_n + \ov\calD_n.
\end{aligned}
\end{equation}

With control of $\pt_t^2\eta$ in hand we can iterate to obtain control of $\pt_t^j$ for $j=3,4,\dots,n+1$. This yields the estimate
\begin{equation}
\begin{aligned}
    \sum_{j=3}^{n+1}\norm{\pt_t^j\eta}_{2n-2j+5/2}^2 &\lesssim \sum_{j=3}^{n+1}\norm{\pt_t^{j-1} u_3}_{H^{2n-2j+5/2}(\Sigma)}^2 + \norm{\pt_t^{j-1} G^3}_{H^{2n-2j+5/2}(\Sigma)}^2 \\
    &= \sum_{j=2}^{n}\norm{\pt_t^j u}_{2n-2j+1}^2 + \norm{\pt_t^j G^3}_{H^{2n-2j+1/2}(\Sigma)}^2 \lesssim \calY_n + \ov\calD_n.
\end{aligned}
\end{equation}


Summing the above bounds then shows the following surface function estimate:
\begin{equation}\label{est:surface-function}
\begin{aligned}
    &\norm{\pt_t\eta}^2_{2n-1} +\sigma^2\norm{\pt_t\eta}^2_{2n+1/2} +  \norm{\pt_t^2\eta}_{2n-2}^2 +\sigma^2\norm{\pt_t^2\eta}_{2n-3/2}^2 \\
    &\hspace{10ex}+\sum_{j=0}^{n-1}\left(\norm{\pt_t^j\eta}_{2n-2j-1/2}^2+\sigma^2\norm{\pt_t^j\eta}_{2n-2j+3/2}^2\right)+\sum_{j=3}^{n+1}\norm{\pt_t^j\eta}_{2n-2j+5/2}^2\lesssim \calY_n + \ov\calD_n.
\end{aligned}
\end{equation}

\textbf{Step 4 -- Improved pressure estimates:}
We now return to \eqref{eqn:dynamic-boundary} with \eqref{est:surface-function} in hand in order to improve our estimates for the pressure. Applying $\pt_t^j$ for $j = 0,1,\dots,n-1$ shows that
\begin{equation}
    \pt_t^j p = -\sigma\Delta\pt_t^j\eta + g\pt_t^j\eta + \pt_t^j\left(A\omega^2f''(\omega t)\eta\right) + 2\pt_3\pt_t^j u_3 + \pt_t^j G^4\cdot e_3.
\end{equation}
We then use this with \eqref{est:dt-eta-Y+D} to bound
\begin{equation}
    \sum_{j=0}^{n-1}\norm{\pt_t^j p}_{H^0(\Sigma)}^2\lesssim \sum_{j=0}^{n-1}\norm{\pt_t^j\eta}_0^2 + \sigma^2\norm{\Delta\pt_t^j\eta}_2^2 + \norm{\pt_t^j u}_2^2 + \norm{\pt_t^j G^4}_{H^0(\Sigma)}^2\lesssim \calY_n + \ov\calD_n.
\end{equation}
Now by a Poincar\'e-type inequality, 
\begin{equation}
    \sum_{j=0}^{n-1}\norm{\pt_t^j p}_0^2\lesssim \sum_{j=0}^{n-1}\norm{\nabla\pt_t^j p}_0^2 + \norm{\pt_t^j p}_{H^0(\Sigma)}^2\lesssim \calY_n + \ov\calD_n.
\end{equation}
Hence
\begin{equation}
    \sum_{j=0}^{n-1}\norm{\pt_t^j p}_{2n-2j}^2\lesssim \sum_{j=0}^{n-1}\norm{\pt_t^j p}_0^2 + \norm{\nabla\pt_t^j p}_{2n-2j-1}^2\lesssim \calY_n + \ov\calD_n.
\end{equation}

\textbf{Step 5 -- Conclusion:}
The estimate \eqref{est:comparison-D} now follows by combining the above bounds.

\end{proof}

We now explore the counterpart for the energy.

\begin{thm}\label{thm:bootstrap-E}
Suppose that $\calE_n^\sigma\leq\delta$ for $\delta\in(0,1)$ given by Proposition \ref{prop:L-infinity}. Let $\calW_n$ be as defined in \eqref{eqn:define-W}. If $\sum_{\ell=2}^{n+1}A\omega^\ell\lesssim 1$, then there exists a polynomial $P$ with nonnegative universal coefficients such that
\begin{equation}\label{est:comparison-E}
    {\mathcal E_n^\sigma} \lesssim P(\sigma)\left(\mathcal W_n + \overline{\mathcal E_n^\sigma}\right).
\end{equation}
\end{thm}

\begin{proof}
We divide the proof into several steps.

\textbf{Step 1 -- Initial free surface terms:}
To begin, note that
\begin{equation}
    \sum_{\substack{\al\in\mathbb N^{1+2} \\ \abs\al\leq 2n}}\norm{\pt^\al\eta}_0^2 + \sigma\norm{\nabla\pt^\al\eta}_0^2 \lesssim \sum_{j=0}^{n}\norm{\pt_t^j\eta}_{2n-2j}^2 + \sigma\norm{\nabla\pt_t^j\eta}_{2n-2j}^2. 
\end{equation}
Since $\pt_t^j\eta$ has zero integral, we can then use Poincar\'e's inequality to conclude that
\begin{equation}\label{est:surface-function-init-E}
    \sum_{j=0}^n \norm{\pt_t^j\eta}_{2n-2j}^2 + \sigma\norm{\pt_t^j\eta}_{2n-2j+1}^2\lesssim \sum_{j=0}^n\norm{\pt_t^j\eta}_{2n-2j}^2 + \sigma\norm{\nabla\pt_t^j\eta}_{2n-2j}^2\lesssim \ov{\calE_n^\sigma}.
\end{equation}

\textbf{Step 2 -- Elliptic estimates:} Rewrite the flattened equations in \eqref{system:ns-flattened} as
\begin{equation}
\begin{cases}
    \nabla p - \mu\Delta u = G^1 - \partial_t u & \text{in $\Omega$} \\
    \diverge u = G^2 & \text{in $\Omega$} \\
    \partial_t\eta = u_3 + G^3 & \text{on $\Sigma$} \\
    (pI - \mu\mathbb Du)e_3 = (-\sigma\Delta\eta + g\eta + G^5)e_3 + G^4 & \text{on $\Sigma$} \\
    u = 0 & \text{on $\Sigma_b$}
\end{cases}.
\end{equation}
Note that in particular $(\partial_t^ju,\partial_t^jp,\partial_t^j\eta)$ for $j=1,2,\ldots,n-1$ satisfy the PDE
\begin{equation}
\begin{cases}
    \nabla \partial_t^jp - \mu\Delta \partial_t^ju = \partial_t^jG^1 - \partial_t^{j+1} u & \text{in $\Omega$} \\
    \diverge \partial_t^ju = \partial_t^jG^2 & \text{in $\Omega$} \\
    \partial_t^{j+1}\eta = \partial_t^ju_3 + \partial_t^jG^3 & \text{on $\Sigma$} \\
    (\partial_t^jpI - \mu\mathbb D\partial_t^ju)e_3 = (-\sigma\Delta\partial_t^j\eta + g\partial_t^j\eta + \partial_t^jG^5)e_3 + \partial_t^jG^4 & \text{on $\Sigma$} \\
    v = 0 & \text{on $\Sigma_b$}
\end{cases}.
\end{equation}
We may appeal to the elliptic estimates for the Stokes problem with stress boundary conditions \eqref{est:stokes-stress} to obtain
\begin{equation}
    \begin{aligned}
        \norm{\partial_t^{n-1}u}_2^2 + \norm {\partial_t^{n-1}p}_1^2&\lesssim \norm{\partial_t^{n-1}G^1 - \partial_t^n u}_0^2 + \norm{\partial_t^{n-1}G^2}_1^2 \\
        &\hspace{10ex}+ \norm{(-\sigma\Delta\partial_t^{n-1}\eta + g\partial_t^{n-1}\eta + \partial_t^{n-1}G^5)e_3 + \partial_t^{n-1}G^4}_{1/2}^2 \\
        &\lesssim \norm{\partial_t^{n-1}G^1}_0^2 + \norm{\partial_t^n u}_0^2 + \norm{\partial_t^{n-1}G^2}_1^2 \\
        &\hspace{10ex}+  \norm{\partial_t^{n-1}\eta}_{1/2}^2 +\sigma^2\norm{\partial_t^{n-1}\eta}_{5/2}^2 + \norm{\partial_t^{n-1}G^5}_{1/2}^2 + \norm{\partial_t^{n-1}G^4}_{H^{1/2}(\Sigma)}^2.
    \end{aligned}
\end{equation}
For the $G^5$ term we bound
\begin{equation}
\begin{aligned}
     \norm{\partial_t^{n-1}G^5}_{1/2}^2
    &\leq\sum_{0\leq\ell\leq n-1}\norm{A\omega^{\ell + 2}f^{(\ell + 2)}(\omega t)\partial_t^{(n-1)-\ell}\eta}_{1/2}^2\lesssim \sum_{0\leq\ell\leq n-1}\norm{\partial_t^\ell\eta}_{1/2}^2.
\end{aligned}
\end{equation}
As a result, we have
\begin{equation}
\begin{aligned}
    &\norm{\partial_t^{n-1}u}_2^2 + \norm {\partial_t^{n-1}p}_1^2 \lesssim \norm{\partial_t^{n-1}G^1}_0^2 + \norm{\partial_t^n u}_0^2 + \norm{\partial_t^{n-1}G^2}_1^2 \\
    &\hspace{10ex}+  \norm{\partial_t^{n-1}\eta}_{1/2}^2 +\sigma^2\norm{\partial_t^{n-1}\eta}_{5/2}^2 + \sum_{0\leq\ell\leq n-1}\norm{\partial_t^\ell\eta}_{1/2}^2 + \norm{\partial_t^{n-1}G^4}_{H^{1/2}(\Sigma)}^2 \lesssim P(\sigma)\left(\mathcal W_n + \overline{\calE_n^\sigma}\right).
\end{aligned}
\end{equation}
We in turn may induct downward to get bounds on $\partial_t^ju$ and $\partial_t^j p$ for $j = n-2,\ldots, 1, 0$.  Doing so, we arrive at the bounds
\begin{equation}\label{est:surface-function-E}
\begin{aligned}
    &\sum_{j=0}^{n-1}\norm{\partial_t^j u}_{2n-2j}^2 + \norm{\partial_t^jp}_{2n-2j-1}^2 \\
    &\hspace{10ex}\lesssim \ov{\calE_n^\sigma} + \sum_{j=0}^{n-1}\norm{\pt_t^j G^1}_{2n-2j-2}^2 + \norm{\pt_t^j G^2}_{2n-2j-1}^2+ \norm{\pt_t^j G^4}_{H^{2n-2j-3/2}(\Sigma)}^2 \lesssim P(\sigma)\left(\calW_n + \ov{\calE_n^\sigma}\right).
\end{aligned}
\end{equation}

\textbf{Step 3 -- Improved estimates for time derivatives of the free surface function:} With the estimates of \eqref{est:surface-function-E} in hand, we can improve the estimates for the time derivatives of the free surface function by employing the kinematic boundary condition
\begin{equation}
    \pt_t^{j+1}\eta = \pt_t^j u_3 + \pt_t^j G^3
\end{equation}
for $j = 0,1,\dots,n-1$. Using this, trace theory, \eqref{est:surface-function-init-E}, and \eqref{est:surface-function-E} provides us with the estimate
\begin{equation}
    \norm{\pt_t\eta}_{2n-1/2}^2\lesssim \norm{u}_{2n}^2 + \norm{G^3}_{H^{2n-1/2}(\Sigma)}^2\lesssim \calW_n + \ov{\calE_n^\sigma}.
\end{equation}
We then iterate this argument to control $\pt_t^j\eta$ for $j = 0, 1, \dots, n-1$. This yields the bound
\begin{equation}
    \sum_{j=1}^{n}\norm{\pt_t^j\eta}_{2n-2j+3/2}^2 \lesssim \sum_{j=0}^{n-1}\norm{\pt_t^j u}_{2n-2j}^2 + \norm{\pt_t^j G^3}_{H^{2n-2j-1/2}(\Sigma)}^2\lesssim \calW_n + \ov{\calE_n^\sigma}.
\end{equation}

\textbf{Step 4 -- Conclusion:} The estimate in \eqref{est:comparison-E} now follows by combining the above bounds.

\end{proof}

\section{Vanishing surface tension problem}\label{section:vanishing-surface-tension}
In this section we complete the development of the a priori estimates for the vanishing surface tension problem and for the problem with zero surface tension. With these estimates in hand we then prove Theorems \ref{thm:main-vanishing-surface-tension} and \ref{thm:main-vanishing-surface-tension-limit}, which establish the existence of global-in-time decaying solutions and study the limit as surface tension vanishes. 

\subsection{Preliminaries}
Here we record a simple preliminary estimate that will be quite useful in the subsequent analysis.
\begin{prop}\label{prop:max-bounds}
For $N\geq 3$ we have that
\begin{equation}
    \calK\lesssim\min\left\{\calE_{N+2}^0,\calD_{N+2}^0\right\},\qquad\calF_{N+2}\lesssim \calE_{2N}^0.
\end{equation}
\end{prop}
\begin{proof}
    By Sobolev embeddings and trace theory, $\calK\lesssim \norm{u}_{7/2}^2+\norm{\eta}_{5/2}^2\leq \norm{u}_4^2 + \norm{\eta}_4^2$ and hence $\calK\lesssim \calE_2^0\leq \calE_{N+2}^0$ and $\calK\lesssim \calD_2^0\leq \calD_{N+2}^0$. On the other hand, $\calF_{N+2} = \norm{\eta}_{2N+4+1/2}^2\leq \norm{\eta}_{2N+5}^2$ and $2N+5\leq 4N$ for $N\geq 3$, so $\calF_{N+2}\leq \calE_{2N}^0$. 
\end{proof}

\subsection{Transport estimate}
We now turn to the issue of establishing structured estimates of the highest derivatives of $\eta$ by appealing to the kinematic transport equation.   

\begin{thm}\label{thm:F-upperbound}
Assume that $\calE_n^\sigma\leq\dt$ for the universal $\delta \in(0,1)$ given by Proposition \ref{prop:L-infinity}. Then
\begin{equation}\label{est:transport-est-main}
    \sup_{0\leq r\leq t}\calF_{2N}(r)\lesssim \exp\left(C\int_0^t\sqrt{\calK(r)}~dr\right)\left[\calF_{2N}(0) + t\int_0^t(1+\calE_{2N}^0)\calD_{2N}^0~dr + \left(\int_0^t\sqrt{\calK\calF_{2N}}~dr\right)^2\right].
\end{equation}
\end{thm}
\begin{proof}
The argument used to prove Theorem 6.3 of \cite{tice2018asymptotic}, which is based on fractional regularity estimates for the transport equation proved by Danchin \cite{danchin2005estimates}, works here as well.  We refer to \cite{tice2018asymptotic} for details.
\end{proof}

Next we show that if we know a prior that $\calG_{2N}$ is small, then in fact it is possible to estimate $\calF_{2N}$ more strongly than is done in Theorem \ref{thm:F-upperbound}.

\begin{thm}\label{thm:F-sup-bound}
Let $\calG_{2N}^0$ be defined by \eqref{definition:calG} for $N\geq 3$.  There exists a universal $\delta\in(0,1)$ such that if $\calG_{2N}^0(T)\leq\delta$ and $\gamma\leq 1$, then
\begin{equation}\label{eqn:F-sup-bound}
    \sup_{0\leq r\leq t}\calF_{2N}(r)\lesssim\calF_{2N}(0) + t\int_0^t\calD_{2N}^0(r)\,dr
\end{equation}
for all $0\leq t\leq T$.
\end{thm}

\begin{proof}
    According to Proposition \ref{prop:max-bounds} and the assumed bounds, we may estimate
    \begin{equation}\label{eqn:calG-one}
        \int_0^t\sqrt{\calK(r)}\,dr\lesssim\int_0^t\sqrt{\calE_{N+2}^0(r)}\,dr\leq\sqrt\delta\int_0^\infty\frac{1}{(1+r)^{2N-4}}\,dr\lesssim\sqrt\delta.
    \end{equation}
    Since $\delta\in(0,1)$, we thus have that for any universal $C > 0$
    \begin{equation}\label{eqn:calG-two}
        \exp\left(C\int_0^t\sqrt{\calK(r)}\,dr\right)\lesssim 1.
    \end{equation}
    Similarly,
    \begin{equation}\label{eqn:calG-three}
        \left(\int_0^t\sqrt{\calK(r)\calF_{2N}(r)}\,dr\right)^2\lesssim\left(\sup_{0\leq r\leq t}\calF_{2N}(r)\right)\left(\int_0^t\sqrt{\calE_{N+2}^0(r)}\,dr\right)^2\lesssim\left(\sup_{0\leq r\leq t}\calF_{2N}(r)\right)\delta.
    \end{equation}
    Then \eqref{eqn:calG-one}, \eqref{eqn:calG-two}, \eqref{eqn:calG-three}, and Theorem \ref{thm:F-upperbound} imply that
    \begin{equation}
        \sup_{0\leq r\leq t}\calF_{2N}(r)\leq C\left(\calF_{2N}(0) + t\int_0^t\calD_{2N}^0(r)\,dr\right) + C\delta\left(\sup_{0\leq r\leq t}\calF_{2N}(r)\right)
    \end{equation}
    for some $C > 0$. Then if $\delta$ is small enough so that $C\delta\leq\tfrac12$, we may absorb the right-hand $\calF_{2N}$ term onto the left and deduce \eqref{eqn:F-sup-bound}.
\end{proof}

\subsection{A priori estimates for $\calG_{2N}^\sigma$}
Our goal now is to complete our a priori estimates for $\calG_{2N}^\sigma$. We start with the bounds of the high-tier terms and $\calF_{2N}$. 

\begin{thm}\label{thm:high-tier}
There exist $\delta_0,\gamma_0\in(0,1)$ such that if $0\leq\sigma\leq 1$, $\calG_{2N}^\sigma(T)\leq \delta_0$, and $\sum_{\ell = 2}^{2N+2}A\omega^\ell\leq \gamma_0$, then
\begin{aligneq}
\sup_{0\leq r\leq t}\calE_{2N}^\sigma(r) + \int_0^t\calD_{2N}^\sigma(r)~dr + \sup_{0\leq r\leq t}\frac{\calF_{2N}(r)}{1+r}\lesssim \calE_{2N}^\sigma(0) + \calF_{2N}(0)
\end{aligneq}
for all $0\leq t\leq T$. 
\end{thm}
\begin{proof}
We first assume that $\delta_0$ is small as in Proposition \ref{prop:L-infinity}, and small as in Proposition \ref{prop:bound-H} so that $\abs{\calH_{2N}}\lesssim (\calE_{2N}^0)^{3/2}$.
\par We invoke Theorems \ref{thm:bootstrap-D} and \ref{thm:bootstrap-E} in order to bound
\begin{equation}
    \calE_{2N}^\sigma\lesssim\calW_{2N} + \ov{\calE_{2N}^\sigma}\qquad\text{and}\qquad \calD_{2N}^\sigma\lesssim\calY_{2N} + \ov{\calD_{2N}}.
\end{equation}
According to Theorem \ref{thm:bound-nonlin-G} we may then bound
\begin{equation}
    \calW_{2N}\lesssim\calE_{2N}^0\calE_{2N}^\sigma + \calK\calF_{2N}\qquad\text{and}\qquad\calY_{2N}\lesssim\calE_{2N}^0\calD_{2N}^\sigma + \calK\calF_{2N}.
\end{equation}
Upon combining the above two equations with the given bound for $\calH_{2N}$, we find that
\begin{equation}
    \calE_{2N}^\sigma\lesssim(\ov{\calE_{2N}^\sigma} + \calH_{2N}) + \calE_{2N}^0\calE_{2N}^\sigma + (\calE_{2N}^0)^{3/2} + \calK\calF_{2N}\quad\text{and}\quad\calD_{2N}^\sigma\lesssim\ov{\calD_{2N}} + \calE_{2N}^0\calD_{2N}^\sigma + \calK\calF_{2N},
\end{equation}
and consequently, if $\delta_0$ is assumed to be small enough we may absorb the $\calE_{2N}^0\calE_{2N}^\sigma + (\calE_{2N}^0)^{3/2}$ and $\calE_{2N}^0\calD_{2N}^\sigma$ terms onto the left to arrive at the bounds
\begin{equation}\label{est:calE-calD}
    \calE_{2N}^\sigma\lesssim(\ov{\calE_{2N}^\sigma} + \calH_{2N}) + \calK\calF_{2N}\qquad\text{and}\qquad \calD_{2N}^\sigma\lesssim\ov{\calD_{2N}} + \calK\calF_{2N}.
\end{equation}
We apply Theorem  \ref{thm:nonlin-ed} with $n=2N$ and integrate in time from $0$ to $t$ to see that
\begin{equation}
\begin{aligned}
    (\ov{\calE_{2N}^\sigma}(t) + \calH_{2N}(t)) + \int_0^t\ov{\calD_{2N}}(r)\,dr&\lesssim (\ov{\calE_{2N}^\sigma}(0) + \calH_{2N}(0)) + \left(\sum_{\ell = 2}^{2N+2}A\omega^\ell\right)\int_0^t\calD_{2N}^0(r)\,dr \\&\hspace{5ex}+ \int_0^t\sqrt{\calE_{2N}^0(r)}\calD_{2N}^\sigma(r)\,dr + \int_0^t\sqrt{\calD_{2N}^\sigma(r)\calK(r)\calF_{2N}(r)}\,dr.
\end{aligned}
\end{equation}
We then combine this with the estimate in \eqref{est:calE-calD} to arrive at the refined bound
\begin{equation}\label{est:refined-bound}
\begin{aligned}
    \calE_{2N}^\sigma(t) + \int_0^t\calD_{2N}^\sigma(r)\,dr&\lesssim\calE_{2N}^\sigma(0) + \left(\sum_{\ell = 2}^{2N+2}A\omega^\ell\right)\int_0^t\calD_{2N}^0(r)\,dr + \int_0^t\sqrt{\calE_{2N}^0(r)}\calD_{2N}^\sigma(r)\,dr \\&\hspace{20ex} +\int_0^t\left(\calK(r)\calF_{2N}(r) + \sqrt{\calD_{2N}^\sigma(r)\calK(r)\calF_{2N}(r)}\right)\,dr.
\end{aligned}
\end{equation}
We now turn our attention to the $\calK\calF_{2N}$ terms appearing on the right side of \eqref{est:refined-bound}.  To handle these we first note that $\calK\lesssim\calE_{N+2}^0$, as is shown in Proposition \ref{prop:max-bounds}.  Thus
\begin{equation}\label{est:K-bound}
    \calK(r)\lesssim\calE_{N+2}^0(r) = \frac{1}{(1+r)^{4N - 8}}(1+r)^{4N - 8}\calE_{N+2}^0\lesssim\frac{1}{(1+r)^{4N - 8}}\calG_{2N}(T)\lesssim\frac{\delta_0}{(1+r)^{4N - 8}}.
\end{equation}

Next we use Theorem \ref{thm:F-sup-bound} to see that for $0\leq r\leq t$ we can estimate
\begin{equation}\label{est:F-bound}
    \calF_{2N}(r)\lesssim\calF_{2N}(0) + (1+r)\int_0^r\calD_{2N}^0(s)\,ds.
\end{equation}
We may then combine \eqref{est:K-bound} and \eqref{est:F-bound} to estimate
\begin{equation}\label{eqn:KF-integral-bound}
\begin{aligned}
    \int_0^t\calK(r)\calF_{2N}(r)\,dr&\lesssim\delta_0\int_0^t\left(\frac{\calF_{2N}(0)}{(1+r)^{4N-8}} + \frac1{(1+r)^{4N-7}}\int_0^r\calD_{2N}^0(s)\,ds\right)\\
    &\lesssim\delta_0\calF_{2N}(0)\int_0^\infty\frac{dr}{(1+r)^{4N-8}} + \delta_0\left(\int_0^t\calD_{2N}^0(r)\,dr\right)\left(\int_0^\infty\frac{dr}{(1+r)^{4N-7}}\right)\\
    &\lesssim\delta_0\calF_{2N}(0) + \delta_0\int_0^t\calD_{2N}^0(r)\,dr,
\end{aligned}
\end{equation}
where here we have used $N\geq 3$ to guarantee that $(1+r)^{4N-8}$ and $(1+r)^{4N-7}$ are integrable on $(0,\infty)$.  Similarly, we may estimate
\begin{equation}\label{eqn:DKF-integral-bound}
\begin{aligned}
    &\int_0^t\sqrt{\calD_{2N}^\sigma(r)\calK(r)\calF_{2N}(r)}\,dr\leq\left(\int_0^t\calD_{2N}^\sigma(r)\,dr\right)^{1/2}\left(\int_0^t\calK(r)\calF_{2N}(r)\,dr\right)^{1/2}\\
    &\hspace{5ex}\lesssim\left(\int_0^t\calD_{2N}^\sigma(r)\,dr\right)^{1/2}\left(\delta_0\calF_{2N}(0) + \delta_0\int_0^t\calD_{2N}^0(r)\,dr\right)^{1/2} \\
    &\hspace{5ex}\lesssim\left(\calF_{2N}(0)+\int_0^t\calD_{2N}^\sigma(r)\,dr\right)^{1/2}\left(\delta_0\calF_{2N}(0) + \delta_0\int_0^t\calD_{2N}^\sigma(r)\,dr\right)^{1/2} \\
    &\hspace{5ex}\lesssim\sqrt{\delta_0}\calF_{2N}(0) + \sqrt{\delta_0}\int_0^t\calD_{2N}^\sigma(r)\,dr.
\end{aligned}
\end{equation}

Now we plug \eqref{eqn:KF-integral-bound} and \eqref{eqn:DKF-integral-bound} into \eqref{est:refined-bound}, bound $\calE_{2N}^0\leq\calG_{2N}^\sigma\leq\delta_0$, and use the fact that $\sqrt{\delta_0}\leq\delta_0$ due to $\delta_0 < 1$ to arrive at the bound
\begin{equation}
    \calE_{2N}^\sigma(t) + \int_0^t\calD_{2N}^\sigma(r)\,dr\lesssim \calE_{2N}(0) + \calF_{2N}(0) + \int_0^t\left(\sqrt\delta + \sum_{\ell = 2}^{2N+2}A\omega^\ell\right)\calD_{2N}^\sigma(r)\,dr.
\end{equation}
Thus if $\gamma_0,\delta_0\in(0,1)$ are chosen to be small enough, we may absorb the $\calD_{2N}^\sigma(r)$ integral term onto the left to deduce that
\begin{equation}\label{est:almost-thm}
    \calE_{2N}^\sigma(t) + \int_0^t\calD_{2N}^\sigma(r)\,dr\lesssim \calE_{2N}^\sigma(0) + \calF_{2N}(0).
\end{equation}
Upon combining \eqref{est:F-bound} and \eqref{est:almost-thm} we deduce that the desired inequality holds.
\end{proof}

Next we establish the algebraic decay results for the low-tier energy.
\begin{thm}\label{thm:low-tier}
There exists $\delta_0,\gamma_0\in(0,1)$ such that if $0\leq\sigma\leq 1$, $\sum_{\ell=2}^{N+4}A\omega^\ell\leq\gamma_0$, and $\calG_{2N}^\sigma(T)\leq\delta_0$, then
\begin{aligneq}\label{est:low-tier-decay-main}
\sup_{0\leq r\leq t}(1+r)^{4N-8}\calE_{N+2}^\sigma(r)\lesssim \calE_{2N}^\sigma(0) + \calF_{2N}(0)
\end{aligneq}
for all $0\leq t\leq T$. 
\end{thm}
\begin{proof}
We prove in four steps.

\textbf{Step 1 -- Set up:}
Assume $\delta_0$ is small as in Propositions \ref{prop:L-infinity} and \ref{prop:bound-H}. The latter allows us to estimate
\begin{aligneq}
\abs{\calH_{N+2}}\lesssim (\calE_{N+2}^0)^{3/2}\lesssim \sqrt{\calE_{2N}^0}\calE_{N+2}^0
\end{aligneq}

since $N\geq 3$. Then by applying Theorems \ref{thm:bootstrap-E} and \ref{thm:bootstrap-D} with $n=N+2$, together with Theorem \ref{thm:bound-nonlin-G} and Proposition \ref{prop:max-bounds} to get rid of the $G$ nonlinearities and the $\calK\calF_{N+2}$ terms, we obtain the bounds
\begin{equation}
\begin{aligned}
\calE_{N+2}^\sigma &\lesssim \left(\ov{\calE_{N+2}^\sigma} + \calH_{N+2}\right) + \sqrt{\calE_{2N}^0}\calE_{N+2}^0 + \calE_{N+2}^0\calE_{N+2}^\sigma + \calE_{N+2}^0\calE_{2N}^0 \\
\calD_{N+2}^\sigma &\lesssim \ov{\calD_{N+2}} + \calE_{N+2}^0\calD_{N+2}^\sigma + \calE_{2N}^0\calD_{N+2}^0
\end{aligned}.
\end{equation}
Thus if we assume that $\delta_0$ is small enough to absorb $\sqrt{\calE_{2N}^0}\calE_{N+2}^0 + \calE_{N+2}^0\calE_{N+2}^\sigma + \calE_{N+2}^0\calE_{2N}^0$ and $\calE_{N+2}^0\calD_{N+2}^\sigma + \calE_{2N}^0\calD_{N+2}^0$ onto the left hand side, then we may arrive at the bounds
\begin{aligneq}\label{est:ED-equivalence}
    \calE_{N+2}^\sigma\lesssim \left(\ov{\calE_{N+2}^\sigma}+\calH_{N+2}\right)\lesssim \calE_{N+2}^\sigma\qquad \calD_{N+2}^\sigma\lesssim \ov{\calD_{N+2}}\lesssim \calD_{N+2}^\sigma.
\end{aligneq}

\textbf{Step 2 -- Interpolation estimates:}
Now set 
\begin{equation}
    \theta\coloneqq \frac{4N-8}{4N-7}\in (0,1).
\end{equation}
We claim that we have the interpolation estimate
\begin{equation}\label{est:interpolation-est}
    \calE_{N+2}\lesssim \left(\calD_{N+2}^\sigma\right)^\theta\left(\calE_{2N}^\sigma\right)^{1-\theta}.
\end{equation}
For most of the terms appearing in $\calE_{N+2}^\sigma$, this is a simple matter. Indeed, the definitions of $\calE_{2N}^\sigma$ and $\calD_{N+2}^\sigma$ and the assumption that $\sigma\leq 1$ allow us to estimate
\begin{aligneq}
&\sum_{j=0}^{N+2}\norm{\pt_t^j u}_{2(N+2)-2j}^2 + \sum_{j=0}^{N+1}\norm{\pt_t^j p}_{2(N+2)-2j-1}^2 + \sigma\norm{\pt_t^j\eta}_{2n-2j+1}^2+ \sum_{j=2}^{N+2}\norm{\pt_t^j\eta}_{2(N+2)-2j+3/2} + \sigma\norm{\eta}_{2n+1}^2 \\
&\hspace{5ex}\lesssim \left(\calD_{N+2}^\sigma\right)^\theta\left(\calE_{2N}^\sigma\right)^{1-\theta}
\end{aligneq}
since the dissipation is actually coercive over the energy on these terms. To handle the remaining terms, we must use Sobolev interpolation. We begin with the most important term, which actually dictates the choice of $\theta$. We have that
\begin{equation}
    \left(2(N+2)-\frac12\right)\theta +  4N(1-\theta) = 2(N+2)\iff \left(2N-\frac72\right)\theta = 2N-4
\end{equation}
so this $\theta$ is compatible with Sobolev norm estimates and so we obtain
\begin{equation}
    \norm{\eta}_{2(N+2)}^2\leq \norm{\eta}_{2(N+2)-1/2}^{2\theta}\norm{\eta}_{4N}^{2(1-\theta)}\lesssim \left(\calD_{N+2}^\sigma\right)^\theta\left(\calE_{2N}^\sigma\right)^{1-\theta}.
\end{equation}
Finally, we bound
\begin{aligneq}
\norm{\pt_t\eta}_{2(N+2)-1/2}^2\lesssim \norm{\pt_t\eta}_{\theta(2(N+2)-1) + (1-\theta)(4N-1/2)}^2\lesssim \norm{\pt_t\eta}_{2(N+2)-1}^{2\theta}\norm{\pt_t\eta}_{4N-1/2}^{2(1-\theta)}\lesssim \left(\calD_{N+2}^\sigma\right)^\theta\left(\calE_{2N}^\sigma\right)^{1-\theta}
\end{aligneq}
and thus we have \eqref{est:interpolation-est} as claimed. 

\textbf{Step 3 -- Differential inequality:}
Next we apply Theorem \ref{thm:nonlin-ed} with $n=N+2$ in conjunction with Proposition \ref{prop:max-bounds} to see that
\begin{equation}
    \ddt\left(\ov{\calE_{N+2}^\sigma}+\calH_n\right) + \ov{\calD_{N+2}}\lesssim \left(\sum_{\ell=2}^{N+4}A\omega^\ell\right)\calD_{N+2}^0 + \sqrt{\calE_{N+2}^0}\calD_{N+2}^\sigma + \sqrt{\calE_{2N}^0}\calD_{N+2}^\sigma.
\end{equation}
We use this together with the bound $\calG_{2N}^\sigma(T)\leq\delta_0$ and the dissipation bounds of \eqref{est:ED-equivalence} to estimate
\begin{equation}
    \ddt\left(\ov{\calE_{N+2}^\sigma}+\calH_n\right) + \ov{\calD_{N+2}}\lesssim \left(\sqrt{\delta_0}+\sum_{\ell=2}^{N+4}A\omega^\ell\right)\ov{\calD_{N+2}}.
\end{equation}
Then by assuming that $\delta_0$ and $\sum_{\ell=2}^{N+4}A\omega^\ell$ are small enough, we may absorb the $\ov{\calD_{N+2}}$ onto the left of this inequality. Doing so and again invoking the dissipation bounds of \eqref{est:ED-equivalence} gives us that
\begin{equation}\label{est:diff-ineq-lem}
    \ddt\left(\ov{\calE_{N+2}^\sigma} +\calH_{N+2}\right)+ C_0\calD_{N+2}^\sigma\leq 0
\end{equation}
for a universal constant $C_0>0$. We then use the energy estimate in \eqref{est:ED-equivalence} to rewrite \eqref{est:interpolation-est} as
\begin{equation}
    \left(\ov{\calE_{N+2}^\sigma}+\calH_{N+2}\right)^{1/\theta}\lesssim \calD_{N+2}^\sigma\left(\calE_{2N}^\sigma\right)^{(1-\theta)/\theta}.
\end{equation}
We chain this together with the estimate in Theorem \ref{thm:high-tier} to write
\begin{equation}\label{est:interpolation+high-tier}
    \frac{C_1}{\calM_0^s}\left(\ov{\calE_{N+2}^\sigma}+\calH_{N+2}\right)^{1+s}\leq\calD_{N+2}^\sigma
\end{equation}
for $C_1>0$ a universal constant, $s\coloneqq (1-\theta)/\theta = 1/(4N-8)$, and $\calM_0\coloneqq \calE_{2N}(0) + \calF_{2N}(0)$. Upon combining \eqref{est:diff-ineq-lem} and \eqref{est:interpolation+high-tier}, we arrive at the differential inequality
\begin{equation}\label{est:diff-ineq}
    \ddt\left(\ov{\calE_{N+2}^\sigma} +\calH_{N+2}\right)+ \frac{C_0C_1}{\calM_0^s}\left(\ov{\calE_{N+2}^\sigma} +\calH_{N+2}\right)^{1+s}\leq 0.
\end{equation}
With \eqref{est:diff-ineq} in hand, we may integrate and argue as in the proofs of Theorem 7.7 of \cite{guo2013almost} or Proposition 8.4 of \cite{jang2016compressible} to deduce that
\begin{equation}\label{est:low-tier-decay}
    \sup_{0\leq r\leq t}(1+r)^{4N-8}\left(\ov{\calE_{N+2}^\sigma(r)} +\calH_{N+2}(r)\right)\lesssim \calM_0 = \calE_{2N}(0) + \calF_{2N}(0).
\end{equation}
Then \eqref{est:low-tier-decay} and the energy bound in \eqref{est:ED-equivalence} yield \eqref{est:low-tier-decay-main}.
\end{proof}

As the final step in our a priori estimates for $\calG_{2N}^\sigma$ we synthesize Theorems \ref{thm:high-tier} and \ref{thm:low-tier}.

\begin{thm}\label{thm:a-priori-vanishing}
There exist $\delta_0,\gamma_0\in(0,1)$ such that if $0\leq\sum_{\ell=2}^{2N+2}A\omega^\ell<\gamma_0$, $0\leq\sigma\leq 1$, and $\calG_{2N}^\sigma(T)\leq\delta_0$, then
\begin{equation}
    \calG_{2N}^\sigma(T)\lesssim \calE_{2N}^\sigma(0) + \calF_{2N}(0).
\end{equation}
\end{thm}
\begin{proof}
    We simply combine the estimates of Theorems \ref{thm:high-tier} and \ref{thm:low-tier}.
\end{proof}

\subsection{Main results for the vanishing surface tension problem}
Now that we have the a priori estimates of Theorem \ref{thm:a-priori-vanishing} in hand, we may prove Theorems  \ref{thm:main-vanishing-surface-tension} and \ref{thm:main-vanishing-surface-tension-limit} following previously developed arguments. For the sake of brevity we will omit full details and simply refer to the existing arguments.

\begin{proof}[Proof of Theorem \ref{thm:main-vanishing-surface-tension}]
The stated results follow by combining the local well-posedness theory, Theorem \ref{thm:local-existence}, with the a priori estimates of Theorem  \ref{thm:a-priori-vanishing} and a continuation argument. The details of the continuation argument may be fully developed by following the arguments elaborated in theorem 1.3 of \cite{guo2013almost} or theorem 2.3 of \cite{jang2016compressible}. 
\end{proof}

\begin{proof}[Proof of Theorem \ref{thm:main-vanishing-surface-tension-limit}]
The results follow from the estimates of Theorem \ref{thm:main-vanishing-surface-tension} and standard compactness arguments. See theorem 1.2 of \cite{tan2014zero} or theorem 2.9 of \cite{jang2016compressible} for details.
\end{proof}

\section{Fixed surface tension problem}\label{section:fixed-surface-tension}
In this section we study the problem \eqref{system:ns-flattened} in the case of a fixed $\sigma>0$. We develop a priori estimates and then present the proof of Theorem  \ref{thm:main-fixed-surface-tension}.  Although the structure of the proof is similar to that in \cite{tice2018asymptotic}, this paper uses $n=1$ to prove the main theorem rather than $n=2$ as done in \cite{tice2018asymptotic}.  This is because we wish to optimize our argument to give asymptotically better parameter regimes for $A$ and $\omega$; had we used $n=2$, then we would have to require $\sum_{\ell=2}^4 A\omega^\ell\lesssim 1$, which is worse than the regime in which $\sum_{\ell=2}^3 A\omega^\ell\lesssim 1$ when we wish to consider large $\omega$.  

Note that in what follows in this section we break our convention of not allowing universal constants to depend on $\sigma$.  All universal constants are allowed to depend on the fixed surface tension constant $\sigma$ but are still not allowed to depend on $A$ or $\omega$.

\subsection{A priori estimates for $\calS_\lambda$} In order to prove Theorem \ref{thm:main-fixed-surface-tension} we will introduce the following notation when $\lambda\in (0,\infty):$
\begin{equation}
\mathcal S_\lambda(T) \coloneqq \sup_{0\leq t\leq T}e^{\lambda t}\calE_1^\sigma(t) + \int_0^T e^{\lambda t} \calD_1^\sigma(t)\,dt.
\end{equation}
We now develop the main a priori estimates with surface tension.

\begin{thm}\label{thm:a-priori-est}
There exists $\delta_0,\gamma_0\in (0,1)$, depending on $\sigma>0$, such that if $\mathcal S_0(T)\le\delta_0$ and
\begin{equation}
    {\left(A\omega^2+A\omega^3\right)\leq \gamma_0},
\end{equation}
then there exists $\lambda = \lambda(\sigma) > 0$ such that
\begin{equation}
    \mathcal S_\lambda(T)\lesssim \mathcal E_1^\sigma(0).
\end{equation}
\end{thm}
\begin{proof}
We assume that $\delta_0$ is small enough that Propositions \ref{prop:L-infinity} and \ref{prop:bound-H} hold.

We use Theorems \ref{thm:bootstrap-D}, \ref{thm:bootstrap-E}, and \ref{thm:bound-nonlin-G}, as well as addition and subtraction of $\mathcal H$, to bound
\begin{equation}
    \calE_1^\sigma\lesssim \left(\ov{\calE_1^\sigma} + \calH_1\right) + (\calE_1^\sigma)^2 + (\calE_1^\sigma)^{3/2} \qquad\text{and}\qquad \calD_1^\sigma\lesssim \ov{\calD_1} + \calE_1^0\calD_1^\sigma.
\end{equation}
By further restricting $\delta_0$ we can use an absorbing argument to conclude
\begin{equation}\label{bound-E+H-1}
    \calE_1^\sigma + \calH_1\leq\calE_1^\sigma\lesssim\ov{\calE_1^\sigma} + \calH_1\quad\text{and}\quad \ov{\calD_1}\leq\calD_1^\sigma\lesssim\ov{\calD_1}.
\end{equation}

We now employ Theorem \ref{thm:nonlin-ed} with $n=1$ (recalling that we now allow universal constants to depend on $\sigma$) and \eqref{bound-E+H-1} to get under appropriate smallness assumptions that
\begin{equation}
\begin{aligned}
    \ddt \parens{\ov{\calE_1^\sigma} + \calH_1} + \ov{\calD_1} &\lesssim 
    \left(A\omega^2 + A\omega^3\right)\ov{\calD_1} + \left(\sqrt{\calE_1^\sigma}\right)\ov{\calD_1}.
    \end{aligned}
\end{equation}
We may then further restrict the size of $\delta_0$ and $\gamma_0$ in order to absorb terms on the right onto the left. Note that this absorption requires $\gamma_0,\delta_0$ to depend on $\sigma$. This yields the inequality
\begin{equation}\label{eqn:ED-differential}
    \ddt \left(\ov{\calE_1^\sigma} + \calH_1 \right)+ \frac12\ov{\calD_1} \leq 0.
\end{equation}
We defined $\calE_1^\sigma$ and $\calD_1^\sigma$ such that $\calE_1^\sigma\lesssim\sigma^{-1}\calD_1^\sigma$, so we can apply \eqref{bound-E+H-1} to get that there exists some $C > 0$ and $\lb > 0$ depending on $\sigma$ such that
\begin{aligneq}
    \frac12\ov{\calD_1} &\geq \frac{2}{4C)}\calD_1^\sigma\geq \frac{1}{4C} \calD_1^\sigma + \frac{\sigma}{4C} \calE_1^\sigma \\
    &\geq \frac1{4C}\calD_1^\sigma + \lambda\left(\ov{\calE_1^\sigma} + \calH_1\right)
\end{aligneq}
Plugging this into \eqref{eqn:ED-differential} gives
\begin{equation}
    \ddt\left(\ov{\calE_1^\sigma} +\calH_1\right) + \lambda \left(\ov{\calE_1^\sigma} + \calH_1\right) + \frac1{4C} \calD_1^\sigma \leq 0.
\end{equation}
We integrate this to get
\begin{equation}
    e^{\lambda t}\left(\ov{\calE_1^\sigma}(t) + \calH_1(t)\right) + \frac{1}{4C} \int_0^t e^{\lb r}\calD_1^\sigma(r)\,dr \leq \left(\ov{\calE_1^\sigma(0)}+\calH_1(0)\right).
\end{equation}

Now, appealing to \eqref{bound-E+H-1}, we deduce
\begin{equation}
    \sup_{0\leq t\leq T} e^{\lb t}\calE_1^\sigma(t) + \int_0^T e^{\lambda t} \calD_1^\sigma(t)dt \lesssim \calE_1^\sigma(0). 
\end{equation}
\end{proof}

\subsection{Proof of main result}
\begin{proof}[Proof of Theorem \ref{thm:main-fixed-surface-tension}]
We combine the local existence result in Theorem \ref{thm:local-existence-sigma} with the a priori estimates in Theorem \ref{thm:a-priori-est} and a continuation argument as in \cite{guo2013almost}.
\end{proof}

\appendix

\section{Elliptic estimates}\label{section:appendix-b}
Here we record basic elliptic estimates.

\subsection{Capillary operator}
Consider the problem
\begin{equation}\label{system:capillary}
    -\sigma\Delta\psi + g\psi = f\qquad \text{on $\mathbb T^n$} 
\end{equation}
for $g,\sigma >0$.  If $f\in H^{-1}(\mathbb T^n) = (H^1(\mathbb T^n))^*$, then a weak solution
If $f\in H^{-1}(\mathbb T^n) = (H^1(\mathbb T^n))^*$, then a weak solution is readily found with a standard application of Riesz's representation theorem: there exists a unique $\psi\in H^1(\mathbb T^n)$ such that
\begin{equation}
    \int_{\mathbb T^n}g\psi\varphi + \sigma\nabla\psi\cdot\nabla\varphi = \left\langle f,\varphi\right\rangle
\end{equation}

\begin{thm}\label{thm:capillary-est}
Let $s\geq 0$ and suppose that $f\in H^s(\mathbb T^n) \subseteq H^{-1}(\mathbb T^n)$.  Let $\psi\in H^1(\mathbb T^n)$ be the weak solution to \eqref{system:capillary}.  Then $\psi\in H^{s+1}(\mathbb T^n)$ and we have the estimates
\begin{equation}\label{equation:cap-est}
\norm{\psi}_s\leq\tfrac 1g\norm f_s\qquad\text{and}\qquad \norm{D^{2+s}\psi}_0\lesssim \tfrac1{\sigma}\norm{D^sf}_0,
\end{equation}
where $D = \sqrt{-\Delta}$.  Moreover, if $\int_{\mathbb T^n}\psi = 0$, then
\begin{equation}\label{equation:cap-est-ii}
    \norm\psi_{s+2}\lesssim\tfrac{1}{\sigma}\norm{D^s f}_0.
\end{equation}
\end{thm}

\begin{proof}
See, for instance, Theorem A.1 of \cite{tice2018asymptotic}. 
\end{proof}

\subsection{Stokes operator with Dirichlet conditions}
Consider the problem
\begin{equation}\label{system:stokes-dirichlet}
\begin{cases}
        -\Delta u+\nabla p = f^1 & \text{in $\Omega$} \\
        \diverge u = f^2 & \text{in $\Omega$} \\
        u = f^3 &\text{on $\Sigma$} \\
        u = 0 & \text{on $\Sigma_b$}
    \end{cases}.
\end{equation}
The estimates for solutions are recorded in the following result, the proof of which is standard and thus omitted.
\begin{thm}\label{thm:stokes-dirichlet}
Let $m\in\mathbb N$. If $f^1\in H^m(\Omega)$, $f^2\in H^{m+1}(\Omega)$, and $f^3\in H^{m+3/2}(\Sigma)$, the the solution pair $(u,p)$ to \eqref{system:stokes-dirichlet} satisfies $u\in H^{m+2}(\Omega)$, $\nabla p\in H^{m+1}(\Omega)$, and we have the estimate
\begin{equation}
    \norm{u}_{m+2} + \norm{\nabla p}_{m}\lesssim \norm{f^1}_m + \norm{f^2}_{m+2} + \norm{f}_{m+3/2}. 
\end{equation}
\end{thm}

\subsection{Stokes operator with stress conditions}
Consider the problem
\begin{equation}\label{system:stokes-stress}
\begin{cases}
        -\Delta u+\nabla p = f^1 & \text{in }\Omega \\
        \diverge u = f^2 & \text{in }\Omega \\
        u = 0 & \text{on }\Sigma_b \\
        (pI-\mathbb Du)e_3 = f^3 &\text{on } \Sigma.
    \end{cases}
\end{equation}

The estimates for solutions needed are recorded in the following result, the proof of which is standard and thus omitted.

\begin{thm}\label{est:stokes-stress}
Let $m\in\mathbb N$.  If $f^1\in H^m(\Omega)$, $f^2\in H^{m+1}(\Omega)$, and $f^3\in H^{m+1/2}(\Sigma)$, then the solution pair $(u,p)$ to \eqref{system:stokes-stress} satisfies $u\in H^{m+2}(\Omega)$, $p\in H^{m+1}(\Omega)$, and we have the estimate
\begin{equation}
    \norm{u}_{m+2} + \norm p_{m+1}\lesssim\norm{f^1}_m + \norm{f^2}_{m+1} + \norm{f^3}_{m+1/2}.
\end{equation}
\end{thm}

\section{Analytic tools}
\subsection{Product estimates}
In this section we record the necessary product estimates on Sobolev norms that we will need to get the correct bounds.

\begin{thm}\label{est:prod_est}
The following hold on $\Sigma$ and on $\Omega$.
\begin{enumerate}
    \item Let $0\leq r\leq s_1\leq s_2$ be such that $s_1 > n/2$.  Let $f\in H^{s_1}$, $g\in H^{s_2}$.  Then $fg\in H^r$ and \begin{equation}
    \norm{fg}_{H^r}\lesssim \norm{f}_{H^{s_1}}\norm{g}_{H^{s_2}}.
    \end{equation}
    \item Let $0\leq r\leq s_1\leq s_2$ be such that $s_2 > r + n/2$.  Let $f\in H^{s_1}$, $g\in H^{s_2}$.  Then $fg\in H^r$ and \begin{equation}
    \norm{fg}_{H^r}\lesssim \norm{f}_{H^{s_1}}\norm{g}_{H^{s_2}}.
    \end{equation}
    \item Let $0\leq r\leq s_1\leq s_2$ be such that $s_2 > r + n/2$.  Let $f\in H^{-r}(\Sigma)$, $g\in H^{s_2}(\Sigma)$.  Then $fg\in H^{-s_1}(\Sigma)$ and 
    \begin{equation}
        \norm{fg}_{-s_1}\lesssim\norm{f}_{-r}\norm{g}_{s_2}.
    \end{equation}
\end{enumerate}
\end{thm}

\begin{proof}
    See for example Lemma A.1 of \cite{guo2013almost}.
\end{proof}

\subsection{Poisson extension}\label{section:poisson-extension}
Suppose that $\Sigma=(L_1\mathbb T)\times(L_2\mathbb T)$. We define the Poisson integral in $\Omega_- = \Sigma\times(-\infty,0)$ by
\begin{equation}
    \calP f(x)\coloneqq\sum_{n\in(L_1^{-1}\mathbb Z)\times (L_2^{-1}\mathbb Z)} e^{2\pi in\cdot x'}e^{2\pi\abs nx_3}\hat f(n),
\end{equation}
where for $n\in (L_1^{-1}\mathbb Z)\times (L_2^{-1}\mathbb Z)$ we have written
\begin{equation}
    \hat f(n)\coloneqq \int_\Sigma f(x')\frac{e^{-2\pi in\cdot x'}}{L_1L_2}~dx'.
\end{equation}
It is well-known that $\calP:H^s(\Sigma)\to H^{s+1/2}(\Omega_-)$ is a bounded linear operator for $s>0$. We now show that derivatives of $\calP f$ can be estimated in the smaller domain $\Omega$. 

\begin{lem}
Let $\calP f$ be the Poisson integral of a function $f$ that is either in $\dot H^q(\Sigma)$ or $\dot H^{q-1/2}(\Sigma)$ for $q\in\mathbb N$. Then
\begin{equation}
    \norm{\nabla^q\calP f}_0^2\lesssim \norm{f}_{\dot{H}^{q-1/2}(\Sigma)}^2\qquad\text{and}\qquad\norm{\nabla^q\calP f}_0^2\lesssim \norm{f}_{\dot{H}^q(\Sigma)}^2.
\end{equation}
\end{lem}
\begin{proof}
See Lemma A.3 in \cite{guo2013almost}.
\end{proof}

We will also need $L^\infty$ estimates. 

\begin{lem}\label{lem:poisson-L-infinity}
Let $\calP f$ be the Poisson integral of a function $f$ that is in $\dot H^{q+s}(\Sigma)$ for $q\geq 1$ an integer and $s>1$. Then
\begin{equation}
    \norm{\nabla^q\calP f}_{L^\infty}^2\lesssim \norm{f}_{\dot H^{q+s}}^2.
\end{equation}
The same estimate holds for $q=0$ if $f$ satisfies $\int_\Sigma f = 0$.
\end{lem}
\begin{proof}
See Lemma A.4 in \cite{guo2013almost}.
\end{proof}

\bibliographystyle{abbrv}
\bibliography{citations}

\end{document}